\documentclass{amsart}

\usepackage[utf8]{inputenc}
\usepackage[T1]{fontenc}
\usepackage[english]{babel}
\usepackage{color}
\usepackage{amsmath} 
\usepackage{amssymb} 
\usepackage{amsthm}
\usepackage{graphicx}
\usepackage[colorlinks=true, linkcolor=blue]{hyperref}
\usepackage{cancel}

\usepackage{lmodern}
\usepackage{microtype}

\usepackage{enumerate}
\usepackage[shortlabels]{enumitem}

\theoremstyle{plain}
\newtheorem{theorem}{Theorem}[section]
\newtheorem*{theorem*}{Theorem}
\newtheorem{prop}[theorem]{Proposition}
\newtheorem{lemma}[theorem]{Lemma}

\newtheorem*{cor*}{Corollary}

\newtheorem{defi}[theorem]{Definition}

\newcommand {\R} {\mathbb{R}} \newcommand {\Z} {\mathbb{Z}}
\newcommand {\T} {\mathbb{T}} \newcommand {\N} {\mathbb{N}}

\newcommand {\p} {\partial}
\newcommand {\dt} {\partial_t}

\begin{document}
\title[Echoes around Waves]{On echo chains in the linearized Boussinesq equations around traveling
  waves}
\begin{abstract}
  We consider the 2D Boussinesq equations with viscous but without thermal dissipation and
  observe that in any neighborhood of Couette flow and hydrostatic balance (with
  respect to local norms) there are time-dependent traveling wave solutions of the form $\omega=-1+
  f(t)\cos(x-ty)$, $\theta=\alpha y + g(t)\sin(x-ty)$.  
  As our main result we show that the linearized equations around these waves
  for $\alpha=0$ exhibit echo
  chains and norm inflation despite viscous dissipation of the
  velocity.
  Furthermore, we construct initial data in a critical Gevrey 3 class, for which
  temperature and vorticity diverge to infinity in Sobolev regularity as
  $t\rightarrow \infty$ but for which the velocity still converges.
\end{abstract}
\author{Christian Zillinger}
\address{Karlsruhe Institute of Technology, Englerstraße 2,
  76131 Karlsruhe, Germany}
\email{christian.zillinger@kit.edu}
\keywords{Boussinesq equations, partial dissipation, resonances, blow-up}
\subjclass[2010]{35Q35,35Q79,76D05,35B40}
\maketitle
\setcounter{tocdepth}{1}
\tableofcontents

\section{Introduction and Main Results}
\label{sec:intro}

In this article we study the long-time asymptotic stability of the 2D
Boussinesq equations without thermal dissipation and with isotropic viscous dissipation:
\begin{align}
  \label{eq:bous}
  \begin{split}
    \dt v + v \cdot \nabla v + \nabla p &= \nu \Delta v +\theta e_2, \\
  \dt \theta + v \cdot \nabla \theta &=0, \\
  \nabla \cdot v &=0, \\
  (t,x,y) &\in (0,\infty) \times \T \times \R. 
  \end{split}
\end{align}
Here $v$ denotes the velocity of a fluid, $p$ denotes the pressure and $\theta$
denotes the temperature of the fluid.

The Boussinesq equations are a coupled system of the Navier-Stokes equations and
a diffusion equation. They describe the evolution of a heat conducting viscous
fluid, where the term $\theta e_2$ causes hot fluid to rise above cold fluid and
thus models buoyancy.
In particular, if a layer of hot fluid lies beneath a layer of cold fluid, the
system may exhibit a so-called Rayleigh-B\'enard instability \cite{doering2018long}, which can be
suppressed by sufficiently strong shear flow or dissipation
\cite{zillinger2020boussinesq}.

The study of the stability and asymptotic behavior of the Boussinesq
equations, in particular with anisotropic or partial dissipation, is an area of
very active research. We in particular mention the recent works
\cite{elgindi2015sharp,widmayer2018convergence,doering2018long,yang2018linear,wu2019stability,deng2020stability,wu2020stabilizing,dong2020stability2,masmoudi2020stability,tao2020stability,lai2021optimal,zillinger2020boussinesq,zillinger2020enhanced}
and the classical wellposedness results \cite{chae1999local,chae2006global}.

Following the seminal works of Villani and Mouhot \cite{Villani_long} on Landau damping in plasma physics and of Bedrossian and Masmoudi \cite{bedrossian2013asymptotic} on inviscid damping
for the Euler equations, questions of the effects of mixing have attracted strong interest.
In the Euler setting one observes that for small, smooth perturbations of an affine shear flow, $v=(y,0)$, the perturbation of the velocity field asymptotically converges.
This asymptotic stability of the velocity field is known as \emph{inviscid
  damping} and related to the Orr mechanism \cite{orr1907stability}.
Subsequently it was shown that this mechanism is very robust and linear inviscid
damping holds for rather general classes of monotone flows at rather low Sobolev regularity
\cite{Zhang2015inviscid,WZZvorticitydepl,Zill5,zillinger2019linear,Zill6,coti2019degenerate,BCZVvortex2017,ionescu2018inviscid,jia2019linear}.
However, for the \emph{nonlinear equations} very high, Gevrey $2$ regularity requirements are imposed to establish stability and inviscid damping \cite{ionescu2018inviscid,ionescu2020nonlinear}.
Already in \cite{bedrossian2015inviscid} it was sketched in terms of a toy
model that iterated nonlinear resonances might lead to norm inflation with an
exponential dependence on the frequency, which gave strong evidence of the
necessity of Gevrey regularity.

In \cite{dengmasmoudi2018} it was shown that there indeed exists data exhibiting
chains of resonances (called echo chains) and associated norm inflation behavior (see also \cite{bedrossian2016nonlinear}
for similar results for the Vlasov-Poisson equations).
This mechanism has been further studied in detail in \cite{dengZ2019} for the
Euler equations linearized around traveling waves. As the main results of \cite{dengZ2019} it is
shown that the norm inflation does \emph{not} necessarily imply that inviscid damping
fails.
On the contrary, there exists data in a critical Gevrey class such that the solution
not only exhibits norm inflation, but even blow-up: the \emph{vorticity} diverges
to infinity in Sobolev regularity as time tends to infinity. Yet, for the same
data the \emph{velocity} still is damped in $L^2$ to another shear flow as time tends
to infinity. Hence, damping of the \emph{velocity}, which is the physical effect of
inviscid damping, persists \emph{despite} blow-up of the \emph{vorticity}. Similar results also holds for the Vlasov-Poisson equations
\cite{zillinger2020landau}.

In this article we study whether such resonances are also be present in the
Boussinesq equations without thermal dissipation and whether the Gevrey regularity
requirements of \cite{masmoudi2020stability} are necessary.
In these equations there is competition of viscous dissipation, destabilization by buoyancy and resonance effects. In particular, it is a priori not clear whether these equations can sustain resonances and, if so, what the implications for norm inflation and asymptotic stability are.
As the main results of this article, we show that the Boussinesq equations
linearized around traveling waves indeed exhibit norm inflation and blow-up of the
temperature in Sobolev regularity as time tends to infinity. Yet, damping of the
velocity field (in lower Sobolev norms) persists despite this blow-up.\\

The key mechanism of this article's results is given by a (nonlinear) resonance mechanism, which exploits the system structure of the Boussinesq equations.
Similar resonance mechanisms also underlie instability results in the Euler equations \cite{dengZ2019}, where they are known as fluid echoes, and the Vlasov-Poisson equations
where they are called plasma echoes \cite{bedrossian2016nonlinear}.
In both the Euler setting \cite{yu2005fluid} and the plasma setting \cite{malmberg1968plasma} these echoes have been experimentally observed.
\\

We observe that a combination of a shear flow and hydrostatic balance
\begin{align}
  \label{eq:stationarysol}
  v=
  \begin{pmatrix}
    y \\0
  \end{pmatrix}, \theta = \alpha y,
\end{align}
is a stationary solution of the Boussinesq equations \eqref{eq:bous}
\begin{align*}
  \dt v + v \cdot \nabla v + \nabla p &= \nu \Delta v +\theta e_2, \\
  \dt \theta + v \cdot \nabla \theta &=0, \\
  \nabla \cdot v &=0.
\end{align*}
for any $\alpha\geq 0$, where in this article, for simplicity, we restrict to the case $\alpha=0$.
In this setting the echo mechanism then works as follows:
\begin{itemize}
\item At the initial time one introduces a perturbation of the temperature of
  the form $\epsilon e^{ikx+i\eta y}$.
  According to the linearized dynamics this solution will be mixed and will weakly converge to zero as time tends to infinity.
\item At a later time $\tau<\frac{\eta}{k}$ we introduce another perturbation of the
  temperature at a different frequency in $x$ of the form $\epsilon e^{ilx}$.
    According to the linearized equation also this perturbation will weakly converge to zero as time tends to infinity.
\item In the linearized equations around the \emph{stationary state} \eqref{eq:stationarysol} both perturbations
  do not interact. However, in the nonlinear evolution (and in the linearization
  around a traveling wave) one observes a large time- and frequency-localized correction. Both perturbations result in a nonlinear \emph{echo}:
  \begin{itemize}
  \item 
    By the buoyancy term the first perturbation at frequency $k$ generates a
    perturbation of the vorticity
    $\omega=\nabla \times v$. 
  \item This vorticity perturbation leads to a frequency-localized resonance in
    the velocity at the resonant time $\frac{\eta}{k}$.
  \item By the nonlinearity $v\cdot \nabla \theta$ this velocity resonance then
    interacts with the second perturbation of the temperature at mode $l$,
    exciting the temperature at frequency $k+l$ in $x$.
  \end{itemize}  
\end{itemize}
We stress the presence of viscous dissipation in the velocity equations. Thus, in contrast to the Euler equations or Vlasov-Poisson equations, where the
density directly generates the velocity, the nonlinear echo effect here relies on the
system structure of the Boussinesq equations. The resonance mechanism starts in $\theta$, then excites $v$ and in turn excites $\theta$.
\\

Building on this heuristic of a single echo interaction, in this article we show
that such perturbations of the temperature can be identified with \emph{traveling wave}
solutions.
Moreover, they can result in not just one echo but rather a chain of echoes.
That is, let there be a traveling wave solution with $l=-1$ and introduce another perturbation at
frequency $k$ in $x$ and $\eta$ in $y$ with $\eta> k>1$.
Then by the above sketch the interaction of the second perturbation with the underlying wave will result in an echo at frequency $k-1$ at around the resonant time $\frac{\eta}{k}$.
In turn, this echo correction at frequency $k-1$ will interact with the underlying wave to generate an echo at frequency $k-2$ at around the later time $\frac{\eta}{k-1}$.
Iterating this procedure, we thus generate an \emph{echo chain}:
\begin{align*}
  k \rightarrow k-1 \rightarrow k-2 \dots \rightarrow 1,
\end{align*}
where the size of the steps in frequency corresponds to the frequency  of the
underlying wave.
As we show in Theorem \ref{theorem:main} of Section \ref{sec:main} the associated norm inflation
along this chain can be of size $\exp(C\sqrt[3]{\eta})$, which corresponds to a
Gevrey $3$ regularity class. This agrees with recent nonlinear stability results
of \cite{masmoudi2020stability}.

In the following subsections we provide an outline of the main results of this article.

\subsection{Traveling Waves}
\label{sec:travelwaves}
As a main result of Section \ref{sec:wave} we show that the tuple
\begin{align}
  \begin{split}
  \label{eq:wave}
  v&=
  \begin{pmatrix}
    y \\0
  \end{pmatrix} + \frac{f(t)}{1+t^2}\sin(x-ty)\begin{pmatrix} t\\ 1\end{pmatrix},\\
  \omega&=\nabla v= -1 + f(t)\cos(x-ty),\\
  \theta&= \alpha y + g(t)\sin(x-ty),
  \end{split}
\end{align}
yields a solution of the nonlinear Boussinesq equations with $\nu\geq 0$ for any $f(t), g(t)$, which
solve an associated ODE (see Proposition \ref{prop:waves}).
In particular, choosing $f(0),g(0)$ small, we can view these \emph{traveling
  wave} solutions as initially arbitrarily small perturbations of the stationary
solutions \eqref{eq:stationarysol}.
Therefore, we suggest that in order to understand the nonlinear perturbation
problem around \eqref{eq:stationarysol}, one should first study the linearized
problem around the waves \eqref{eq:wave}.
We remark that such waves are also solutions of the inviscid problem. However, in that case generically $f(t)$ does not remain
bounded as $t\rightarrow\infty$ (see Lemma \ref{lemma:boundf}).

The linearized problem around such a wave in vorticity formulation then reads
\begin{align}
  \label{eq:linwavefull}
  \begin{split}
  \dt \omega + y \p_x \omega + \frac{f(t)}{1+t^2}\sin(x-ty) (\p_y-t\p_x) \omega + v \cdot \nabla f(t)\cos(x-ty) &= \nu \Delta \omega + \p_x \theta, \\
  \dt \theta + y \p_x \theta + \frac{f(t)}{1+t^2}\sin(x-ty) (\p_y-t\p_x) \theta + v \cdot \nabla g(t)\cos(x-ty)  + v_2 \alpha &=0.
  \end{split}
\end{align}
As a simplification, throughout this article we consider the following setup:
\begin{itemize}
\item We consider the case $\alpha=0$, which implies that $g(t)=g(0)$ is
  independent of time.
  Furthermore, we assume that $g=2c\nu$ for a small constant $0<c<0.001$, which
  by Lemma \ref{lemma:boundf} further implies that
  \[f(t)\leq \frac{4c}{1+t^2}.\]
\item We remove the shear term $\frac{f(t)}{1+t^2}\sin(x-ty) (\p_y-t\p_x)$ and
  we fix the $x$ average of $\theta$ and $\omega$ to be zero by a small forcing.
  As we discuss in Section \ref{sec:wave} we expect the shear term to not change the
  dynamics qualitatively since $\frac{f(t)}{1+t^2}$ is small, rapidly decaying
  and integrable in time.
  Similarly, a change of $x$ average would correspond to a change of the underlying
  shear flow $U(y,t)$ which we expect to be controlled in terms of a change of coordinates
  $z=U(y,t)$.
  However, the associated change of variables
  would introduce small further than nearest neighbor interactions and variable
  coefficients in the differential operators, thus making the analysis technically
  much more involved. We hence neglect these effects in the present article.
\end{itemize}
The equations studied in this article thus read (see Definition \ref{defi:waveequ})
\begin{align}
  \label{eq:linwave}
  \begin{split}
    \dt \omega + y \p_x \omega &= -(v \cdot \nabla f \cos(x-ty))_{\neq} + \nu \Delta \omega + \p_x \theta, \\
    \dt \theta + y \p_x \theta  + (v \cdot \nabla g \cos(x-ty))_{\neq}  &=0, \\
    (t,x,y) &\in (0,\infty)\times \T \times \R,
  \end{split}
\end{align}
where $()_{\neq}$ denotes the projection removing the $x$-average and $g=c\nu, f(t)\leq \frac{c}{1+t^2}$ (see Lemmas \ref{lemma:boundf} and \ref{lem:linearizationaroundwave}).
We then show that this system exhibits chains of resonances.

We remark that the structure of the equations \eqref{eq:linwave} is very similar to the one of the linearized Euler
equations around a traveling wave, $\omega=-1 + c \cos(x-ty)$ (see \cite{dengZ2019}), with the following main differences:
\begin{itemize}
\item Let $G:= \nu \p_x \omega + \p_x^2 \Delta^{-1} \theta$.
  Then the replacement of the Biot Savart law $v_1=\p_y \phi = \p_y \Delta^{-1}
  \omega$ in the Euler equations is given by
  \[\p_y \phi = \nu^{-1} \p_x \p_y \Delta^{-2} \theta + \nu^{-1} \p_y\p_x^{-1} \Delta^{-1}G.\]
  In particular, this mapping is of order $-2$ with respect to $\theta$
  instead of $-1$ and corresponds to a real-valued instead of an imaginary
  Fourier multiplier. For this reason we here choose the underlying wave to be
  given by a cosine instead of a sine. 
\item Here $\frac{g}{2\nu}$ serves as a parameter of the strength of the
  interaction. In this article we will focus on the setting where this parameter is
  small, i.e. $g = 2c \nu$ for a small constant.
  We point out that this coupling parameter is a large challenge if one were to
  consider the inviscid limit $\nu \downarrow 0$ with $g$ fixed (or slowly
  decaying in $\nu$) instead.
\item We observe that in terms of frequency $\eta$ with respect to $y$, in the
  Euler setting we have a decay of multipliers with a rate $\eta^{-1}$ as $\eta
  \rightarrow \infty$, since the Biot-Savart law is of order $-1$.
  In contrast in the present setting we have decay with a rate $\eta^{-3}$. This implies that
  the decoupling of neighboring modes becomes much stronger for large $\eta$.
\item As we discuss in Section \ref{sec:wave} the resonances rely on the
  coupling between the temperature $\theta$ and the vorticity $\omega$ by $G$.
  More precisely, the underlying wave in $\theta$ leads to a growth of the
  velocity at a critical time. While the perturbation of the vorticity is then
  subsequently damped, this velocity induces a growth of a different mode of the
  temperature perturbation, which then excites the velocity again at a later
  time.
  We stress that this system structure of the resonance mechanism strongly
  differs from the one in the Euler equations \cite{dengZ2019}. In particular,
  here the vorticity and velocity experience strong, mixing-enhanced
  dissipation. The resonance mechanism hence has to exploit the absence of
  thermal dissipation while making use of resonances in the velocity.
\end{itemize}
We in particular stress the $\eta^{-3}$ decay. In stark contrast to the Euler equations
considered in \cite{dengZ2019}, here the regime where $\eta$ is arbitrarily large actually
turns out to be \emph{better behaved} due to stronger decay of coefficients and as result
stronger separation of frequencies.

Indeed, in Section \ref{sec:G} we study the resonance mechanism for those
frequencies $\eta$ with
\begin{align}
  \label{eq:largeness}
  |\eta|\geq c^{-1},
\end{align}
where $g=c\nu$.
This restriction is justified in Section \ref{sec:small}, where we show that
otherwise the evolution is asymptotically stable and does not exhibit large norm
inflation.

\subsection{Main Results}
\label{sec:main}

As the main result of our article we show that the linearized equations
\eqref{eq:linwave} around a
wave indeed exhibit the sketched echo chain mechanism. Moreover, in addition to
norm inflation there is a critical Gevrey $3$ class of initial data for which the
\emph{temperature} and \emph{vorticity} diverge to infinity in Sobolev regularity as time
tends to infinity, but the \emph{velocity} still converges. Thus damping of the
\emph{velocity} may persist despite blow-up and viscous dissipation is \emph{not} sufficient
to suppress resonances in the form of echo chains.
\begin{theorem}[Stability, norm inflation and blow-up]
  \label{theorem:main}
  Consider the equation \eqref{eq:linwave} and suppose that $\frac{g}{2\nu}=:c$
  satisfies $c<0.001$. Further define $G=\nu \p_x \omega + \p_x^2 \Delta^{-1}\theta$
  (see Lemma \ref{lemma:goodunknown} for a formulation of \eqref{eq:linwave} in
  terms of $G$).
  \begin{itemize}
  \item There exists $C>0$ such that if the (Fourier transform of the) initial data satisfy
    \[\int \exp(C\sqrt[3]{|\eta|}) (1+k^2)^N (|\mathcal{F} \theta_0|^2+ \eta
      |\mathcal{F}G_0|^2) < \infty,\]
    then for all times $t>0$ it holds that
    \[\int \exp(\frac{C}{2} \sqrt[3]{|\eta-kt|}) (1+k^2)^N (|\mathcal{F} \theta(t)|^2+ \eta |\mathcal{F}G(t)|^2) < \infty.\]
   The \emph{evolution preserves Gevrey $3$ regularity} up to a loss of constant. 
\item For $c \eta > \nu^{-3/2}$ there exists initial data
  $\theta_0 \in L^2$ localized at frequency $\eta$ and $G_0=0$, such that for
  all $t> 2\eta$ the solution satisfies
  \[\|\theta(t)\|_{L^2}\geq \exp(\sqrt[3]{c\eta}).\]
  There exists frequency localized initial data which exhibits \emph{norm
    inflation}.
  However, after attaining this norm inflation the solution is stable for all
  future times.
\item Moreover, for every $s_0 \in \R$ there exists $0<C'<C$ and initial data with
      \[\int \exp(C'\sqrt[3]{\eta}) (1+k^2)^N (|\mathcal{F} \theta_0|^2+
        |\mathcal{F}\omega_0|^2) < \infty,\]
      such that $\theta(t)$ converges in $H^s$ for $s<s_0$ and diverges to
      infinity in $H^s$ for $s>s_0$, as $t\rightarrow \infty$.
      The Gevrey $3$ regularity class is hence a critical space for stability
      and damping may persist despite \emph{blow-up}. 
  \end{itemize}
\end{theorem}
Let us comment on these results:
\begin{itemize}
\item The resonance mechanism here relies on the coupling between temperature
  and vorticity by means of $G$ (as we discuss in Section \ref{sec:wave} a
  similar unknown has previously been introduced in \cite{masmoudi2020stability}). In particular, while the equations for $\omega$
  exhibit very strong, mixing-enhanced dissipation, the norm of the temperature
  does not asymptotically decay and hence norm inflation persists.
\item We remark that in the Euler equations the critical Gevrey class is given
  by Gevrey $2$. The Boussinesq equations thus rely on a different resonance
  mechanism where it is not the interplay of vorticity and velocity resulting in
  resonances, but of the temperature and the velocity.
\item In \cite{masmoudi2020stability} stability of the nonlinear Boussinesq
  equations without thermal dissipation was established for Gevrey $3$ regular
  data. There a toy model suggests a norm inflation for frequency
  localized data, thus giving evidence of the necessity of Gevrey regularity.
  This article shows that this norm inflation indeed happens for \eqref{eq:linwave} and that chains of
  resonances are a feature of the linear equations around traveling waves.
  Furthermore, there is not only norm inflation but blow-up, yet for a critical
  class of data damping of the velocity still persists.
  This hence raises the question whether then in the nonlinear problem there
  exists a critical class and whether there, as in the linear dynamics around
  waves, damping may persist despite instability.
\item In \cite{bedrossian21} it is shown that the inviscid, nonlinear
  Boussinesq equations with stably stratified temperature exhibit damping of the
  velocity for small Gevrey regular initial data, but algebraic instability of
  the vorticity and of the gradient of the temperature.
\end{itemize}

As we discuss in Section \ref{sec:wave} this growth is driven by a chain of
resonances which happens on a time interval $(c^{-1/3}\eta^{2/3}, 2 \eta)$.
Here we used that the coefficient functions in \eqref{eq:linwave} in coordinates
$(x-ty,y)$ do not depend on $y$ anymore. Thus, in these coordinates the
evolution equations decouple with respect to $\eta$ and we may hence treat
$\eta$ as a given parameter.
Therefore in order to establish the Gevrey regularity bounds of Theorem
\ref{theorem:main}
we may consider $\eta$ to be arbitrary but fixed and show that for
$\omega(t,x)$, $\theta(t,x)$ the norms with respect to $x$ grow at most by a
factor $\exp(\sqrt[3]{c\eta})$ as time tends to infinity.
\subsection{Outline of the Article}
\label{sec:outline}
The remainder of the article is structured as follows:
\begin{itemize}
\item In Section \ref{sec:wave} we construct traveling wave solutions of the
  Boussinesq equations. Due to its system structure here waves propagate both in
  temperature and vorticity and the magnitude of both waves is coupled.
  We further show that for the viscous problem with $\alpha=0$, the wave in the
  vorticity decays as time tends to infinity while the wave in the temperature
  keeps it shape.
  In contrast, for the inviscid problem one encounters algebraic instability of
  the vorticity wave in agreement with the Miles-Howard criterion \cite{howard1961note}.
\item In Section \ref{sec:stability} we establish stability of the evolution in
  Gevrey $3$ regularity.
  In particular, we discuss the growth in different time regimes in different
  subsections.
  For instance in Section \ref{sec:small} we show that no large norm inflation
  may happen until a time of size about $c^{-1/3}\eta^{2/3}$ or after time $2\eta$.
  Hence, all resonances have to happen inside this time interval, where we
  distinguish between intermediate time intervals where $t\approx
  \frac{\eta}{k}$ and $\frac{c\eta}{k^3}$ is not yet large, treated in Section
  \ref{sec:intermediate}
  and the main resonance mechanism discussed in Section \ref{sec:G}.
  The proof of the stability result of Theorem \ref{theorem:main} is then given
  in Section \ref{sec:proofofmain}.
\item In Section \ref{sec:blow-up} we construct data exhibiting norm
  inflation and blow-up.
  Here we study the evolution of frequency-localized initial data throughout the
  various time regimes.
  In Subsection \ref{sec:contract} we use a contraction mapping approach in
  strongly weighted spaces to show that the solution remains localized until
  time about $c^{-1/3}\eta^{2/3}$.
  We then control the solution for intermediate times in Subsection
  \ref{sec:upper} by means of a bootstrap approach.
  Subsection \ref{sec:norminf} then forms the core of our norm inflation argument
  where we show that this data achieves the norm inflation estimated in Section
  \ref{sec:G}.
  Subsequently we show that this norm inflation of the temperature persists for all future times.

  Given these global in time solutions exhibiting norm inflation, we further
  construct data in critical Gevrey $3$ class which exhibit blow-up in Sobolev regularity.
\end{itemize}

\subsection{Notation}
\label{sec:notation}

In this article we consider the linearized Boussinesq equations around traveling
waves
\begin{align*}
   \omega&= -1 + f(t)\cos(k(x-ty)), \\
    \theta&= g \sin(k(x-ty)),
\end{align*}
where 
\begin{align*}
  g&= 2c\nu, \\
  f(t)&\leq \frac{4c}{1+t^2},
\end{align*}
and $0<c<0.001$ is a small constant.

Since these waves are stationary in coordinates $(x-ty,y)$ throughout this
article we work in these coordinates and note that (after some simplification)
the linear system \eqref{eq:simplewave} around these waves reads
  \begin{align*}
    \begin{split}
    \dt \omega &= (\frac{f(t)}{\nu}\cos(x)\p_y (\p_x^2+(\p_y-t\p_x)^2) ^{-1} \omega)_{\neq} + \nu (\p_x^2+(\p_y-t\p_x)^2) \omega + \p_x \theta, \\
    \dt \theta &= (\frac{g}{\nu} \cos(x) \p_y (\p_x^2+(\p_y-t\p_x)^2) ^{-1} \omega)_{\neq},\\
    (t,x,y) &\in (0,\infty)\times \T \times \R,
    \end{split}
  \end{align*}
where $()_{\neq}$ denotes the projection removing the $x$-average.  
  
Here it turns out to be advantageous to consider the \emph{good unknown} (see
Lemma \ref{lemma:goodunknown})
  \begin{align*}
    G= \nu \p_x \omega +\p_x^2 (\p_x^2+(\p_y-t\p_x)^2) ^{-1} \theta
  \end{align*}
  in place of the vorticity $\omega$ (see Section \ref{sec:wave} for a
  discussion), which leads to the system \eqref{eq:simplewave2}
\begin{align*}
      \dt \theta &= (\frac{g}{\nu} \cos(x) \p_y \Delta_t^{-1} (\p_x^{-1} G + \p_x \Delta_t^{-1} \theta))_{\neq} \\
    \dt G &= \nu \Delta_t G + (f \p_x \cos(x) \p_y \Delta_t^{-1} (\p_x^{-1} G + \p_x \Delta_t^{-1} \theta))_{\neq}  + 2 \p_x^3 (\p_y-t\p_x) \Delta_t^{-2} \theta + \p_x^2 \Delta_t^{-1} (\p_t \theta),
\end{align*}
where for brevity, we use the short notation
  \begin{align*}
    \Delta_t&:= \p_x^2 + (\p_y-t\p_x)^2.
  \end{align*}
We observe that none of the coefficient functions in this evolution equation
depend on $y$ explicitly.
Therefore the equations decouple with respect to the Fourier variable
\begin{align*}
  \eta \in \R
\end{align*}
with respect to $y$. We thus tend to consider $\eta$ as arbitrary
but fixed and study $\theta$ and $G$ as functions of $t$ and $x \in \T$ only.

Furthermore, we observe that also with respect to $x$ the only explicit
  coefficient functions are given by $\cos(x)$, which corresponds to a shift by
  $\pm 1$ in Fourier space.
  We may this equivalently consider this evolution equation as a nearest neighbor
  system for the Fourier modes
  \begin{align*}
    \theta_{k}, G_k,
  \end{align*}
  for $k \in \Z\setminus \{0\}$.
The corresponding ODE system and its integral formulation are given in Section \ref{sec:G}.

Some of our estimates consider the regime where $\eta$ is very large or use that
$c$ is very small.
Here, specifically our notation of very large or very small is given by a factor
$1000$ and we use
\begin{align*}
  a \ll b
\end{align*}
if 
\begin{align*}
  1000|a|< |b|.
\end{align*}
Similarly, we write
\begin{align*}
  a\gtrsim b
\end{align*}
if there exists a universal constant $C>\frac{1}{1000}$ such that
\begin{align*}
  |a| \geq C |b|.
\end{align*}

\section{Traveling Wave Solutions and a Good Unknown}
\label{sec:wave}

While the linearized Euler, Navier-Stokes or Boussinesq equations around an
affine velocity $v=(y,0)$ have been well-studied, the nonlinear problems have
proved much more challenging with very active research in recent years.
Similarly to the Vlasov-Poisson equations of plasma physics a main challenge in
this nonlinear analysis is given by nonlinear resonances called \emph{echoes},
which have also been experimentally observed \cite{yu2005fluid,malmberg1968plasma}.
We briefly recall the experiment for the Euler setting:
\begin{itemize}
\item 
At an initial time one introduces a first perturbation of the vorticity which looks like a mode
$e^{ikx}$. It is then mixed by the evolution and the perturbation of the
velocity field is damped.
\item At a later time one introduces a second perturbation of the vorticity which looks like a mode
$e^{ilx}$.
According to the linearized equations around the stationary state both
perturbations do not interact (the linearized equations decouple with respect to
frequency in $x$) and both perturbations are expected to be damped.
\item Yet, at a predictable later time one observes a peak of the velocity field at a
mode $e^{i(k+l)x}$. Both perturbations have interacted by means of the
nonlinearity to induce a correction at frequency $k+l$ in $x$ (and some
frequency in $y$), which later unmixes and yields a peak. The perturbations
result in an \emph{echo} at a later time.
\end{itemize}
This effect, or more accurately chains of echoes where one echo causes another
echo at a later time and so on, underlies the Gevrey regularity requirement of
the nonlinear analysis
\cite{bedrossian2016nonlinear,zillinger2020landau,dengZ2019,dengmasmoudi2018}.

Since the stationary states $\omega=-1, \theta=\alpha y$ of the Boussinesq equation are independent of $x$, the linearized problem
around them does not include this chain of resonances. Hence these linearized equations
exhibit radically different stability properties than the nonlinear problem.
We thus aim to find a different nearby solutions of the Boussinesq equations which
are $x$-dependent and whose linearized equations include this resonance
mechanism.
Such solutions are given by \emph{traveling waves}:
\begin{prop}[Traveling waves]
  \label{prop:waves}
  Let $\alpha \in \R, k \in\Z \setminus \{0\}$ and $\nu \geq 0$ be given.
  Then the tuple
\begin{align*}
  \begin{split}
  v&=
  \begin{pmatrix}
    y \\0
  \end{pmatrix} + \frac{f(t)}{k^2+k^2t^2}\sin(k(x-ty))\begin{pmatrix} kt\\ 1\end{pmatrix},\\
  \omega&= -1 + f(t)\cos(k(x-ty)),\\
  \theta&= \alpha y + g(t)\sin(k(x-ty)),
\end{split}
\end{align*}
solves the Boussinesq equations \eqref{eq:bous} if and only if $f$ and $g$ solve
the ODE system
\begin{align}
  \label{eq:ode}
  \begin{split}
  \dt f &= -\nu(k^2+k^2 t^{2}) f - k g, \\
  \dt g &= \frac{\alpha}{k^2 + k^2t^2} f.
\end{split}
\end{align}
We call such solutions \emph{traveling waves}.
\end{prop}

We stress that these traveling wave solutions exist both in the inviscid and
viscous problem. However, due to the dissipation, the asymptotic behavior as
$t\rightarrow \infty$ is very different in both cases, as we discuss in Lemma \ref{lemma:boundf}.

\begin{proof}[Proof of Proposition \ref{prop:waves}]
  We make the ansatz that
  \begin{align*}
    \omega= -1 + f(t)\cos(k(x-ty)), \\
    \theta= \alpha y + g(t) \sin(k(x-ty)),
  \end{align*}
  with $f(t)$ and $g(t)$ to be determined.

  Then since $\cos(k(x-ty))$ is an eigenfunction of the Laplacian, the stream function $\phi:=(-\Delta)^{-1}\omega$ is given by
  \begin{align*}
    \phi= -\frac{y^2}{2} + \frac{1}{k^2+ k^2 t^2} \cos(k(x-ty))
  \end{align*}
  and $v= \nabla^{\perp}\phi$ is of the claimed form.
  In particular, we note that
  \begin{align*}
    (\p_t + y\p_x) f(t) \cos(k(x-ty)) &= \dot{f}(t)\cos(k(x-ty)), \\
    (\p_t + y\p_x) g(t) \sin(k(x-ty)) &= \dot{g}(t)\sin(k(x-ty))
  \end{align*}
  and, since $\sin(x)=\pm \sqrt{1-\cos^2(x)}$ can locally be expressed as a
  function of $\sin(x)$ it also holds that
  \begin{align*}
    (\nabla^{\perp} \frac{1}{k^2+ k^2 t^2} \cos(k(x-ty))) \cdot \nabla f(t)\cos(k(x-ty)) &=0,\\
    (\nabla^{\perp} \frac{1}{k^2+ k^2 t^2} \cos(k(x-ty))) \cdot \nabla g(t)\sin(k(x-ty)) &=0.
  \end{align*}
  Therefore, plugging in this ansatz the Boussinesq equations read
  \begin{align*}
    \dot{f}(t) \cos(k(x-ty))&= -\nu (k^2-k^2t^2) f(t)\cos(k(x-ty)) - g(t)k \cos(k(x-ty)), \\
    \dot{g}(t) \sin(k(x-ty)) & = \frac{\alpha}{k^2+k^2 t^2} f(t) \sin(k(x-ty)).
  \end{align*}
  Since these equations are supposed to hold for all $x$, we may assume that
  $\sin(k(x-ty))\neq 0 \neq \cos(k(x-ty))$ and hence obtain the claimed system
  of ordinary differential equations.
\end{proof}

In this article we will focus on the case $k=1$ and thus traveling waves at the
lowest non-trivial frequency.
Moreover, we restrict to considering the case $\alpha=0$ for which $g(t)$ is
constant in time and thus allows for an explicit characterization of the evolution.

\begin{lemma}[Bounds on $f$ and $g$]
\label{lemma:boundf}
  Let $\alpha=0, \nu>0$, $k=1$ and let $f(t), g(t)$ be the solution of \eqref{eq:ode}
  with initial data $f_0, g_0$.
  Then for all $t>0$ it holds that
  \begin{align*}
    |f(t)- \exp(-\nu (t+t^3/3))f_0| &\leq \frac{4}{\nu (1+t^2)} |g_0|, \\
    g(t)&=g_0.
  \end{align*}

  Let instead $\nu=0$ and $\alpha \in \R$ and define
  \begin{align*}
    \gamma= \Re \sqrt{\frac{1}{4}-\alpha},
  \end{align*}
  then $f(t),g(t)$ can be explicitly computed in terms of hypergeometric
  functions and there exist solutions for which $|f(t)|\sim
  t^{\frac{1}{2}+\gamma}$ and $|g(t)|\sim t^{-1/2+\gamma}$.  
\end{lemma}
In particular, we note that if $f_0=0$ and $g=c\nu>0$ it follows that
\begin{align}
  \label{eq:fsmall}
  \begin{split}
  \frac{g}{\nu}&= c, \\
  f(t)&\leq \frac{3c}{1+t^2},
\end{split}
\end{align}
for all times $t>0$.

In the inviscid case, we instead observe that for $\alpha\geq \frac{1}{4}$,
$\gamma$ vanishes and we hence observe growth and decay with rates $t^{\pm
  \frac{1}{2}}$.
We remark that here another common notational convention is to normalize $\alpha$ instead of
gravity, so that
\begin{align*}
  \dt f(t)&= - \beta^2 g, \\
  \dt g(t)&= \frac{1}{1+t^2} f(t),
\end{align*}
where $\beta^2$ is the Richardson number. In this convention $\gamma$ vanishes
if $\beta^2\geq \frac{1}{4}$, which agrees with the Miles-Howard criterion \cite{howard1961note}.
Since our focus in this article is on the viscous problem, we do not pursue this
further.

\begin{proof}
  We observe that $\dt g=0$ and thus $g(t)=g_0$ for all times.
  Furthermore, since $f$ solves
  \begin{align*}
    \dt f = - \nu (1+t^2) f - g_0
  \end{align*}
  it follows that
  \begin{align}
    \label{eq:evoft}
    f(t)= \exp(-\nu (t+t^3/3)) f_0 + \int_0^t \exp(-\nu (t-s + t^{3}/3 -s^3/3)) ds g_0.
  \end{align}
  It thus only remains to bound the integral
  \begin{align*}
    \int_0^t \exp(-\nu (t-s + t^{3}/3 -s^3/3)) ds &= \int_{0}^{t/2}\exp(-\nu (t-s + t^{3}/3 -s^3/3)) ds \\
    & \quad + \int_{t/2}^{t}\exp(-\nu (t-s + t^{3}/3 -s^3/3)) ds.
  \end{align*}
  On the interval $(0,t/2)$ we may control
  \begin{align*}
    t-s \geq t/2 &\geq s, \\
    t^3/3 - s^3/3 \geq t^3/3 - t^3/24 \geq t^3/4 &\geq \frac{t^2}{2} s.
  \end{align*}
  Thus this integral can be bounded by
  \begin{align*}
    & \quad \int_{0}^{t/2}\exp(-\nu s - \nu \frac{t^2}{2} s) ds = \frac{1}{\nu (1+ \frac{t^2}{2})} (1- \exp(-\nu t/2 - \nu t^3/8)) \\
    &\leq \frac{1}{\nu (1+ \frac{t^2}{2})}.
  \end{align*}
  On the interval $(t/2,t)$ we may integrate by parts to obtain
  \begin{align*}
    & \quad \int_{t/2}^{t}\frac{1}{\nu (1+s^2)} \p_s \exp(-\nu (t-s + t^{3}/3 -s^3/3)) ds \\
    &= \frac{1}{\nu (1+s^2)} \exp(-\nu (t-s + t^{3}/3 -s^3/3))|_{s=t/2}^t\\
    & \quad     - \int_{t/2}^t \exp(-\nu (t-s + t^{3}/3 -s^3/3)) \p_s \frac{1}{\nu (1+s^2)} ds \\
    &\leq \frac{1}{\nu (1+\frac{t^2}{2})},
  \end{align*}
  where we estimated $0\leq \exp(-\nu (t-s + t^{3}/3 -s^3/3)) \leq 1$.

  Combining the estimates for both intervals we deduce that
  \begin{align*}
    |f(t)- \exp(-\nu (t+t^3/3)) f_0| \leq \frac{2}{\nu(1+\frac{t^2}{2})} |g_0| \leq \frac{4}{\nu (1+t^2)}|g_0|, 
  \end{align*}
  which concludes the proof of this case.

  We next turn to the case $\nu=0$, $\alpha \in \R$, for which the equations
  read
  \begin{align*}
    \dt f &= -g, \\
    \dt g &= \frac{\alpha}{1+t^2} f.
  \end{align*}
  It then follows that
  \begin{align*}
    \dt^2 f = - \frac{\alpha}{1+t^2} f.
  \end{align*}
  Thus $f(t)$ is explicitly given in terms of hypergeometric functions of the
  second kind:
  \begin{align*}
    f(t)&= c_1 F^1_2 (-\frac{1}{4}- \frac{1}{4}\sqrt{1-4\alpha}, -\frac{1}{4}+\frac{1}{4}\sqrt{1-4\alpha}, 1/2, -t^2)\\
    &\quad + c_2 F^1_2 (\frac{1}{4}- \frac{1}{4}\sqrt{1-4\alpha}, \frac{1}{4}+\frac{1}{4}\sqrt{1-4\alpha}, 1/2, -t^2)
  \end{align*}
  and
  \begin{align*}
    g = -\dt f.
  \end{align*}
  In particular, for $t$ large these hypergeometric functions behave as
  \begin{align*}
    t^{\frac{1}{2}- \frac{1}{2}\sqrt{1-4\alpha}}, \ t^{\frac{1}{2} +\frac{1}{2}\sqrt{1-4\alpha}}, 
  \end{align*}
  which concludes the proof.
\end{proof}
In the following we discuss the linearized equations around a given traveling wave.

\begin{lemma}
  \label{lem:linearizationaroundwave}
  Consider a traveling wave solution
  \begin{align*}
    \omega_{\star}&= -1 + f(t)\cos(k(x-ty)), \\
    \theta_{\star}&= \alpha y + g(t) \sin(k(x-ty)),
  \end{align*}
  as in Proposition \ref{prop:waves}. Then the linearized equations for the
  perturbations $\omega, \theta$ around this
  wave are given by the system:
\begin{align*}
  \begin{split}
  \dt \omega + y \p_x \omega + \frac{f(t)}{1+t^2}\sin(x-ty) (\p_y-t\p_x) \omega + v \cdot \nabla f(t)\cos(x-ty) = \nu \Delta \omega + \p_x \theta, \\
  \dt \theta + y \p_x \theta + \frac{f(t)}{1+t^2}\sin(x-ty) (\p_y-t\p_x) \theta + v \cdot \nabla g(t)\cos(x-ty)  + v_2 \alpha =0.
  \end{split}
\end{align*}
\end{lemma}

\begin{proof}[Proof of Lemma \ref{lem:linearizationaroundwave}]
  In the linearization we omit the quadratic nonlinearities $v \cdot \nabla
  \omega$ and $v \cdot \nabla \theta$.
\end{proof}

With these preparations, we can sketch the echo mechanism for the linearized
equations around a traveling wave:
\begin{enumerate}
\item At the initial time we introduce a perturbation $e^{ik_0 x} e^{i\eta
    y}$ to the temperature $\theta$.
  This perturbation roughly evolves by transport for a length of time.
\item At around the critical time $\frac{\eta}{k_0}$ the buoyancy term
  $\p_x\theta$ in the evolution equation for $\omega$ will also cause the mode
  $e^{ik_0 x} e^{i\eta y}$ of $\omega$ to grow.
\item The mode $e^{ik_0 x} e^{i \eta y}$ of the vorticity leads to a peak of the
  velocity at that mode at the resonant time. This perturbation \emph{interacts with
  the underlying wave} in $\theta$ at frequencies $\pm 1$ to yield a
contribution to the temperature at frequencies $e^{i (k_0 \pm 1) x} e^{i \eta y}$.
\item We repeat the cycle with $k_0-1$.  
\end{enumerate}
We remark that the \emph{underlying wave} here determines the \emph{step size} of the chain,
$k_0 \leadsto k_0 -1  \leadsto k_0-2 \leadsto \dots \leadsto 1$ and hence choosing the wave
with the lowest frequency, $k=1$, yields the longest chains. At each resonant time our
linear perturbation picks up some energy from the underlying wave and moves to a
lower frequency. It stops once it has reached the lowest frequency $1$.

As noted following Proposition \ref{prop:waves} throughout this article we thus make two simplifications:
\begin{itemize}
\item We only study the case $\alpha=0$, so $g$ is constant.
  We further assume that
  $g=c\nu$ for a small constant $c>0$ and $f_0=0$. In particular, as noted in
  \eqref{eq:fsmall} this implies that $f(t)\leq \frac{c}{1+t^2}$ is also small
  and decreasing in time.
\item We omit the advection term $\frac{f(t)}{1+t^2}\sin(x-ty) (\p_y-t\p_x)$, which is a shear by
  $(0,\frac{1}{1+t^2}\sin(x))$ in coordinates $(x-ty,y)$ and fix the $x$ average
  of both $\theta$ and $\omega$ as zero.
  Since $\frac{1}{1+t^2}$ is quickly decaying and integrable in time, we
  expect that it is possible to remove this simplification.
  Similarly, we expect that changes to the $x$-average can be controlled by
  using the dissipation and mixing effects.
  However, the resulting change of variables and the associated modification of
  integro-differential operators makes this problem technically very challenging
  and is hence omitted in this article.
  This has the additional benefit that after a Fourier transform in $x$ the evolution
  for frequencies $k>0$ and $k<0$ decouple.
\end{itemize}
For later reference we formulate this problem in coordinates $(x-ty,y)$ moving
  with the shear as a definition.
\begin{defi}
  \label{defi:waveequ}
  Let $0< g < \nu$ be a given constant and formally set $f\equiv 0$.
  Then the \emph{wave perturbation equations} in coordinates $(x-ty,y)$ moving
  with the shear are given by
  \begin{align}
    \label{eq:simplewave}
    \begin{split}
    \dt \omega &= (\frac{f(t)}{\nu}\cos(x)\p_y (\p_x^2+(\p_y-t\p_x)^2) ^{-1} \omega)_{\neq} + \nu (\p_x^2+(\p_y-t\p_x)^2) \omega + \p_x \theta, \\
    \dt \theta &= (\frac{g}{\nu} \cos(x) \p_y (\p_x^2+(\p_y-t\p_x)^2) ^{-1} \omega)_{\neq},\\
    (t,x,y) &\in (0,\infty)\times \T \times \R,
    \end{split}
  \end{align}
  where $()_{\neq}$ projects out the $x$-average. 
  We also introduce the notation $\Delta_t=\p_x^2+(\p_y-t\p_x)^2$.
\end{defi}
Similarly to \cite{masmoudi2020stability} it turns out to be advantageous to
equivalently reformulate this system in terms of another unknown.
\begin{lemma}[Good unknown]
  \label{lemma:goodunknown}
  Let $\omega, \theta$ be given functions and $\nu>0$ be given.
  We further define the \emph{good unknown}
  \begin{align*}
    G= \nu \p_x \omega +\p_x^2 \Delta_t^{-1} \theta.
  \end{align*}
  Then $(\omega, \theta)$ is a solution of \eqref{eq:simplewave} if and only if
  $(G, \theta)$ solve
  \begin{align}
    \label{eq:simplewave2}
    \begin{split}
    \dt \theta &= (\frac{g}{\nu} \cos(x) \p_y \Delta_t^{-1} (\p_x^{-1} G + \p_x \Delta_t^{-1} \theta))_{\neq} \\
    \dt G &= \nu \Delta_t G + (f\p_x  \cos(x) \p_y \Delta_t^{-1} (\p_x^{-1}G + \p_x \Delta_t^{-1} \theta))_{\neq}  \\
    &\quad + 2 \p_x^2 (\p_y-t\p_x) \Delta_t^{-2} \theta + \p_x \Delta_t^{-1} (\p_t \theta).
    \end{split}
  \end{align}
  Here we use the short notation $\Delta_t=\p_x^2 + (\p_y-t\p_x)^2$.
\end{lemma}

\begin{proof}
  Direct calculation. We further use that by assumption the $x$-averages of
  $\omega$ vanishes and we hence lose no information by considering the $x$
  derivative only.
\end{proof}
We remark that in \cite{masmoudi2020stability} instead the unknown $K=\Delta_t
\p_x^{-1}G$ is considered. The present choice is made for two reasons:
\begin{itemize}
\item Since $\p_x^2 \Delta_t^{-1}$ is an order $0$ multiplier, $G$ and $\theta$
  can be treated similarly and we for instance do not have to introduce energies
  such as $\|\nabla \omega\|_{L^2}^2 + \|\theta\|_{L^2}^2$ with different
  numbers of derivatives.
\item The evolution equation of $G$ follows more immediately from the equation
  for $\omega$ and $\p_x^2 \Delta_{t}^{-1} \p_t \theta$ exhibits further
  cancellations.
  More precisely, an operator such as
  $\p_x^2\Delta_t^{-1}(\cos(x)\Delta_t^{-1})$ exhibits good bounds since
  $\cos(x)$ induces a Fourier shift and if a frequency $k$ in $x$ resonant, then
  $k+1, k-1$ are non-resonant (see Section \ref{sec:G} for a definition of
  resonance and corresponding estimates).
\end{itemize}

In order to introduce ideas, we first consider a further simplified model
problem.
\begin{defi}
  \label{defi:model}
  In a \emph{model problem} we formally set
  $G\equiv 0$, which yields the following system:
  \begin{align}
    \label{eq:toy}
  \begin{split}
  G &= 0, \\
  \dt \theta &= (\frac{g}{\nu} \cos(x) \p_y \Delta_t^{-1} \p_x \Delta_t^{-1} \theta)_{\neq},
  \end{split}
\end{align}
and its equivalent Fourier characterization (in coordinates $(x+ty,y)$):
\begin{align*}
  \begin{split}
  \dt \tilde{\theta}(t,k,\eta) + i \eta \frac{g(t)}{2} \tilde{\phi}(t,k+1,\eta) + i \eta \frac{g(t)}{2}\tilde{\phi}(t,k-1,\eta) &=0,\\
  \nu (k^2+ (\eta-kt)^2)^2 \tilde{\phi}(t,k,\eta) + ik \tilde{\theta}(t,k,\eta)&=0.
  \end{split}
\end{align*}
\end{defi}
Inserting the second equation into the first, we obtain the following
nearest neighbor ode system for $\theta$:
\begin{align*}
  \dt \tilde{\theta}(k) + c_{k-1} \tilde{\theta}(k-1) + c_{k+1} \tilde{\theta}(k+1)&=0, \\
  c_{l} &= \frac{g(t)}{2} \frac{\eta l}{\nu (l^2 + (\eta-lt)^2)^2}\\
  &= \frac{g(t)}{2\nu}\frac{\eta}{l^3} \frac{1}{(1+(\frac{\eta}{l}-t)^2)^2}
\end{align*}
Next, for a heuristic argument assume that $\frac{\eta}{k^2}$ is very large, $t
\approx \frac{\eta}{k}$ and that $c:=\frac{g}{2\nu}$ is small.
Then $c_k(t)= c \frac{\eta}{k^3}\frac{1}{(1+(\frac{\eta}{k}-t)^2)^2}$ will be
resonant and possibly very large.
In contrast, if $l\neq k$, then $(\frac{\eta}{l}-t)^2 \geq
\frac{1}{4}\max(\frac{\eta}{k^2}, \frac{\eta}{l^2})$ and thus $c_{l}$ can be
controlled in terms of $\frac{c}{l}\max(\frac{\eta}{k^2}, \frac{\eta}{l^2})^{-3}$ and will
be very small.
Thus, it seems reasonable to expect that the growth mechanism is largely
determined by
\begin{align*}
  \dt \tilde{\theta}(k-1)\approx c \frac{\eta}{k^3}\frac{1}{(1+(\frac{\eta}{k}-t)^2)^2} \tilde{\theta}(k).
\end{align*}
Since
\begin{align*}
  \int_{\R} \frac{1}{(1+(\frac{\eta}{k}-t)^2)^2} = \frac{\pi}{2}
\end{align*}
and $c\frac{\eta}{k^3}$ is large, this causes $\tilde{\theta}(k-1)$ to grow by a large multiple of $\tilde{\theta}(k)$.
Iterating this procedure with decreasing $k$, thus suggests that
\begin{align*}
  \theta(1, t\approx 1) \approx \frac{(c\eta)^k}{(k!)^3} \tilde{\theta}(k,t\approx \frac{\eta}{k}).
\end{align*}
This product is maximal for $k \approx \sqrt[3]{c\eta}$
\begin{align*}
  \frac{(c\eta)^k}{(k!)^3} \approx \frac{1}{(c\eta)^{2/3}} \exp(\sqrt[3]{c\eta}).
\end{align*}
Thus, these heuristics suggest that the model problem exhibits norm inflation at
this exponential rate and hence Gevrey $3$ regularity is critical.\\

A main challenge in this article is to show that one indeed can reduce to
the case $t\approx \frac{\eta}{k}$ with $\frac{\eta}{k^2}$ large and that the
full evolution matches the growth behavior of \eqref{eq:toy}, though resonances
in $G$ and the system structure make this problem more challenging.

For later reference, we introduce the following notational conventions:
\begin{defi}
  \label{defi:Ik}
  Let $\eta>0$ be given. Then for any $k > 0$ we define
  \begin{align*}
    t_{k}= \frac{1}{2}\left( \frac{\eta}{k+1}+ \frac{\eta}{k}\right)
  \end{align*}
  and $t_0:= 2 \eta$.
  Then the $k$-th resonant time interval $I_k$ is given by $I_k=(t_k,t_{k-1})$.
\end{defi}
In particular, we observe that for $t \in I_{k}$ it holds that for any $l \neq k$
\begin{align*}
  |t -\frac{\eta}{l}|&\geq \min(|t- \frac{\eta}{k-1}|, |t- \frac{\eta}{k+1}|)\\
 &\geq \frac{1}{4} \max\left(\frac{|\eta|}{k^2}, \frac{|\eta|}{l^2}\right).
\end{align*}
Thus, for any $l \neq k$ and $t \in I_k$ we may estimate
\begin{align*}
  c_l \leq \frac{g(t)}{2 k \nu} (\frac{\eta}{k^2})^{-3} \ll 1.
\end{align*}
Furthermore, we note that the interval $I_k$ has length about $\frac{\eta}{k^2}$
and thus also $\int_{I_k} c_l$ is small.

In order to make our stability analysis in the following sections precise, we need to specify a space $X$ with respect to
which stability is measured.
\begin{defi}[The space $X$]
  \label{defi:X}
  Consider a weight function $\lambda(l)>0$ on $\ell^2$ with
  \begin{align*}
    \sup\limits_{l\in \Z} \frac{\lambda(l\pm 1)}{\lambda(l)} < 2.
  \end{align*}
  Then the Hilbert space $X$ is given by all sequences $(u_l)_l$ such that
  $(\lambda_l u _l)_{l}\in \ell^2$ with the associated inner product.
\end{defi}
The main examples we are interested in are
\begin{itemize}
\item $\lambda(l)\equiv 1$, which yields $X=\ell^2$, 
\item $\lambda(l)= 1+ 2^{-N} |l|^N$, which corresponds to $H^N$ in physical
  space, and 
\item $\lambda(l)=2^{|l|}$, which implies that $u$ is
  analytic in physical space.
\end{itemize}
Throughout the remainder of the article we will consider $X$ to be an arbitrary but fixed such space.
Our plan for the remainder of this article is the following:
\begin{itemize}
\item In Section \ref{sec:small} we show that if we pick initial data $\theta, G
  \in X$ which is localized at frequency $\eta$, then the solution remains
  stable until a time much larger than $\eta^{2/3}$ (and for all time if $\eta$
  is small) and is stable again after the time $2\eta$. Thus, any norm inflation
  has to happen between these two times.
  Furthermore, while the evolution of $(\theta,G)$ is not invertible due to dissipation, for
  small $G$ we show that the evolution of $\theta$ is a small perturbation
  of the identity if the frequency $\eta$ is large.
\item In Section \ref{sec:G} we study the resonance mechanism on the remaining
  time interval. Here we show that resonances could result in bounded norm
  inflation.
  More precisely, we establish upper bounds on possible growth, which are valid
  for all initial data. Lower bounds are studied in Section \ref{sec:blow-up}
\item In Section \ref{sec:proofofmain} we combine the bounds establish for the
  various time regimes to prove the Gevrey $3$ regularity result of Theorem
  \ref{theorem:main}.
\item In Section \ref{sec:blow-up} we show that these upper bounds are optimal
  (up to changes of constants in the exponents) by constructing data exhibiting
  norm inflation. Moreover, we construct data in a critical Gevrey regularity
  class that not only exhibits norm inflation but blow-up as time tends to
  infinity.
  We stress here that in the Euler equations or Vlasov-Poisson equations
  \cite{dengZ2019,zillinger2020landau} the lack of dissipation allowed for an
  inversion of the time direction and thus to more easily construct initial data
  producing desired final data. In contrast, the viscous dissipation of the
  Boussinesq equations prevents any invertibility in time. It thus is a very
  challenging problem to ensure the existence of data achieving norm inflation. 
\end{itemize}

\section{Stability and Gevrey 3 regularity}
\label{sec:stability}
In the heuristic model of Section \ref{sec:wave} we showed that most
growth is expected to occur when
\begin{align*}
  t\approx \frac{\eta}{k}
\end{align*}
and $k$ is such that $c\frac{\eta}{k^3}$ is large and that this growth is
expected to be bounded by
\begin{align*}
  \exp(C \sqrt[3]{c\eta})
\end{align*}
for some constant $C$.

In this section we prove a corresponding stability estimate which shows that
this factor indeed provides an upper bound.
As a complementary result, in Section \ref{sec:blow-up} we show that there exist solutions
attaining such growth (possibly with smaller constant $C$). We then use the norm
inflation solutions as building blocks to construct solutions exhibiting blow-up.

In our analysis we first show that if $\eta$ is much smaller than $c^{-1}$, then
the evolution is globally stable for all times.
If $\eta$ is not this small we consider four time regimes (see also Figure \ref{fig:regimes}):
\begin{itemize}
\item Define $k_0\approx \sqrt[3]{c\eta\pi}$ (rounded down) and let $k_1=20 k_0$.
  Then for all $k\geq k_1$ it holds that $\frac{c\eta}{k^3}$ is much smaller
  than $1$.
  We hence consider the interval of times
  \begin{align*}
    0\leq t\leq t_{k_1}
  \end{align*}
   with $t_{k_1}= \frac{1}{2}(\frac{\eta}{k_1+1} + \frac{\eta}{k_1})<
   \frac{\eta}{k_1}$ (see Definition \ref{defi:Ik}) as the
  \emph{small time} regime.
  Here we show that the evolution is stable and, in a suitable sense, close to
  the identity.
\item  We observe that for $k\geq k_0$, $\frac{c\eta}{k^3}$ is bounded by $1$
  but not necessarily small. We hence call
  \[(t_{k_1}, t_{k_0})\]
  the \emph{intermediate} time regime.
  We here derive rough upper bounds showing that the norm at most
  grows by $\exp(40 \sqrt[3]{c\eta})$.
\item We next consider the interval \[(t_{k_0}, 2\eta),\] which is composed of
  intervals $I_k$ with
  \begin{align*}
    \frac{c\eta\pi}{k^3}\geq 1
  \end{align*}
  and potentially very large.
  Here we encounter the main resonance mechanism discussed in the heuristic model of Section
  \ref{sec:wave} and establish an upper growth bound by
  $\exp(10 \sqrt[3]{c\eta})$. We call this interval the \emph{resonant} regime.
  In Section \ref{sec:blow-up} we further introduce an additional time $t_{k_3}$
  with $k_3= \frac{k_0}{1000}$ up to which we also establish lower bounds on
  norm inflation.
\item Finally, we show that on the interval \[(2\eta,\infty)\] the evolution is
  stable and that the evolution of $\theta$ is a bounded perturbation of the identity.
  Thus, there can be no further norm inflation in this \emph{long time} regime.
\end{itemize}
\begin{figure}
  \centering
  \includegraphics[width=0.5\linewidth]{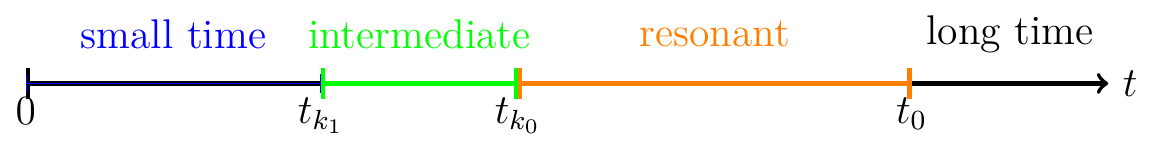}
  \caption{The time regimes are determined by the size of $\frac{c\eta}{k^3}$
    for $t \in I_k$.
    If $k$ is very large, i.e. for small times, this factor is very small. If
    instead $k\leq k_0$, then this factor is large, leading to possibly large
    resonances.}
  \label{fig:regimes}
\end{figure}

We stress that the asymptotic stability in the long time regime highlights that norm
inflation by a finite time does not imply instability as time tends to infinity.
Hence in Section \ref{sec:blow-up} we need to combine infinitely many solutions
exhibiting norm inflation to obtain non-trivial asymptotic behavior like
blow-up.

\subsection{Small Time, Large Time and Small Frequencies}
\label{sec:small}

In this subsection we consider the regimes where $t$ or $\eta$ are small or when
$t>2 \eta$ is large.
In all these regimes it turns out that the evolution is well-controlled and, in
a sense, a small perturbation of the identity for $\theta$ and a dissipative
equation for $G$.
Our main results are summarized in the following theorem.
\begin{theorem}[Stable regimes]
\label{theorem:easy}
  Let $\eta \in \R$ be given and for simplicity of notation assume $\eta>0$.
  Let further $0\leq \frac{g}{\nu}=c<0.001$.
  
  Then we have the following stable time intervals depending on $\eta$:
  \begin{enumerate}
  \item If $c \eta \ll 1$, then for all times $t>0$ it holds that
    \begin{align*}
      40^2 \|\theta(t)\|_X + \|G(t)\|_{X} \leq 4( 40^2 \|\theta(0)\|_{X} + \|G(0)\|_{X}).
    \end{align*}
  \item Suppose that $c\eta \geq 0.001$, then for all times $t> 2\eta=:t_0$ it holds
    that
    \begin{align*}
      \|\theta(t)\|_X + \|G(t)\|_{X} \leq 4 (\|\theta(t_0)\|_{X} + \|G(t_0)\|_{X}).
    \end{align*}
    Moreover,
    \begin{align*}
      \|\theta(t)-\theta(t_0)\|_X\leq c \eta^{-1} 2 (\|\theta(t_0)\|_{X} + \|G(t_0)\|_{X}).
    \end{align*}
    The evolution of $\theta$ is a bounded perturbation of the identity.
  \item Consider the early time regime given by $0<t<
    \frac{\eta}{\sqrt{8000 c\eta \pi}}$. Then on this time interval it
    holds that
        \begin{align*}
      40^2 \|\theta(t)\|_X + \|G(t)\|_{X} \leq 2 (40^2 \|\theta(0)\|_{X} +\|G(0)\|_{X}).
    \end{align*}
  \end{enumerate}
\end{theorem}
The factor $40$ here is used to control a potentially large interaction term between
$\theta$ and $G$. Since this term is small in the large time regime, there we
may consider more symmetric energies.

We remark that this discussion omits the two regimes $t \in I_k$, where
in the intermediate time regime
\begin{align*}
  \frac{c\eta}{k^3} \pi \in (\frac{1}{8000},1)
\end{align*}
or in the resonant time regime
\begin{align*}
  \frac{c\eta}{k^3} \pi \geq 1.
\end{align*}
Establishing bounds on the growth in these regimes requires
significantly more effort and is the main challenge of Sections~\ref{sec:G} and
\ref{sec:intermediate}.

\begin{proof}[Proof of Theorem \ref{theorem:easy}]
  We recall from Lemma \ref{lemma:goodunknown} that the evolution equations with
  respect to $\theta$ and $G$ are given by \eqref{eq:simplewave2}:
  \begin{align*}
    \dt G &= \nu \Delta_{t} G + \nu f \p_x \cos(x)\p_y \phi+ 2\p_x^3 (\p_y-t\p_x) \Delta_t^{-2} \theta - \p_x^2\Delta_t^{-1}(g \cos(x) \p_y\phi), \\
    \dt \theta &= g \cos(x) \p_y\phi,
  \end{align*}
  where
  \begin{align*}
    G&= \nu \p_x \omega - \p_x^2 \Delta_{t}^{-1} \theta, \\
    \phi &= \nu^{-1} \Delta_{t}^{-1} \p_x^{-1} G + \nu^{-1}\p_x \Delta_t^{-2} \theta.
  \end{align*}

  \underline{The model case:}
  In order to introduce ideas, let us first consider the model problem of
  Definition \ref{defi:model} where we
  formally set $G \equiv 0$.
  Then we define the Fourier multiplier
  \begin{align*}
    A(t,l,\eta) = \exp\left(C (\arctan(\frac{\eta}{l-1}-t) + \arctan(\frac{\eta}{l}-t) +\arctan(\frac{\eta}{l+1}-t)) \right),
  \end{align*}
  where we omit terms in which we would divide by $0$ and $C=0.01$ is a small constant.
  In particular, we note that $A(t,l,\eta)\approx 1$ uniformly in $t,l, \eta$ and
  that $A$ is decreasing in time.
  We then aim to show that the energy
  \begin{align*}
    \|A \theta\|_{X}^2
  \end{align*}
  is non-increasing in time and thus serves as a Lyapunov functional.
  We hence compute the time derivative of this energy and observe
  that on the one hand
  \begin{align*}
    \langle A \theta, \dot{A}\theta \rangle\leq 0
  \end{align*}
  is non-positive provides decay in terms of the Fourier multipliers
  \begin{align*}
    -C \frac{1}{1+(\frac{\eta}{l-1}-t)^2}  -C \frac{1}{1+(\frac{\eta}{l}-t)^2}-C \frac{1}{1+(\frac{\eta}{l+1}-t)^2}.
  \end{align*}
  On the other hand the contribution
  \begin{align*}
    \langle A \theta, A \dt \theta \rangle
  \end{align*}
  is possibly positive, but can be bounded from above in terms of the Fourier
  multipliers
  \begin{align*}
    c \frac{(l-1)\eta}{((l-1)^2 + (\eta-(l-1)t)^2)^2} +c \frac{l\eta}{(l^2 + (\eta-lt)^2)^2}+ c \frac{(l+1)\eta}{((l+1)^2 + (\eta-(l+1)t)^2)^2},
  \end{align*}
  by using Young's inequality and the fact that there is only interaction
  between neighboring modes.
  It hence suffices to show that
  \begin{align}
    \label{eq:easy}
    c \frac{l\eta}{(l^2 + (\eta-lt)^2)^2} = \frac{c\eta}{l^3} \frac{1}{(1+(\frac{\eta}{l})^2)^2}  \leq \frac{C}{2} \frac{1}{1+(\frac{\eta}{l}-t)^2},
  \end{align}
  In any of the regimes considered we can then use different arguments to show
  that \eqref{eq:easy} holds:
  \begin{enumerate}
  \item If $c \eta\ll 1$ is small, $\frac{c\eta}{l^3}$ is small for all $l$ and thus
    \begin{align*}
       c \frac{l\eta}{(l^2 + (\eta-lt)^2)^2} = c\frac{\eta}{l^3} \frac{1}{(1+(\frac{\eta}{l}-t)^2)^2}\leq c \eta \frac{1}{1+(\frac{\eta}{l}-t)^2}\leq \frac{C}{2} \frac{1}{1+(\frac{\eta}{l}-t)^2}
    \end{align*}
    satisfies \eqref{eq:easy} for all $l$. 
  \item Similarly, if $t>2\eta$, then we can control by
    \begin{align*}
      c\frac{\eta}{l^3} \frac{1}{1+\eta^2} \frac{1}{1+(\frac{\eta}{l}-t)^2} \leq c \frac{1}{1+(\frac{\eta}{l}-t)^2} \leq \frac{C}{2} \frac{1}{1+(\frac{\eta}{l}-t)^2},
    \end{align*}
    and hence \eqref{eq:easy} holds for all $l$. 
  \item Finally, let $ c \eta \gtrsim 1$ and suppose that  $t \approx \frac{\eta}{k}
    \lesssim c^{-1/3} \eta^{2/3}$ is not too large.
    Then it follows that
    \begin{align*}
      c \frac{\eta}{k^3} \approx c \eta^{-2} t^{3}\ll 1,
    \end{align*}
    which is the reason for our choice of upper bound on $t$.
    
    We then distinguish two cases. If $l>\frac{k}{2}$, then 
    \begin{align*}
      c \frac{\eta}{l^3} \leq 8 c \frac{\eta}{k^3} \ll 1,
    \end{align*}
    and thus the estimate \eqref{eq:easy} holds for such $l$.
    If instead $l\leq \frac{k}{2}$, then $\frac{\eta}{l}\geq 2 \frac{\eta}{k}\geq
  2t$ and hence
  \begin{align*}
    c \frac{\eta}{l^3}\frac{1}{1+(\frac{\eta}{l}-t)^2} \leq c  \frac{\eta}{l^3}\frac{1}{1+\frac{1}{4}(\frac{\eta}{l})^2} \leq 4 c \ll 1.
  \end{align*}
  Thus, the estimate \eqref{eq:easy} also holds for such $l$, which concludes
  the proof for the model case.
\end{enumerate}

\underline{The general case:}
Building on the insights developed in the model case, we turn to the full
problem:
  \begin{align*}
    \dt G &= \nu \Delta_{t} G + \nu f\p_x  \cos(x) \p_y \phi + \p_x^2\Delta_t^{-1}(g \cos(x) \p_y\phi) + 2 (\p_y-t\p_x)\p_x^3 \Delta_{t}^{-2}\theta, \\
    \dt \theta &= g \cos(x) \p_y\phi,\\
    \phi &= \nu^{-1} \Delta_{t}^{-1} \p_x^{-1} G + \nu^{-1}\p_x \Delta_t^{-2} \theta.        
  \end{align*}
Here, main additional challenges are given by the slower decay of
$\Delta_{t}^{-1}$ as compared to $\Delta_t^{-2}$ and the potentially large size
of the multiplier corresponding to $(\p_y-t\p_x)\p_x^3 \Delta_{t}^{-2}$.

\underline{The large time case $t > 2 \eta$:}
We first discuss the simplest case where $t>2\eta$ is large.

Here we observe the following multiplier estimates:
\begin{align}
  \label{eq:multipliers}
  \begin{split}
  g \nu^{-1} \p_y \Delta_{t}^{-2}  \mapsto \frac{c\eta}{k^3}\frac{1}{(1+(\frac{\eta}{k}-t)^2)^2} &\leq c \frac{1}{(1+(t/2)^2)^{3/2}}, \\
  g \nu^{-1} \p_y \Delta_{t}^{-1} \p_x^{-1} \mapsto \frac{c\eta}{k^3} \frac{1}{1+(\frac{\eta}{k}-t)^2} &\leq c \eta \frac{1}{(1+(t/2)^2)}, \\
  2 (\p_y-t\p_x)\p_x^3 \Delta_{t}^{-2} \mapsto 2 \frac{\frac{\eta}{k}-t}{(1+(\frac{\eta}{k}-t)^2)^2} &\leq \frac{2}{(1+(t/2)^2)^{3/2}}, \\
  \p_x^2\Delta_t^{-1} \mapsto \frac{1}{1+(\frac{\eta}{k}-t)^2} &\leq \frac{1}{1+(t/2)^2},\\
  \nu f &\leq \frac{c}{1+t^2},
\end{split}
\end{align}
where we used that $t-\frac{\eta}{k}\geq t -\eta \geq t/2$.

We thus conclude that
\begin{align*}
  \dt (\|G\|_X^2 + \|\theta\|_{X}^2) \leq - \nu \|\nabla_t G\|_{X}^2 + \max(c \eta \frac{1}{(1+(t/2)^2)}, \frac{2}{(1+(t/2)^2)^{3/2}}) (\|G\|_X^2 + \|\theta\|_{X}^2).
\end{align*}
Since
\begin{align*}
  \int_{2\eta}^{\infty} c \eta \frac{1}{(1+(t/2)^2)} + \frac{2}{(1+(t/2)^2)^{3/2}} dt \leq c + (1+|\eta|)^{-1}
\end{align*}
it hence follows that
\begin{align*}
  \|G(t)\|_X^2 + \|\theta(t)\|_{X} \leq \exp(c+ (1+|\eta|)^{-1}) (\|G(2 \eta)\|_X^2 + \|\theta(2\eta)\|_{X}).
\end{align*}
The solution is stable.
Moreover, we may insert these bounds into the integral equations to conclude
that
\begin{align*}
  \|\theta(t)-\theta(2\eta)\|_X &\leq \int_{2\eta}^{\infty} \|\dt \theta\|_X \leq c \|G\|_{L^\infty X} + (1+|\eta|)^{-1}\|\theta\|_{L^\infty X} \\
  &\leq
  (c+(1+|\eta|)^{-1}) \exp(c+ (1+|\eta|)^{-1}) (\|G(2 \eta)\|_X^2 + \|\theta(2\eta)\|_{X}).
\end{align*}
In view of the dissipation term, for $G(t)$ we instead estimate
\begin{align*}
  \|G(t)\|_X^2 - \|G(2\eta)\|_{X}^2 \leq (1+|\eta|)^{-1} \exp(c+ (1+|\eta|)^{-1}) (\|G(2 \eta)\|_X^2 + \|\theta(2\eta)\|_{X}).
\end{align*}

\underline{Small frequencies:}
We next turn to considering the regime where
\begin{align*}
  c\eta \leq 0.001
\end{align*}
is very small.
Again considering the multipliers shown in \eqref{eq:multipliers} we observe
that, for instance, we may estimate
\begin{align}
  \label{eq:smallfreq1}
  \begin{split}
  g \nu^{-1} \p_y \Delta_{t}^{-2}  \mapsto \frac{c\eta}{k^3}\frac{1}{(1+(\frac{\eta}{k}-t)^2)^2} &\leq 0.001 \frac{1}{(1+(\frac{\eta}{k}-t)^2)^{2}}, \\
  g \nu^{-1} \p_y \Delta_{t}^{-1} \p_x^{-1} \mapsto \frac{c\eta}{k^3} \frac{1}{1+(\frac{\eta}{k}-t)^2} &\leq 0.001 \frac{1}{1+(\frac{\eta}{k}-t)^2},
\end{split}
\end{align}
for all $k \in \Z \setminus\{0\}$.
The bound on the right-hand-side here is integrable in time with small norm for
any fixed $k$. However, if one were to first take the supremum with respect to
$k$ and then integrate the resulting norm would be large, since for any $t$
there exists some $k$ such that $t\approx \frac{\eta}{k}$.

We thus consider a \emph{frequency-dependent} multiplier $A$ of a similar form as in the
model problem in order to construct a Lyapunov functional:
\begin{align*}
  A(t,k) =  \exp\left(0.1 \sum_{l \in \{k-1,k,k+1\}}(\arctan(\frac{\eta}{l}-t)-\arctan(\frac{\eta}{l}-t_k)) \right).
\end{align*}
We note that $0.1 |\arctan(\frac{\eta}{l}-t)-\arctan(\frac{\eta}{l}-t_k)| \leq
0.1 \pi$ and hence $A$ is comparable to $1$ at all times.
Moreover, $A(t)$ is decreasing and 
\begin{align*}
  \dot A = A \mathcal{F}^{-1} \frac{-0.1}{1+(\frac{\eta}{k}-t)^2}\mathcal{F}
\end{align*}
is sufficiently negative to absorb multipliers such as in \eqref{eq:smallfreq1}.

We observe that one multiplier 
\begin{align}
\label{eq:smallfreq2} 
  2 (\p_y-t\p_x)\p_x^3 \Delta_{t}^{-2} \mapsto 2 \frac{\frac{\eta}{k}-t}{(1+(\frac{\eta}{k}-t)^2)^2} &\leq 2 \frac{1}{1+(\frac{\eta}{k}-t)^2}
\end{align}
does not necessarily include a small prefactor.
For this reason we make the following ansatz for our energy 
\begin{align*}
  E(t):= 40^2 \|A(t)\theta\|_{X}^2 + \|A(t)G\|_{X}^2,
\end{align*}
which weighs $\theta$ more highly. Here, with slight abuse of notation, we
identify $A$ and the corresponding Fourier multiplier and hence write $A\theta$
instead of $\mathcal{F}^{-1}(A\mathcal{F} \theta)$.

We then claim that $E(t)$ is non-increasing and thus provides a Lyapunov
functional.
We compute the time-derivative of $E(t)$ as
\begin{align*}
  \dt E(t) &=  2\cdot 40^2 \langle A \theta, \dot A \theta \rangle + 2 \langle A G, \dot{A} G \rangle\\
  & \quad + 2\cdot 40^2 \langle A \theta, A \dt \theta \rangle + 2 \langle A G, A \dt G \rangle,
\end{align*}
where $\langle , \rangle$ denotes the inner product in $X$.

By the estimate \eqref{eq:smallfreq1} it then follows that 
\begin{align*}
  40^2 \langle A \theta, A \dt \theta \rangle & \leq 0.001 \left\|\sum_{l \in \{k-1,k,k+1\}}\frac{1}{\sqrt{1+(\frac{\eta}{l}-t)^2}} 40 A\theta\right\|_X^2\\
                                              & \quad + 40\cdot 0.001 \left\|\sum_{l \in \{k-1,k,k+1\}}\frac{1}{\sqrt{1+(\frac{\eta}{l}-t)^2}} 40 A\theta\right\|_X \\
  & \quad \quad \quad \cdot \left\|\sum_{l \in \{k-1,k,k+1\}}\frac{1}{\sqrt{1+(\frac{\eta}{l}-t)^2}} A G\right\|_X, 
\end{align*}
where we used a sum to account for the Fourier shift due to $\cos(x)$.

Using Young's inequality in the last factor we thus conclude that any possible
growth due to $40^2 \langle A \theta, A \dt \theta \rangle$ can be absorbed by
the decay terms due to $\dot A$.

Similarly, for the time derivative of $G$ we compute:
\begin{align*}
  \langle A G, A \dt G \rangle &= - \nu \langle A \nabla_t G, A\nabla_t G \rangle\\
  & \quad +  \langle A G, A \p_x^3 (\p_y-t\p_x) \Delta_{t}^{-2} A \theta \rangle\\
  & \quad + \langle AG, A \nu f \cos(x) \p_y \phi \rangle \\
  & \quad + \langle A G, A \p_x^2 \Delta_{t}^{-1}\p_t \theta \rangle.
\end{align*}
The dissipation term is non-positive and thus potentially helpful. However, as
$\nu$ is allowed to be arbitrarily small, we do not make use of it in the
following.
For the next term we make use of the estimate \eqref{eq:smallfreq2} and express
\begin{align*}
  \langle A G, A \p_x^3 (\p_y-t\p_x) \Delta_{t}^{-2} A \theta \rangle = \langle A G, A \frac{1}{40} \p_x^3 (\p_y-t\p_x) \Delta_{t}^{-2} 40 A \theta \rangle
\end{align*}
to obtain a small factor $\frac{1}{40}$.

Finally, for the contributions by $\nu f \p_y \phi$ and by 
\begin{align*}
  \p_x^2 \Delta_{t}^{-1}\p_t \theta = \p_x^2 \Delta_{t}^{-1} g \cos(x) \p_y \phi
\end{align*}
we simply control the multiplier corresponding to $\p_x^2 \Delta_{t}^{-1}$ by
$1$ and $\nu f\leq \frac{c}{1+t^2}\leq c$.
The estimate of this term then follows as for $\langle A \theta, A \p_t \theta \rangle$.

In conclusion, we have thus shown that
\begin{align*}
  \dt E \leq 40^2 \langle A \theta, \dot A \theta \rangle +  \langle A G, \dot{A} G \rangle \leq 0
\end{align*}
and thus $E(t)$ is non-increasing, which is our desired stability estimate.

\underline{The small time regime:}
As a last regime we consider the case when $c\eta$ is allowed to be large but
$t$ is bounded by $\frac{\eta}{\sqrt[3]{8000 c\eta}}$.
Here we again aim to construct a Lyapunov functional by combining the estimates
of the previous two regimes, but will require small changes to the multiplier $A$.

Consider for instance the multiplier
\begin{align*}
  g \nu^{-1} \p_y \Delta_{t}^{-2}  \mapsto \frac{c\eta}{k^3}\frac{1}{(1+(\frac{\eta}{k}-t)^2)^2},
\end{align*}
and let $k_1 \approx \sqrt[3]{8000 c\eta}$ (rounded up).
Then for all $k \geq 0.5 k_1$ it holds that
\begin{align*}
  \frac{c\eta}{k^3} \leq \frac{2^3}{8000}\leq 0.001,
\end{align*}
and for these $k$ we may hence argue as in the small frequency regime.
Conversely, if $k < 0.5 k_1$, then
\begin{align}
  \label{eq:smallfreqnonres}
  \frac{\eta}{k} > 2 t \Rightarrow |\frac{\eta}{k}-t| \geq \frac{1}{2}|\frac{\eta}{k}|
\end{align}
and hence
\begin{align*}
  \frac{c\eta}{k^3}\frac{1}{(1+(\frac{\eta}{k}-t)^2)^2} \leq c \frac{1}{1+(\frac{\eta}{k}-t)^2}.
\end{align*}

Similarly, for the multiplier
\begin{align*}
  g \nu^{-1} \p_y \Delta_{t}^{-1} \p_x^{-1} \mapsto \frac{c\eta}{k^3} \frac{1}{1+(\frac{\eta}{k}-t)^2}
\end{align*}
we may estimate as in the small frequency regime if $k\geq 0.5 k_1$.
If $k<0.5 k_1$ we cannot spare powers of $t$ but note that by \eqref{eq:smallfreqnonres}:
\begin{align*}
  & \quad \int_{t \in \R: |t|<\frac{1}{2}|\frac{\eta}{k}|} \frac{c\eta}{k^3}\frac{1}{1+(\frac{\eta}{k}-t)^2} \\
  &\leq \frac{c\eta}{k^3} \int_{\tau\in R: |\tau|\geq \frac{1}{2}|\frac{\eta}{k}|}\frac{1}{1+\tau^2} \\
  &\leq \frac{c\eta}{k^3} \frac{1}{1+\frac{1}{2}|\frac{\eta}{k}|} \leq \frac{c}{k}.
\end{align*}
We thus consider a slightly different multiplier
\begin{align*}
  A(t) = \exp( 0.01 \arctan(\frac{\eta}{k}-t)- \arctan(\frac{\eta}{k}-t_k) + 1_{k<k_0} \int_{t_k}^t 10 \frac{c\eta}{k^3}\frac{1}{1+(\frac{\eta}{k}-s)^2}ds).
\end{align*}
Then by the same argument as in the previous regime it follows that
\begin{align*}
   E(t):= 40^2 \|A(t)\theta\|_{X}^2 + \|A(t)G\|_{X}^2
\end{align*}
is non-increasing. Since $A(t,k)$ is comparable to $1$, $E$ is thus a Lyapunov functional.
\end{proof}

This theorem shows that any norm inflation has to happen in the remaining time
regimes, where we distinguish between the intermediate time regime
\begin{align*}
  t \in I_k : \frac{c\eta}{k^3} \in (\frac{1}{8000}, \frac{1}{\pi}),
\end{align*}
which is considered in the following Section \ref{sec:intermediate} and the resonant time regime
\begin{align*}
  t \in I_k : \frac{c\eta}{k^3}\geq \pi^{-1},
\end{align*}
which is considered in Section \ref{sec:G}.

In both cases we determine the growth on individual intervals $I_k$ and then
show that the total growth obtained by iterating over all possible values of $k$
is bounded by $\exp(C \sqrt[3]{c\eta})$ for a suitable constant $C$, which is
consistent with Gevrey $3$ regularity.

\subsection{The Intermediate Regime}
\label{sec:intermediate}
In this section we consider the time intervals $I_k$ where
\begin{align*}
  c\frac{\eta}{k^3} \in (\frac{1}{8000}, \frac{1}{\pi}).
\end{align*}
We recall that the heuristic model of Section \ref{sec:wave} suggested a growth
of the neighbors of the resonant frequency by a factor $c \frac{\eta}{k^3}\pi$. In the present regime this
factor is still smaller than $1$ but not very small anymore.

As a main result of this subsection we show that on each interval $I_k$
the norm in $X$ grows at most by a factor $C$ independent of $\eta$.
As we show in Section \ref{sec:blow-up} this bound is probably far from optimal, but sufficient
for our stability estimates.
Indeed, we observe that if $t \approx \frac{\eta}{k}$ then in this regime $t$ is
proportional to $\eta^{2/3}$ and hence might be very large, but $k$ is contained
in the interval
\begin{align*}
  (\sqrt[3]{\frac{1}{\pi}c \eta}, \sqrt[3]{8000c \eta}).
\end{align*}
In particular, there are at most $20 \sqrt[3]{c\eta}$ such $k$.

Hence, a growth bound by a factor $C$ for each $k$ implies a total growth bound by
\begin{align*}
  C^{20 \sqrt[3]{c\eta}}= \exp(20 \ln(C) \sqrt[3]{c\eta}),
\end{align*}
which is consistent with Gevrey $3$ regularity.

Our results are summarized in the following theorem.
\begin{theorem}
  \label{theorem:intermediate}
  Let $\eta \gg 1$, $0<c<0.001$ be given. Let further $k\in \N$ with
  \begin{align*}
    \sqrt[3]{\pi c \eta} \leq k \leq \sqrt[3]{8000 c\eta\pi}.
  \end{align*}
  Then on the time interval $I_k=(t_{k}, t_{k-1})$ and for $t=t_{k-1}$ it holds that
  \begin{align}
    \label{eq:intermediateIk}
    \|\theta(t)\|_{X}+ \|G(t)\|_X \leq e^{3\pi} (\|\theta(t_{k})\|_X + \|G(t_{k})\|_X).
  \end{align}

  In particular, if $k_0\approx \sqrt[3]{\pi c \eta}$, $k_1\approx \sqrt[3]{8000
    c\eta\pi}$ (rounded to an integer), then it holds that
  \begin{align*}
    \|\theta(t_{k_0})\|_{X} + \|G(t_{k_0})\|_X \leq e^{20\cdot 3 \pi \sqrt[3]{c\eta}}(\|\theta(t_{k_1})\|_{X}+ \|G(t_{k_1})\|_X).
  \end{align*}
\end{theorem}

\begin{proof}[Proof of Theorem \ref{theorem:intermediate}]
Suppose for the moment the estimate \eqref{eq:intermediateIk} holds for all
$k_0\leq k\leq k_1$.
Then it follows that
\begin{align*}
  \|\theta(t_{k_0})\|_{X}\leq e^{2}\|\theta(t_{k_0+1})\|_X \leq e^2 e^2 \|\theta(t_{k_0+2})\|_X\leq e^{2|k_0-k_1|}\|\theta(t_{k_1}).
\end{align*}
Thus the claimed bound follows by noting that $|k_0-k_1|\leq
\sqrt[3]{8000c\eta}= 20 \sqrt[3]{c\eta}$.

It remains to prove the estimate \eqref{eq:intermediateIk}, for which we again
want to use a multiplier argument.
More precisely, we note that on the interval $I_k$ the coefficients
in front of the mode $k$ are the largest.
For instance,
  \begin{align*}
    \frac{c\eta}{k^3} \frac{1}{(1+(\frac{\eta}{k}-t)^2)^2} \geq |\frac{c\eta}{l^3} \frac{1}{(1+(\frac{\eta}{l}-t)^2})^2|
  \end{align*}
  for all $l \in \Z \setminus \{0\}$.

We thus consider the \emph{frequency-independent} multiplier
\begin{align*}
  A(t)= \exp(-3 \int_{t_{k}}^t \frac{1}{1+(\frac{\eta}{k}-\tau)^2} d\tau),
\end{align*}
Unlike in the regimes considered in Section \ref{sec:small} we here do not need
to control commutators due to the evaluation of $A$ at neighboring modes. Hence, here
we do not require a small exponent $0.01$ but can choose $3$ as a comparably
large exponent. We remark that a similar idea has been used in the analysis of
the nonlinear equations in \cite{masmoudi2020stability}.
We now claim that
\begin{align*}
 E(t):= A(t) (\|\theta(t)\|_X^2 + \|G(t)\|_{X}^2)
\end{align*}
is non-increasing for $t \in I_k$.
Since $A(t_k)=1$, $A(t)$ is decreasing in time and $A(t)\geq \exp(-3\pi)$, this
further implies
\begin{align*}
  \|\theta(t)\|_X^2 + \|G(t)\|_{X}^2 = A(t)^{-1} E(t)\leq e^{3\pi} E(t_{k})= e^{3\pi} (\|\theta(t_k)\|_X^2 + \|G(t_k)\|_{X}^2),
\end{align*}
and thus establishes the desired bound.

It remains to prove that $E(t)$ indeed is non-increasing.
We recall that
\begin{align*}
  \dt \theta = c \cos(x) (\p_y \p_x^{-1}\Delta_t^{-1} G + \p_y \p_x \Delta_t^{-2}\theta)
\end{align*}
and that at frequency $l$
\begin{align*}
  c \p_y \Delta_{t}^{-2}  &\mapsto \frac{c\eta}{l^3}\frac{1}{(1+(\frac{\eta}{l}-t)^2)^2}, \\
  c \p_y \Delta_{t}^{-1} \p_x^{-1} &\mapsto \frac{c\eta}{l^3} \frac{1}{1+(\frac{\eta}{l}-t)^2}. 
\end{align*}
Since $t \in I_k$ these multipliers are largest (in absolute value) if $l=k$ and
are hence both bounded by
\begin{align*}
  \frac{c\eta}{k^3}\frac{1}{1+(\frac{\eta}{k}-t)^2}
\end{align*}
and can therefore be absorbed into the decay of $A(t)$.

Similarly, we recall that
\begin{align*}
  \dt G &= \nu \Delta_{t} G + \nu f\p_x  \cos(x) \p_y \phi + \p_x^2\Delta_t^{-1}(g \cos(x) \p_y\phi) + 2 (\p_y-t\p_x)\p_x^3 \Delta_{t}^{-2}\theta, \\
  \nu \p_y \theta &= \p_y \p_x^{-1}\Delta_t^{-1} G + \p_y \p_x \Delta_t^{-2}\theta.
\end{align*}
Then the dissipation term yields a non-positive contribution by
\begin{align*}
  -\nu A(t) \|\nabla_t G\|_X^2\leq 0.
\end{align*}
For the contributions by
\begin{align*}
  \nu f\p_x  \cos(x) \p_y \phi +\p_x^2\Delta_t^{-1}(g \cos(x) \p_y\phi),
\end{align*}
we may very roughly estimate $f\nu \leq c$,
$\frac{1}{1+(\frac{\eta}{l}-t)^2}\leq 1$ and argue as for the estimate of $\dt
\theta$.
It thus only remains to discuss
\begin{align*}
  2 (\p_y-t\p_x)\p_x^3 \Delta_{t}^{-2} \mapsto \frac{2 (\frac{\eta}{l}-t)}{(1+(\frac{\eta}{l}-t)^2)},
  \end{align*}
  which in absolute value can be bounded by
  \begin{align*}
    2 \frac{1}{1+(\frac{\eta}{k}-t)^2}
  \end{align*}
  and hence can also be absorbed by the decay of $A(t)$.

In conclusion, we have shown that $E(t)$ indeed is non-increasing, which
completes the proof.  
\end{proof}

We next turn to studying the main resonance mechanism for times $t \in I_k$ with
\begin{align*}
  \frac{c\eta}{k^3}\pi \geq 1.
\end{align*}
Here we derive a slightly suboptimal upper bound of possible growth by a factor
\begin{align*}
  1+ \frac{c\eta}{k^3}\pi \leq 2 \frac{c\eta}{k^3}\pi,
\end{align*}
which then implies a total growth bound by
\begin{align*}
  \prod_{k=1}^{k_0} 2 \frac{c\eta}{k^3}\pi \leq \exp(\sqrt[3]{2 c\eta}).
\end{align*}
In Section \ref{sec:blow-up} we show that for special initial data this
growth is attained up to loss of factor in the exponent.
We remark that in the Euler equations or Vlasov-Poisson equations
\cite{dengZ2019,zillinger2020landau} one can explicitly construct initial data
generating desired data at a later time $t_{k_0}$ by inverting the time direction. As the
Boussinesq equations include viscous dissipation this is not possible in the
present setting and we instead have to invest considerable effort to
characterize which data can be generated at time $t_{k_0}$ starting from
suitable initial data.

\subsection{The Resonance Mechanism}
\label{sec:G}

In this section we study the norm inflation mechanism on the time interval
$I_k$, when
\begin{align*}
  1\leq \pi \frac{c\eta}{k^3}\leq \pi c\eta.
\end{align*}
Since this factor might be very large, we cannot allow for rough, Gronwall-type
bound of growth by
\begin{align*}
  \exp(\pi \frac{c\eta}{k^3}).
\end{align*}
Instead, we show that, as sketched for the model problem of Section
\ref{sec:wave}, only the neighbors of the mode $k$ grow by a factor at most
\begin{align*}
  2 \pi \frac{c\eta}{k^3},
\end{align*}
and all other modes only change mildly.

The main estimates of this section are summarized in the following theorem.
\begin{theorem}[Resonance Mechanism]
  \label{theorem:mechanism}
  Let $0<c<0.01$ be as in Theorem \ref{theorem:main} and let $\eta \gtrsim
  c^{-1}$ be given and $k\in \N$ be such that $c\frac{\eta}{k^3}\gtrsim 1$.
  Then for any data $\theta^{\star}, G^\star \in X$ prescribed at the time
  $t_k=\frac{1}{2} (\frac{\eta}{k+1}+ \frac{\eta}{k-1})$ the corresponding
  solution of the wave perturbation equation
  \eqref{eq:simplewave2} satisfies
  \begin{align*}
    \|\theta(t) - \theta^{\star}\|_X \leq  2 c \frac{\eta}{k^3}\pi (\|\theta^{\star}\|_X + \|G^{\star}\|_{X})
  \end{align*}
  and
  \begin{align*}
    \|G(t)\|_X \leq 2 c \frac{\eta}{k^3}\pi ( \|\theta^{\star}\|_X + \|G^{\star}\|_{X}).
  \end{align*}
  for all $t_k \leq t \leq t_{k-1}$.

  Moreover,  the following mode-wise bounds hold:
  \begin{align*}
    |\theta_{k\pm 1}(t_{k-1}) - \theta_{k \pm 1}(t_k) - c\frac{\eta}{k^3}\pi \theta_{k}(t_k)|  &\leq \frac{1}{k} c \frac{\eta}{k^3}\pi (\|\theta^{\star}\|_X + k \|G^{\star}\|_{X}), \\
    |\theta_{k}(t_{k-1})-\theta_k(t_k)| &\leq c (\frac{\eta}{k^2})^{-1}\pi (\|\theta^{\star}\|_X + k \|G^{\star}\|_{X})
  \end{align*}
  and for all $l \not \in \{k-1,k,k+1\}$ it holds that
  \begin{align*}
    |\theta_l(t_{k-1}) - \theta_l(t_k)| \leq c (\frac{\eta}{k^2})^{-1} \pi (\|\theta^{\star}\|_X + \|G^{\star}\|_{X}).
  \end{align*}
\end{theorem}

We emphasize that the present theorem only provides an upper bound on growth.
In Section \ref{sec:blow-up} we will show that there indeed exists data
saturating this growth (up to a factor).
In particular, we construct global in time solutions exhibiting norm inflation
due to \emph{echo chains}. Using these solutions as building blocks we then
construct a critical class of initial data exhibiting blow-up in Sobolev regularity.

The proof of Theorem \ref{theorem:mechanism} concludes in Subsection \ref{sec:mechanism} and builds on multiple steps which are formulated as
propositions and lemmas.
Unlike the Lyapunov energy approach of Sections \ref{sec:small} and \ref{sec:intermediate} we here
iteratively construct solutions in weighted $\ell^\infty$ spaces, which we will
then use to deduce analogous estimates in the space $X$.

Before stating these estimates we introduce an integral formulation of the
equations and collect estimates on time integrals of the coefficient functions.
The interplay of these estimates then determines admissible weights.

The equations \eqref{eq:simplewave2} in Fourier variables read
\begin{align}
  \label{eq:thetaeq}
  \begin{split}
  \dt \theta_l = \frac{g}{2\nu} \frac{\eta}{(l+1)^3} \frac{1}{(1+(\frac{\eta}{l+1}-t)^2)^2} \theta_{l+1}
  + \frac{g}{2\nu} \frac{\eta}{(l-1)^3} \frac{1}{(1+(\frac{\eta}{l-1}-t)^2)^2} \theta_{l-1}\\
  + \frac{g}{2\nu} \frac{\eta}{(l+1)^3} \frac{1}{1+(\frac{\eta}{l+1}-t)^2} G_{l+1}\\
  + \frac{g}{2\nu} \frac{\eta}{(l-1)^3} \frac{1}{1+(\frac{\eta}{l-1}-t)^2} G_{l-1}\\
  =: c_{l}^{+} \theta_{l+1} + c_{l}^{-} \theta_{l-1} + d_{l}^{+} G_{l+1} + d_{l}^{-} G_{l-1}.
  \end{split}
\end{align}
and
\begin{align}
  \label{eq:Geq}
  \begin{split}
    \dt G_{l} &= -\nu (l^2+(\eta-lt)^2) G_l + f\frac{\nu}{g} il (d_l^{+} G_{l+1} + d_l^{-1} G_{l-1})\\
    &+ f \frac{\nu}{g}il (c_{l}^+ \theta_{l+1} + c_{l}^{-1}\theta_{l-1})\\
    &+ 2 \frac{(\frac{\eta}{l}-t)}{(1+(\frac{\eta}{l}-t)^2)^2} \theta_{l}\\
  &+ \frac{1}{1+(\frac{\eta}{l}-t)^2}\left( c_{l}^{+} \theta_{l+1} + c_{l}^{-} \theta_{l-1}\right) \\
  &+ \frac{1}{1+(\frac{\eta}{l}-t)^2} \left(d_{l}^{+} G_{l+1} + d_{l}^{-} G_{l-1} \right),
  \end{split}
\end{align}
where we use $G_{l}$ and $\theta_l$ to denote the Fourier coefficients.
We recall that, as a simplification, throughout this article we assume that the
$x$-averages $\theta_{0}, G_0$ identically vanish. Hence, these equations should
be interpreted as being valid for $l \in \Z \setminus\{0\}$ and all terms
involving modes $G_0, \theta_0$ are trivial.

These differential equations are  equivalent to the following integral equations:
\begin{align}
\label{eq:integralequations}  
  \begin{split}
   \theta_l(T)- \theta_l(t_k) &= \int_{t_k}^T c_{l}^{+} \theta_{l+1} + c_{l}^{-} \theta_{l-1} + d_{l}^{+} G_{l+1} + d_{l}^{-} G_{l-1} dt, \\
   G_{l}(T)- &\exp\left(-\nu\int_{t_k}^{T}l^2+(\eta-lt)^2 dt\right) G_{l}(t_k)\\
  &= \int_{t_k}^{T}\exp(-\nu\int_{t}^{T}l^2+(\eta-l\tau)^2 d\tau) 2 \frac{(\frac{\eta}{l}-t)}{(1+(\frac{\eta}{l}-t)^2)^2} \theta_{l} dt \\
  &\quad + \int_{t_k}^{T}\exp(-\nu\int_{t}^{T}l^2+(\eta-l\tau)^2 d\tau) f il \frac{\nu}{g} (d_l^{+} G_{l+1} + d_l^{-} G_{l-1}) dt \\
  &\quad + \int_{t_k}^{T}\exp(-\nu\int_{t}^{T}l^2+(\eta-l\tau)^2 d\tau) f il \frac{\nu}{g} (c_l^{+} \theta_{l+1} + c_l^{-} \theta_{l-1}) dt \\
  & \quad + \int_{t_k}^{T}\exp(-\nu\int_{t}^{T}l^2+(\eta-l\tau)^2 d\tau) \frac{1}{1+(\frac{\eta}{l}-t)^2}\left( c_{l}^{+} \theta_{l+1} + c_{l}^{-} \theta_{l-1}\right) dt \\
  & \quad + \int_{t_k}^{T}\exp(-\nu\int_{t}^{T}l^2+(\eta-l\tau)^2 d\tau) \frac{1}{1+(\frac{\eta}{l}-t)^2} \left(d_{l}^{+} G_{l+1} + d_{l}^{-} G_{l-1} \right) dt.
  \end{split}
\end{align}
In the following we will show by means of a bootstrap argument that the modes $G_{l}, \theta_{l}$ when
adjusted with a suitable weight function remain bounded uniformly in time.

For easier reference and to motivate our choice of weight function the $L^1_t$ estimates on the coefficients required for the
control of $\theta$ are collected in the following lemma.
\begin{lemma}
  \label{lemma:coeffestimates}
  Let $k, \eta, c$  be as in Theorem \ref{theorem:mechanism} and let $t \in I_k$.
  Let further $c_l^{\pm}, d_{l}^{\pm}$ be defined by equation \eqref{eq:thetaeq}:
  \begin{align*}
    c_{l}^{+}&= c \frac{\eta}{(l+1)^3} \frac{1}{(1+(\frac{\eta}{l+1}-t)^2)^2},\\
    c_{l}^{-}&= c \frac{\eta}{(l-1)^3} \frac{1}{(1+(\frac{\eta}{l-1}-t)^2)^2},\\
    d_{l}^{+}&= c \frac{\eta}{(l+1)^3} \frac{1}{1+(\frac{\eta}{l+1}-t)^2},\\
    d_{l}^{-}&= c \frac{\eta}{(l-1)^3} \frac{1}{1+(\frac{\eta}{l-1}-t)^2}.
  \end{align*}
  Then it holds that
  \begin{align}
    \label{eq:resonantcoeffs}
    \begin{split}
    \int c_{k\pm 1}^{\mp} &= c \frac{\eta}{k^3}\int \frac{1}{(1+(\frac{\eta}{k}-t)^2)^2} ,\\
    \int d_{k\pm 1}^{\mp} &= c \frac{\eta}{k^3}\int \frac{1}{1+(\frac{\eta}{k}-t)^2},
  \end{split}
  \end{align}
  and
  \begin{align*}
    \int_{\R} \frac{1}{(1+(\frac{\eta}{k}-t)^2)^2}  = \frac{\pi}{2}, \\
    \int_{\R} \frac{1}{1+(\frac{\eta}{k}-t)^2}= \pi.
  \end{align*}
  Since $c \frac{\eta}{k^3}$ is possibly very large,
  we call these the \emph{resonant} cases.

  For the remaining \emph{non-resonant} cases, the following bounds hold:
  \begin{align}
    \label{eq:nonresonantcoeffs}
    \begin{split}
      \int_{I_k} c_{l}^{\pm} &\leq \frac{4c}{k} (\frac{\eta}{k^2})^{-2},\\
    \int_{I_k} d_{l}^{\pm} &\leq  \frac{4c}{k}.
    \end{split}
  \end{align}
\end{lemma}
We stress that unlike in Section \ref{sec:small} here $c\frac{\eta}{k^3}$ can be very large.
Hence, the integrals in \eqref{eq:resonantcoeffs} can be very large.
On the other hand the integrals in \eqref{eq:nonresonantcoeffs} are quite small
and the integral over $c_{l}^{\pm}$ becomes even smaller the larger
$\frac{\eta}{k^2}$ is.
As we discuss after the following lemma, these large resonant coefficients and
small non-resonant coefficients determine the structure of our choice of weight function.

The estimates required to control the evolution of $G$ are also collected in a lemma.
\begin{lemma}
  \label{lemma:coefficientsForG}
  Let $k,\eta,c$ be as in Theorem \ref{theorem:mechanism} and let $t \in I_k$
  and consider the integrals stated in \eqref{eq:integralequations}.

  Then for the \emph{resonant case} $l=k$ it holds that 
 \begin{align}
   \label{eq:Gthetares}
  \int_{t_k}^{T}\exp(-\nu\int_{t}^{T}k^2+(\eta-k\tau)^2 d\tau) 2 \frac{(\frac{\eta}{k}-t)}{(1+(\frac{\eta}{k}-t)^2)^2} \leq 2.
 \end{align}
 For the \emph{non-resonant} cases $l\neq k$ we instead estimate
 \begin{align}
   \label{eq:Gthetanonres}
    \int_{t_k}^{T}\exp(-\nu\int_{t}^{T}l^2+(\eta-l\tau)^2 d\tau) \left| 2 \frac{(\frac{\eta}{l}-t)}{(1+(\frac{\eta}{l}-t)^2)^2} \right| dt\leq 2 (\frac{\eta}{k^2})^{-2}.
 \end{align}

 For the coefficient functions involving $f$ we control: 
 \begin{align}
   \label{eq:Gf}
   \begin{split}
     \int_{t_k}^t f \frac{\nu}{g}il c_{l}^\pm  &\leq
     \begin{cases}
       c (\frac{\eta}{k^2})^{-1} &, \text{ if } l=k\mp 1,\\
       c (\frac{\eta}{k^2})^{-4} &, \text{ else}.
     \end{cases}
     ,\\
     \int_{t_k}^t f \frac{\nu}{g} il d_{l}^\pm &\leq \begin{cases}
       c (\frac{\eta}{k^2})^{-1} &, \text{ if } l=k\mp 1,\\
       c (\frac{\eta}{k^2})^{-2} &, \text{ else}.
     \end{cases}. 
   \end{split}
 \end{align}

 We further estimate
 \begin{align}
   \label{eq:Gthetac}
   \int_{t_k}^{T}\exp(-\nu\int_{t}^{T}l^2+(\eta-l\tau)^2 d\tau) \frac{1}{1+(\frac{\eta}{l}-t)^2}c_{l}^{\pm} \leq
   \begin{cases}
     32 \frac{c}{k} (\frac{\eta}{k^2})^{-4}, & \text{ if } l\neq k \neq l\pm 1,\\
     \frac{c}{k} (\frac{\eta}{k^2})^{-3} 16 \pi, & \text{ if } l=k, \\
    \frac{c}{k} (\frac{\eta}{k^2})^{-1} \frac{\pi}{2}, & \text{ if } l=k \pm 1.
   \end{cases}
 \end{align}
Finally, we control
 \begin{align}
   \label{eq:GGd}
   \int_{t_k}^{T}\exp(-\nu\int_{t}^{T}l^2+(\eta-l\tau)^2 d\tau)\frac{1}{1+(\frac{\eta}{l}-t)^2} d_{l}^{\pm} \leq
      \begin{cases}
     \frac{c}{k} (\frac{\eta}{k^2})^{-2}, & \text{ if } l\neq k \neq l \pm 1,\\
     \frac{c}{k} (\frac{\eta}{k^2})^{-1} \pi, & \text{else}.
   \end{cases}
 \end{align}
\end{lemma}
We stress that the contributions by resonant frequencies as stated in
\eqref{eq:Gthetares} while not large in an absolute sense, are not small.
In contrast all other coefficients provide a gain of negative powers of
$(\frac{\eta}{k^2})$.
We postpone the proof of the coefficient estimates formulated in Lemmas
\ref{lemma:coefficientsForG} and \ref{lemma:coeffestimates} to Subsection
\ref{sec:coeffestimates}.

The estimates of Lemmas \ref{lemma:coeffestimates} and
\ref{lemma:coefficientsForG} suggest that if initially 
\begin{align*}
  \theta_{l}(t_k)= \delta_{lk}, \ G_l(t_k)=0,
\end{align*}
then for $t_k< t< t_{k-1}$ one should expect the following \emph{heuristic} bounds:
\begin{align}
  \label{eq:expectedweights}
  \theta_k(T)& \approx 1, \quad &G_k(T)&\leq 1, \notag \\
  \theta_{k-1}(T)& \approx c \frac{\eta}{k^3}, \quad &G_{k-1}(T) & \leq \frac{c}{k} (\frac{\eta}{k^2})^{-1}, \\
  \theta_{k-1-j} (T) & \leq c \frac{\eta}{k^3} (c(\frac{\eta}{k^2})^{-2})^j , \quad &G_{k-1-j}(T) & \leq \frac{c}{k} (\frac{\eta}{k^2})^{-1}(c(\frac{\eta}{k^2})^{-2})^j, \notag
\end{align}
for $j \in \N$ and analogously for $\theta_l, G_l$ with $l>k$.

The following Proposition \ref{prop:res} proves that this heuristic is indeed
valid and establishes error bounds on the heuristic approximations.
The cases of initial data concentrated on a different mode
\begin{align*}
\theta_l(t_k)=\delta_{k_0 l}, G_l(t_k)=0,
\end{align*}
or
\begin{align*}
\theta_l(t_k)=0, G_l(t_k)=  \delta_{k_0 l},
\end{align*}
are considered in Proposition \ref{prop:nonres}.
In Subsection \ref{sec:mechanism} we then show how to pass from weighted
$\ell^\infty$ estimates to bounds in $X$ and thus establish Theorem
\ref{theorem:mechanism}.

\begin{prop}
  \label{prop:res}
  Let $\eta, c, k$ be as in Theorem \ref{theorem:mechanism}.
  Consider the evolution equation \eqref{eq:integralequations} with 
  \begin{align*}
    \theta_l(t_k)=\delta_{lk}, G_l(t_k)= 0
  \end{align*}
  for all $l\in \Z$.
  Let further $c_{l}^\pm$, $d_{l}^{\pm}$ be defined as in Lemma \ref{lemma:coeffestimates}.
  
  Then the following estimates hold for all times $t_k \leq T \leq t_{k-1}$:
  \begin{enumerate}[start=1, label={(B\arabic*)}]
  \item \label{item:B1} For the resonant mode $\theta_k$ it holds that
    \[|\theta_k(T)-1| \leq \frac{10 c}{k} (\frac{\eta}{k^2})^{-1}.\] 
  \item \label{item:B2} For $l \in \{k-1,k+1\}$ it holds that
    \[\left| \theta_l(T)- \int_{t_k}^Tc_{l}^{\mp} dt - \int_{t_k}^T d_{l}^{\mp}
       G_k(t)dt \right|\leq \frac{0.5}{k} c \frac{\eta}{k^3}.\]
  \item\label{item:B3} For all $l\not \in \{k-1,k,k+1\}$ it holds that
    \[|\theta_l(T)| \leq c \frac{\eta}{k^3} (c (\frac{\eta}{k^2})^{-2})^{|l-k|+1}.\]
  \item \label{item:B4} For the mode $G_k$ it holds that
    \[|G_k(T)-  \int_{t_k}^t \exp(-\nu\int_{t_k}^t k^2 +(\eta-ks)^2 ds)
      2\frac{(\frac{\eta}{k}-t)}{(1+(\frac{\eta}{k}-t)^2)^2} dt| \leq
      \frac{2}{k}.\]
  \item\label{item:B5} For all $l\neq k$ it holds that
    \[|G_l(T)| \leq  \frac{\eta}{k^3} (c (\frac{\eta}{k^2})^{-2})^{|l-k|}.\]
  \end{enumerate}
  If we instead consider $\theta_l(t_k)=0$, $G_{l}(t_k)=
  \delta_{lk}$, then \ref{item:B3} and \ref{item:B5} still hold and
  \ref{item:B1}, \ref{item:B2}, \ref{item:B4} are replaced by
  \begin{itemize}
  \item [(B1')] It holds that
    \[|\theta_k(T)|\leq \frac{c}{k}.\]
  \item [(B2')] It holds that
    \[\left| \theta_l(T) - \int_{t_k}^T d_{l}^{\mp}
        G_k(t)dt \right|\leq \frac{0.5}{k} c \frac{\eta}{k^3}.\]
  \item [(B4')] For the mode $G_k$ it holds that
    \[|G_k(T)- \exp(-\nu\int_{t_k}^t k^2 +(\eta-ks)^2 ds)| \leq \frac{2}{k}.\]
  \end{itemize}
\end{prop}
These estimates quantify our heuristic estimates \eqref{eq:expectedweights}. In
particular, in addition to an upper bound, \ref{item:B2} also provides a lower
bound. This norm inflation mechanism then will form the core of our echo chain
construction of Section \ref{sec:proofofmain}.

\begin{proof}[Proof of Proposition \ref{prop:res}]
  In our proof we use \ref{item:B1}--\ref{item:B5} as bootstrap estimates.
  More precisely, our strategy is the following:
  \begin{itemize}
  \item We first show that these estimates hold at least on a small
    time interval $(t_k, T_{\star})$.
  \item Subsequently, we prove that on that time interval the estimates
  self-improve. That is, variants of \ref{item:B1}--\ref{item:B5} hold where the
  right-hand-side is improved by a factor $0.9$.
  \item Choosing $T_{\star}$ maximal with the property that
    \ref{item:B1}--\ref{item:B5} are satisfied, it follows that equality in
    these estimates is not attained. However, by local in time arguments this
    implies that if $T_{\star}$ were smaller than $t_{k-1}$, the estimates would
    remain valid at least for a small additional time.
    Since this contradicts the maximality of $T_{\star}$ it follows that
    $T_{\star}=t_{k-1}$, which concludes the proof.
  \end{itemize}
  
\underline{Establishing the initial bootstrap:}
We note that the right-hand-side estimates in \ref{item:B1}--\ref{item:B5}
include a power law $(c (\frac{\eta}{k^2})^{-2})^{|l-k|}$, which is not bounded below and thus at first
sight might seem problematic for local in time continuity results.
We thus instead consider the equivalent unknowns
\begin{align}
  \label{eq:equivunknowns}
  (c (\frac{\eta}{k^2})^{-2})^{|l-k|} \theta_l, (c (\frac{\eta}{k^2})^{-2})^{|l-k|}G_l,
\end{align}
where we note that at time $t_k$, these unknowns equal $\delta_{lk}$ and $0$, respectively.

We further observe that the equations \eqref{eq:integralequations} only include nearest
neighbor interactions and that the quotients
\begin{align*}
 \frac{(c (\frac{\eta}{k^2})^{-2})^{|l \pm 1 -k|}}{(c (\frac{\eta}{k^2})^{-2})^{|l-k|}}
\end{align*}
are bounded above and below.
Thus expressing the integral equations \eqref{eq:integralequations} with respect
to the unknowns \eqref{eq:equivunknowns} we observe that all integrands are
bounded uniformly in time (but may depend on $c$ and $\eta$).
Thus, choosing $T_{\star}$ such that $|T_{\star}-t_k|$ is sufficiently small, we
obtain a contraction mapping in $\ell^\infty$ and at least for a small time it
holds that
\begin{align*}
  \|(c (\frac{\eta}{k^2})^{-2})^{|l-k|} \theta_l -\delta_{lk} \|_{\ell^\infty} + \|(c (\frac{\eta}{k^2})^{-2})^{|l-k|} G_l\|_{\ell^{\infty}}
\end{align*}
is sufficiently small that \ref{item:B1}--\ref{item:B5} are satisfied.

Having established our initial bootstrap estimates, we now let $T_{\star}\leq t_{k-1}$ be
the maximal time such that these estimates hold. We then show that all bootstrap
estimates improve by a factor and thus $T_{\star}$ can only have been maximal if $T_{\star}=t_{k-1}$.

\underline{Improving \ref{item:B1}:}
Let $t_{k}\leq T \leq T_{\star}$ and consider the integral equation
\eqref{eq:integralequations} for $\theta_k(T)$.
Then by our choice of initial data it holds that
\begin{align*}
  \theta_k(T)- 1= \int_{t_k}^T c_{k}^{+} \theta_{k+1} + c_{k}^{-} \theta_{k-1} + d_{k}^{+} G_{k+1} + d_{k}^{-} G_{k-1} dt.
\end{align*}
Since $c_{k}^{\pm}$ and $d_{k}^{\pm}$ are non-resonant, by the estimates
\eqref{eq:nonresonantcoeffs} of Lemma \ref{lemma:coeffestimates} it follows that
\begin{align*}
  |\theta_k(T)-1| \leq \frac{4c}{k} (\frac{\eta}{k^2})^{-2} (\|\theta_{k+1}\|_{L^\infty} + \|\theta_{k-1}\|_{L^\infty})
  + \frac{4 c}{k} (\|G_{k+1}\|_{L^\infty} + \|G_{k-1}\|_{L^\infty}).
\end{align*}
By the bootstrap estimate \ref{item:B5} it follows that
\begin{align*}
  \frac{4 c}{k} (\|G_{k+1}\|_{L^\infty} + \|G_{k-1}\|_{L^\infty}) \leq \frac{8c}{k^2} (\frac{\eta}{k^2})^{-1},
\end{align*}
and by the bootstrap estimate \ref{item:B2}:
\begin{align*}
  & \quad \frac{4c}{k} (\frac{\eta}{k^2})^{-2} (\|\theta_{k+1}\|_{L^\infty} + \|\theta_{k-1}\|_{L^{\infty}}) \\
  & \leq \frac{4c}{k} (\frac{\eta}{k^2})^{-2} c \frac{\eta}{k^3}\pi \leq \frac{4 c^2 \pi}{k^2} (\frac{\eta}{k^2})^{-1}.
\end{align*}
Thus combining both estimates we obtain
\begin{align*}
  |\theta_k(T)-1|  \leq \frac{8c}{k} (\frac{\eta}{k^2})^{-1} +  \frac{4 c^2 \pi}{k^2} (\frac{\eta}{k^2})^{-1} < \frac{10c}{k} (\frac{\eta}{k^2})^{-1}.
\end{align*}

\underline{Improving \ref{item:B2}:}
We next consider $\theta_{l}$ for $l\in \{k-1,k+1\}$, for which by the integral equation \eqref{eq:integralequations} it holds that
\begin{align*}
  \theta_{k \pm 1}(T)&= \int_{t_k}^T c_{k\pm 1}^{\mp} (\theta_k -1 +1) + \int_{t_{k}}^T d_{k\pm 1}^{\mp} G_k \\
  & \quad + \int_{t_k}^T c_{k\pm 1}^{\pm} \theta_{k\pm 2} + \int_{t_k}^T d_{k\pm 1}^{\pm} G_{k\pm 2}.
\end{align*}
We then subtract the contributions by $\theta_k(t_k)=1$ from both sides and
control:
\begin{align*}
  \int_{t_k}^T c_{k\pm 1}^{\mp} (\theta_k -1) &\underset{\eqref{eq:resonantcoeffs},\ref{item:B1}}{\leq} c \frac{\eta}{k^3} \pi \|\theta_k-1\|_{L^\infty}\leq \frac{10c^2}{k^2} \pi,\\
  \int_{t_{k}}^T d_{k\pm 1}^{\mp} (G_k -\hat{G}_k) &\underset{\eqref{eq:resonantcoeffs},\ref{item:B4}}{\leq} c \frac{\eta}{k^3}\pi \|G_k-\hat{G}_k\|\leq \frac{1}{k} c \frac{\eta}{k^3}\pi\\
  \int_{t_k}^T c_{k\pm 1}^{\pm} \theta_{k\pm 2} &\underset{\eqref{eq:nonresonantcoeffs},\ref{item:B3}}{\leq} \frac{4c}{k} (\frac{\eta}{k^2})^{-2} c\frac{\eta}{k^3} (c (\frac{\eta}{k^2})^{-2})\\
  \int_{t_k}^T d_{k\pm 1}^{\pm} G_{k\pm 2}&\underset{\eqref{eq:nonresonantcoeffs},\ref{item:B5}}{\leq} c \frac{1}{k} (\frac{\eta}{k^2}) (c (\frac{\eta}{k^2})^{-2})^2,
\end{align*}
where
\begin{align*}
  \hat{G_k}=  \int_{t_k}^t \exp(-\nu\int_{t_k}^t k^2 +(\eta-ks)^2 ds)2 \frac{(\frac{\eta}{k}-t)}{(1+(\frac{\eta}{k}-t)^2)^2} dt
\end{align*}
is the contribution to $G_k$ by $\theta_k(t_k)=1$.

Here the contributions by $\theta_{k\pm 2}$, $G_{k \pm 2}$ are much
smaller than $\theta_{k\pm 1}$ by the bootstrap assumption and moreover the
coefficient functions are small.
The size of $\theta_{k\pm 1}$ is thus mostly determined by the modes
$\theta_{k}$ and $G_k$, which form the core of the resonance mechanism.

The claimed estimate thus follows.

\underline{Improving \ref{item:B3}:}
For $l \not \in \{k-1,k,k+1\}$ the integral equations
\eqref{eq:integralequations} read
\begin{align*}
  \theta_l(T) = \int_{t_k}^T c_{l}^{+} \theta_{l+1} + c_{l}^{-} \theta_{l-1} + d_{l}^{+} G_{l+1} + d_{l}^{-} G_{l-1} dt.
\end{align*}
By our choice of $l$ all coefficient functions $c_{l}^\pm, d_{l}^{\pm}$ are
controlled by the estimate \eqref{eq:nonresonantcoeffs} of Lemma
\ref{lemma:coeffestimates}.

Combining these estimates, we deduce that
\begin{align*}
  |\theta_{l}(T)|& \leq 4c \frac{1}{k} (\frac{\eta}{k^2})^{-2} (\|\theta_{l+1}\|_{L^\infty} + \|\theta_{l-1}\|_{L^\infty})\\
                 & \quad + \frac{c}{k} (\|G_{l+1}\|_{L^\infty} + \|G_{l-1}\|_{L^\infty})                   
\end{align*}
As in the estimate of \ref{item:B2} we here observe that by the bootstrap
assumptions the nodes $l\pm 1$ closer to $k$ are potentially much larger and
thus determine the achievable upper bound.
In particular, we gain a factor $ \frac{4}{k} c(\frac{\eta}{k^2})^{-2}$ with
respect to the neighbors in $\theta$ and a factor $\frac{c}{k}$ with respect to the
neighboring modes of $G$, which by the bootstrap assumptions themselves satisfy bounds with an additional power
$(\frac{\eta}{k^2})^{-2}$.

Hence, using the bootstrap estimates \ref{item:B2}, \ref{item:B3} and
\ref{item:B5} we deduce that
\begin{align*}
  |\theta_{l}(T)| &\leq \frac{1}{k} c(\frac{\eta}{k^2})^{-2 + \min(|l-1-k|, |l+1-k|)+1} \frac{0.5}{k} c\frac{\eta}{k^3} \\
                  & \quad + c \frac{1}{k} (\frac{\eta}{k^2}) (c (\frac{\eta}{k^2})^{-2})^{\min(|l+1-k|, |l-1-k|}\\
  & < \frac{0.5}{k} c \frac{\eta}{k^3} (c (\frac{\eta}{k^2})^{-2})^{|l-k|+1}.
\end{align*}

\underline{Improving \ref{item:B4}:}
We next turn to studying the evolution of $G_k$:
\begin{align*}
  G_k(T)=
  &= \int_{t_k}^{T}\exp(-\nu\int_{t}^{T}k^2+(\eta-k\tau)^2 d\tau) \frac{2}{k} \frac{(\frac{\eta}{k}-t)}{(1+(\frac{\eta}{k}-t)^2)^2} (\theta_{k}-1 +1) dt \\
  &\quad + \int_{t_k}^{T}\exp(-\nu\int_{t}^{T}k^2+(\eta-k\tau)^2 d\tau) f \frac{\nu}{g} ik(d_k^{+} G_{k+1} + d_k^{-} G_{k-1}) dt \\
  &\quad + \int_{t_k}^{T}\exp(-\nu\int_{t}^{T}k^2+(\eta-k\tau)^2 d\tau) f \frac{\nu}{g} ik(c_k^{+} \theta_{k+1} + c_k^{-} \theta_{k-1}) dt \\
  & \quad + \int_{t_k}^{T}\exp(-\nu\int_{t}^{T}k^2+(\eta-k\tau)^2 d\tau) \frac{1}{1+(\frac{\eta}{k}-t)^2}\left( c_{k}^{+} \theta_{k+1} + c_{k}^{-} \theta_{k-1}\right) dt \\
  & \quad + \int_{t_k}^{T}\exp(-\nu\int_{t}^{T}k^2+(\eta-k\tau)^2 d\tau) \frac{1}{1+(\frac{\eta}{k}-t)^2} \left(d_{k}^{+} G_{k+1} + d_{k}^{-} G_{k-1} \right) dt.
\end{align*}
We then subtract the contribution by $\theta_k(t_k)=1$ from both sides and
estimate
\begin{align*}
  \int_{t_k}^{T}\exp(-\nu\int_{t}^{T}k^2+(\eta-k\tau)^2 d\tau) \frac{2}{k} \frac{(\frac{\eta}{k}-t)}{(1+(\frac{\eta}{k}-t)^2)^2} (\theta_{k}-1) & \underset{\eqref{eq:Gthetares},\ref{item:B1}}{\leq} \frac{2}{k}\frac{0.5}{k}= \frac{1}{k^2}, \\
  \int_{t_k}^{T}\exp(-\nu\int_{t}^{T}k^2+(\eta-k\tau)^2 d\tau) f \frac{\nu}{g} ik(d_k^{+} G_{k+1} + d_k^{-} G_{k-1}) dt&\underset{\eqref{eq:Gf},\ref{item:B5}}{\leq}
c (\frac{\eta}{k})^{-2} \frac{1}{k} (\frac{\eta}{k^2})^{-1}, \\
  \int_{t_k}^{T}\exp(-\nu\int_{t}^{T}k^2+(\eta-k\tau)^2 d\tau) f \frac{\nu}{g} ik (c_k^{+} \theta_{k+1} + c_k^{-} \theta_{k-1})& \underset{\eqref{eq:Gf}, \ref{item:B2}}{\leq}c (\frac{\eta}{k})^{-4} c \frac{\eta}{k^3}, \\
  \int_{t_k}^{T}\exp(-\nu\int_{t}^{T}k^2+(\eta-k\tau)^2 d\tau) \frac{1}{1+(\frac{\eta}{k}-t)^2}\left( c_{k}^{+} \theta_{k+1} + c_{k}^{-} \theta_{k-1}\right) &\underset{\eqref{eq:Gthetac}, \ref{item:B2}}{\leq}
  \frac{c}{k} (\frac{\eta}{k^2})^{-3} 16 \pi c\frac{\eta}{k^3},\\
  \int_{t_k}^{T}\exp(-\nu\int_{t}^{T}k^2+(\eta-k\tau)^2 d\tau) \frac{1}{1+(\frac{\eta}{k}-t)^2}
  \left(d_{k}^{+} G_{k+1} + d_{k}^{-} G_{k-1} \right) &\underset{\eqref{eq:GGd},\ref{item:B5}}{\leq} \frac{c}{k}(\frac{\eta}{k^2})^{-2} \pi  \frac{c}{k} (\frac{\eta}{k^2})^{-1}.
\end{align*}
We note that here the largest contribution is given by the first line and that
hence estimate \ref{item:B4} improves.

\underline{Improving \ref{item:B5}:}
Finally, we consider $G_l(T)$ for $l\neq k$, which satisfies:
\begin{align*}
  G_{l}(T)&= \int_{t_k}^{T}\exp(-\nu\int_{t}^{T}l^2+(\eta-l\tau)^2 d\tau) \frac{2}{l} \frac{(\frac{\eta}{l}-t)}{(1+(\frac{\eta}{l}-t)^2)^2} \theta_{l} dt \\
            &\quad + \int_{t_k}^{T}\exp(-\nu\int_{t}^{T}l^2+(\eta-l\tau)^2 d\tau) f \frac{\nu}{g} il (d_l^{+} G_{l+1} + d_l^{-} G_{l-1}) dt \\
            &\quad + \int_{t_k}^{T}\exp(-\nu\int_{t}^{T}l^2+(\eta-l\tau)^2 d\tau) f \frac{\nu}{g} il(c_l^{+} \theta_{l+1} + c_l^{-} \theta_{l-1}) dt \\
            & \quad + \int_{t_k}^{T}\exp(-\nu\int_{t}^{T}l^2+(\eta-l\tau)^2 d\tau) \frac{1}{1+(\frac{\eta}{l}-t)^2}\left( c_{l}^{+} \theta_{l+1} + c_{l}^{-} \theta_{l-1}\right) dt \\
            & \quad + \int_{t_k}^{T}\exp(-\nu\int_{t}^{T}l^2+(\eta-l\tau)^2 d\tau) \frac{1}{1+(\frac{\eta}{l}-t)^2} \left(d_{l}^{+} G_{l+1} + d_{l}^{-} G_{l-1} \right) dt.
\end{align*}
We remark that if, for instance, $c\leq \nu$ then all coefficient functions on
the right-hand-side can easily be dominated by the exponential decay and using
that $(\eta-lt)^2$ is bounded below, since $l\neq k$.
A key effort of this proof hence lies in establishing estimates that are valid
when $\nu < c$ as well.

Since $l\neq k$ is non-resonant, we may control
\begin{align*}
  & \quad \int_{t_k}^{T}\exp(-\nu\int_{t}^{T}l^2+(\eta-l\tau)^2 d\tau) \frac{2}{l} \frac{(\frac{\eta}{l}-t)}{(1+(\frac{\eta}{l}-t)^2)^2} \theta_{l}\\
  &\underset{\eqref{eq:Gthetanonres}, \ref{item:B2}, \ref{item:B3}}{\leq}
  \frac{2}{k} (\frac{\eta}{k^2})^{-2} c \frac{\eta}{k^3} (c (\frac{\eta}{k^2})^{-2})^{|l-k|+1}.
\end{align*}

For the remaining estimates, we first study the case $l \not \in \{k-1,k,k+1\}$,
for which
\begin{align*}
  &\quad \int_{t_k}^{T}\exp(-\nu\int_{t}^{T}l^2+(\eta-l\tau)^2 d\tau) f \frac{\nu}{g} il(d_l^{+} G_{l+1} + d_l^{-} G_{l-1}) \\
  &\underset{\eqref{eq:Gf}, \ref{item:B5}}{\leq}
  c (\frac{\eta}{k^2})^{-2} (c(\frac{\eta}{k^2})^{-2})^{|l-k|} , \\
  &\quad \int_{t_k}^{T}\exp(-\nu\int_{t}^{T}l^2+(\eta-l\tau)^2 d\tau) f \frac{\nu}{g} il(c_l^{+} \theta_{l+1} + c_l^{-} \theta_{l-1})\\
  &\underset{\eqref{eq:Gf}, \ref{item:B3}}{\leq}
   c (\frac{\eta}{k^2})^{-4} c \frac{\eta}{k^3} (c(\frac{\eta}{k^2})^{-2})^{|l-k|}\\
  &\quad \int_{t_k}^{T}\exp(-\nu\int_{t}^{T}l^2+(\eta-l\tau)^2 d\tau) \frac{1}{1+(\frac{\eta}{l}-t)^2}\left( c_{l}^{+} \theta_{l+1} + c_{l}^{-} \theta_{l-1}\right)\\
  &\underset{\eqref{eq:Gthetac}, \ref{item:B3}}{\leq}
   32 \frac{c}{k} (\frac{\eta}{k^2})^{-4}  c \frac{\eta}{k^3} (c(\frac{\eta}{k^2})^{-2})^{|l-k|-1}\\
  &\quad \int_{t_k}^{T}\exp(-\nu\int_{t}^{T}l^2+(\eta-l\tau)^2 d\tau) \frac{1}{1+(\frac{\eta}{l}-t)^2} \left(d_{l}^{+} G_{l+1} + d_{l}^{-} G_{l-1} \right)\\
  &\underset{\eqref{eq:GGd}, \ref{item:B5}}{\leq}
  \frac{c}{k} (\frac{\eta}{k^2})^{-1} \frac{1}{k} (c(\frac{\eta}{k^2})^{-2})^{|l-k|-1}.
\end{align*}
In particular, all estimates indeed yield an improvement by a factor
$c(\frac{\eta}{k^2})^{-2}$ compared to its neighbors and thus this bootstrap
estimate is improved.

Finally, we discuss the case $l \in \{k-1,k+1\}$.
Here the above estimates also apply to the contributions due to $\theta_{k\pm
  1}$, $\theta_{k\pm 2}$ and $G_{k\pm 2}$.
For the contributions by $\theta_k$ and $G_k$ we instead establish the following
estimates:
\begin{align*}
  \int_{t_k}^{T}\exp(-\nu\int_{t}^{T}l^2+(\eta-l\tau)^2 d\tau) f \frac{\nu}{g} ild_{l}^{\mp} G_k
  &\underset{\eqref{eq:Gf}, \ref{item:B4}}{\leq}
  c (\frac{\eta}{k^2})^{-1} \frac{2}{k}, \\
  \int_{t_k}^{T}\exp(-\nu\int_{t}^{T}l^2+(\eta-l\tau)^2 d\tau) f \frac{\nu}{g} ilc_{l}^{\mp} \theta_k
  &\underset{\eqref{eq:Gf}, \ref{item:B1}}{\leq}
  \frac{c}{k} (\frac{\eta}{k^2})^{-1} 2, \\
  \int_{t_k}^{T}\exp(-\nu\int_{t}^{T}l^2+(\eta-l\tau)^2 d\tau) \frac{1}{1+(\frac{\eta}{l}-t)^2} c_{l}^{\mp} \theta_k
  &\underset{\eqref{eq:Gthetac}, \ref{item:B1}}{\leq}
  \frac{c}{k^2} (\frac{\eta}{k^2})^{-1} \pi, \\
  \int_{t_k}^{T}\exp(-\nu\int_{t}^{T}l^2+(\eta-l\tau)^2 d\tau) \frac{1}{1+(\frac{\eta}{l}-t)^2} d_{l}^{\mp} G_k
  &\underset{\eqref{eq:GGd}, \ref{item:B4}}{\leq}
  \frac{c}{k}(\frac{\eta}{k^2})^{-1} \frac{2\pi}{k}.
\end{align*}
It thus follows that
\begin{align*}
  |G_{k\pm 1}(T)|\leq 10 \frac{c}{k} (\frac{\eta}{k^2})^{-1},
\end{align*}
which is exactly the desired estimate \ref{item:B5} in that case.

\underline{Improving (B1'), (B2'), (B4'):}
These estimates follow by the same argument as in the unmodified cases. We
omit the details for brevity. 
\end{proof}

We next turn to the non-resonant cases, where the initial data is localized on a
mode $k_0\neq k$.
Similarly to the resonant case we here show that the effect on neighboring modes
decreases in terms of powers of $c(\frac{\eta}{k^2})^{-2}$ except for the
interact of the mode $k$ with its neighbors, which increases by a factor $c
\frac{\eta}{k^3}$.
In particular, we observe that since the evolution equations
\eqref{eq:integralequations} only explicitly include nearest neighbor
interactions (and one interaction $\theta_l \mapsto G_l$) the estimates derived
in Proposition \ref{prop:res} only relied on the relative growth or decrease of
these weights.
Hence, the proofs in the non-resonant case are largely identical to the resonant
case except for a different choice of initial data and multiplication by a
suitable factor.

The estimates for the non-resonant case are summarized in the following proposition.
\begin{prop}
  \label{prop:nonres}
  Let $k_0\neq k$ and suppose that at time $t_k$, it holds that
  \begin{align*}
    \theta_l=\delta_{l k_0},\  G_l\equiv 0.
  \end{align*}
  We then define a weight function $\gamma: \Z \rightarrow \R_{+}$ to satisfy the
  following properties:
  \begin{enumerate}
  \item $\gamma(k_0)=1$. 
  \item $\gamma(k+1)=\gamma(k-1)= c \frac{\eta}{k^3} \gamma(k)$. 
  \item $\gamma(l+1)= c (\frac{\eta}{k^2})^{-2} \gamma(l)$ if $l>k_0$ and
    $\gamma(l-1)= c (\frac{\eta}{k^2})^{-2} \gamma(l)$ if $l<k_0$ unless this
    would violate the second property.
  \end{enumerate}
  
  Then the following bootstrap estimates hold:
  \begin{enumerate}[start=1, label={(C\arabic*)}]
  \item \label{item:C1} For all $l$ it holds that
    \[|\theta_l(T)-\delta_{l k_0}|\leq \gamma(l).\]
  \item \label{item:C2} For all $l$ it holds that
    \[|G_{l}(T)|\leq c (\frac{\eta}{k^2})^{-2} \gamma(l).\]
  \end{enumerate}

  If instead $G_l(t_k)= \delta_{l k_0}$ for some $k_0 \neq k$,
  then the following bootstrap estimates hold:
  \begin{enumerate}[start=1, label={(D\arabic*)}]
  \item \label{item:D1} For all $l$ it holds that
    \[|\theta_l(T)| \leq \frac{0.5}{k} \gamma(l) .\]
  \item \label{item:D2} For all $l$ it holds that
    \[|G_l(T)-\delta_{l k_0}\exp(-\nu \int_{t_k}^Tk_0^2+(\eta-k_0s)^2ds)|\leq \gamma(l).\]
  \end{enumerate}
\end{prop}

\begin{proof}[Proof of Proposition \ref{prop:nonres}]
  In this proof we follow the same bootstrap strategy as in Proposition
  \ref{prop:res}, where the initial bootstrap estimates again follow by a local
  contraction argument.

  It thus remains to be shown that the bootstrap estimates \ref{item:C1},
  \ref{item:C2} and \ref{item:D1}, \ref{item:D2} self-improve.
  Here we note that in our proof the estimates \ref{item:B1}--\ref{item:B5} we
  only used the relative size of the bounds on neighboring modes compared to the
  size of the desired bound on the current mode.
  Therefore, in the current estimate we only need to control
  \begin{align*}
    \frac{\gamma(l)}{\gamma(l\pm 1)},
  \end{align*}
  which by construction satisfies the same estimates as in Proposition
  \ref{prop:res}.
  In particular, most bounds follow by the exact same argument. In the interest
  of brevity we hence only comment on possible differences in the proof.

  \underline{Improving \ref{item:C1}, \ref{item:C2}:}
  Since $\theta_{k}(t_k)=0$, $G_k(t_k)=0$ the estimates for $\theta_{k\pm 1}$ simplify and
  follow by the same argument as for \ref{item:B1} to \ref{item:B5}.
  As the only difference we note that for $\theta_{k_0}(T)$ we, of course,
  consider
  \begin{align*}
    \theta_{k_0}(T)- \theta_{k_0}(t_k)= \theta_{k_0}(T)-1. 
  \end{align*}

  \underline{Improving \ref{item:D1},\ref{item:D2}:}
  We again note that in our bounds \ref{item:B1}--\ref{item:B5} we only required
  control on the relative size of the desired estimates.
  For instance, in order to control $\|\theta_{k}-1\|_{L^\infty}$ in \ref{item:B1} we only
  used that $\|G_{k\pm 1}\|_{L^\infty}$ and $\|\theta_{k\pm 1}\|_{L^\infty}$
  were controlled in terms of $c \frac{\eta}{k^3}$ as compared to the bound by
  $1$ imposed on $\theta_{k}$ by the bootstrap assumption.

  The estimates hence follow completely analogously, with the only difference
  that for $G_{k_0}$ the integral equations \eqref{eq:integralequations} include
  a contribution by the initial data.
\end{proof}

We have thus shown that Propositions \ref{prop:res} and \ref{prop:nonres} follow
as consequences of the coefficient estimates collected in Lemmas \ref{lemma:coeffestimates} and \ref{lemma:coefficientsForG}.
In turn, in Section \ref{sec:mechanism} we will use these propositions to
establish Theorem \ref{theorem:mechanism}.
The coefficient estimates are then proven in Subsection
\ref{sec:coeffestimates}.

\subsubsection{Proof of Theorem \ref{theorem:mechanism}}
\label{sec:mechanism}
In Propositions \ref{prop:res} and \ref{prop:nonres} we have established that
$\theta$ and $G$ satisfy suitable estimates in weighted $\ell^\infty$ spaces.
In the following we discuss how to pass from these $\ell^\infty$ estimates to
estimates on $X$ and thus prove Theorem \ref{theorem:mechanism}.

\begin{proof}[Proof of Theorem \ref{theorem:mechanism}]
Let $\theta(t_k), G(t_k) \in X$ be given initial data. Then by linearity we may
express the solution $\theta(T), G(T)$ as the sum over the solutions with
initial data localized on a single mode.
If we denote the weight of Proposition \ref{prop:nonres} for a mode $k_0$ by
$\gamma_{k_0}$, it thus follows
from Propositions \ref{prop:res}, \ref{prop:nonres} that for any $l \not \in
\{k-1,k, k+1\}$:
\begin{align*}
  |\theta_l(T)- \theta_{l}(t_k)| &\underset{\ref{item:B3}, \ref{item:C1}, \ref{item:D1}}{\leq} |\theta_k(t_k)| c \frac{\eta}{k^3} (c (\frac{\eta}{k^2})^{-2})^{|l-k|+1}\\
  & \quad\quad \quad \quad + \sum_{k_0\neq k} |\theta_{k_0}(t_k)| \gamma_{k_0}(l) + |G_{k_0}(t_k)| \gamma_{k_0}(l),
\end{align*}
and
 \begin{align*}
   & \quad\quad \quad \quad  |G_l(T)- G_l(t_k)\exp(-\nu\int_{t_k}^T l^2+(\eta-ls)^2ds)| \\
   & \underset{\ref{item:B4},\ref{item:C2},\ref{item:D2}}{\leq}
   \frac{\eta}{k^3} (c (\frac{\eta}{k^2})^{-2})^{|l-k|} |\theta_k(t_k)|
   + \frac{2}{k} |G_k(t_k)|\\
   &\quad \quad \quad \quad  + \sum_{k_0\neq k} \gamma_{k_0}(l) |\theta_k(t_k)| + |G_{k_0}(t_k)| \gamma_{k_0}(l).
 \end{align*}
We recall that $\gamma_{k_0}(l)$ rapidly decays in $|k_0-l|$ and thus interpret
the right-hand-side as discrete convolutions or rather integral kernels applied
to the (absolute values of the) initial data.

Hence, in the following we intend to pass from a point-wise bound
\begin{align*}
  |\theta_{l}(T)- \theta_{l}(t_k)|\leq \sum_{k_0} K(k_0,l) (|\theta_{k_0}|+|G_{k_0}|)
\end{align*}
to a bound on weighted $\ell^2$ spaces, $X$.

For this purpose we note that by Schur's test if the kernel $K(\cdot, \cdot)$ satisfies
\begin{align*}
  \sup_{l} \sum_{k_0} |K(k_0,l)| \leq C_1< \infty, \\
  \sup_{k_0} \sum_{l} |K(k_0,l)| \leq C_2 <\infty,
\end{align*}
then the associated integral operator maps $\ell^2$ to $\ell^2$ with operator
norm bounded by $\sqrt{C_1 C_2}$.

Applied to our case we observe that
\begin{align*}
  K(k_0, l) \leq c\frac{\eta}{k^3} (c (\frac{\eta}{k^2})^{-2})^{\max(0, |l-k_0|-1)},
\end{align*}
and therefore by the geometric series
\begin{align*}
  C_1=C_2 \leq c\frac{\eta}{k^3} \frac{1}{1-c (\frac{\eta}{k^2})^{-2}}\leq 1.1c\frac{\eta}{k^3}.
\end{align*}
It thus follows that
\begin{align*}
  \|(\theta(T)-\theta(t_k))_{l \not \in \{k-1,k,k+1\}}\|_{\ell^2} \leq 2  c\frac{\eta}{k^3} (\|\theta(t_k)\|_{\ell^2} + \|G(t_k)\|_{\ell^2}),
\end{align*}
and analogous estimates hold for $G$.
Moreover, by the Definition \ref{defi:X} of the space $X$ its weight function is
bounded by $2^{|l|}$ and we may thus apply the same argument with $2^{|k_0-l|}
|K(k_0,l)|$ to also deduce bounds on $X$.

It thus remains to discuss the modes $l \in \{k-1,k,k+1\}$.
Here we observe that by the bootstrap estimates and the triangle inequality
\begin{align*}
  |\theta_k(T)- \theta_{k}(t_k)| \leq \frac{10c}{k}(\frac{\eta}{k^2})^{-1}|\theta_k(t_k)|
  + \sum_{k_0 \neq k}\gamma_{k_0}(l) (|\theta_{k_0}(t_k)| + |G_{k_0}(t_k)|).
\end{align*}
Similarly, the modes $k-1, k+1$ by \ref{item:B2} and the other bootstrap
estimates satisfy
\begin{align*}
  & \quad |\theta_{k\pm 1}(T)- \theta_{k \pm 1}(t_k)-\int_{t_k}^T c_{k+1}^{-} \theta_k(t_k) - \int_{t_k}^T d_{k+1}^{-} \hat{G_K}|\\
  &\leq \frac{0.5}{k} c \frac{\eta}{k^3} |\theta_{k}(t_k)| + \sum_{k_0\neq k} \gamma(l) (|\theta_{k_0}(l)|+ |G_{k_0}(l)|),
\end{align*}
where
\begin{align*}
  \hat{G_k}(t)&= G_{k}(t_k) \exp(-\nu \int_{t_k}^t k^2+(\eta-ks)^2 ds) \\
  &\quad + \theta_{k}(t_k) \int_{t_k}^t \exp(-\nu\int_{t_k}^t k^2 +(\eta-ks)^2 ds)2 \frac{(\frac{\eta}{k}-\tau)}{(1+(\frac{\eta}{k}-\tau)^2)^2} d\tau
\end{align*}
accounts for the explicit influence of $G_k(t_k)$ and $\theta_k(t_k)$.

We further recall that by Lemma \ref{lemma:coeffestimates}
\begin{align*}
  \int_{t_k}^{t_{k-1}} c_{k+1}^{-} dt \approx c\frac{\eta}{k^3}\pi, \\
  \int_{t_k}^{t_{k-1}} d_{k+1}^{-} dt \approx c\frac{\eta}{k^2}\frac{\pi}{2}.
\end{align*}
and note that
\begin{align*}
  \int_{t_k}^{t_{k-1}} \exp(-\nu\int_{t_k}^t k^2 +(\eta-ks)^2 ds)\frac{2}{k} \frac{(\frac{\eta}{k}-t)}{(1+(\frac{\eta}{k}-t)^2)^2} dt \leq 2.
\end{align*}
Hence, $\theta_{k-1}$ satisfies the desired upper bound
\begin{align*}
  |\theta_{k\pm 1}(T)- \theta_{k\pm 1}(t_{k}) \int_{t_k}^T c_{k+1}^{-} \theta_k(t_k) - \int_{t_k}^T d_{k+1}^{-} \hat{G_K}|\\
  \leq \max(c\frac{\eta}{k^3}, \frac{1}{k}) (\|\theta(t_k)\|_{X} + \|G(t_k)\|_{X}),
\end{align*}
where we may insert $T=t_{k-1}$.
\end{proof}
In the following subsection we provide the proof of the
coefficient estimates of Lemma \ref{lemma:coeffestimates} and \ref{lemma:coefficientsForG}.

\subsubsection{Proof of Coefficient Estimates}
\label{sec:coeffestimates}

In Section \ref{sec:G} we have stated estimates on the coefficients in the
evolution equations \eqref{eq:integralequations} and used them to establish
bounds on the resonance mechanism.
In the following we prove these estimates. 
\begin{proof}[Proof of Lemma \ref{lemma:coeffestimates}]
  We note that for $l \in \{k-1,k+1\}$ it holds that
  \begin{align*}
    c_{k \pm 1}^{\mp} &= c \frac{\eta}{k^3} \frac{1}{(1+(\frac{\eta}{k}-t)^2)^2}, \\
    d_{k \pm 1}^{\mp} &= c\frac{\eta}{k^3} \frac{1}{1+(\frac{\eta}{k}-t)^2},
  \end{align*}
  which we integrate in time to obtain \eqref{eq:resonantcoeffs}.
  In particular, we stress that if one integrates over all of $I_k$ and
  $\frac{\eta}{k^2}$ is large, the integral is comparable to integral over all
  of $\R$ and
  \begin{align*}
    \int_{\R}\frac{1}{(1+(\frac{\eta}{k}-t)^2)^2}dt &= \frac{\pi}{2}, \\
    \int_{\R}\frac{1}{1+(\frac{\eta}{k}-t)^2}dt &= \pi,
  \end{align*}
  by explicit calculation.
  Therefore the resonant contributions \eqref{eq:resonantcoeffs}, for this
  choice of times, are comparable to $\frac{g}{\nu} \frac{\eta}{k^3}$ and thus potentially very large.\\

  We next turn to estimating all other coefficients. We recall that:
  \begin{align*}
    c_{l}^{\pm} &= c \frac{\eta}{(l\pm 1)^3} \frac{1}{(1+(\frac{\eta}{l\pm 1}-t)^2)^2}, \\
    d_{l}^{\pm} &= c \frac{\eta}{(l\pm 1)^3} \frac{1}{1+(\frac{\eta}{l\pm 1}-t)^2}.
  \end{align*}
  In this non-resonant-case considered in \eqref{eq:nonresonantcoeffs}, it holds
  that $l\pm 1 \neq k$ and thus
  \begin{align*}
    (\frac{\eta}{l\pm 1}-t)^2 \geq \frac{1}{4} \max(t, \frac{\eta}{(l\pm 1)^2}, \frac{\eta}{k^2})^2.
  \end{align*}
  It thus follows that
  \begin{align*}
    c_{l}^{\pm} \leq 16 c \frac{\eta}{(l\pm 1)^3} \max(t, \frac{\eta}{(l\pm 1)^2}, \frac{\eta}{k^2})^{-4} \leq 16 \frac{c}{k}  (\frac{\eta}{k^2})^{-3},
  \end{align*}
  and
  \begin{align*}
     d_{l}^{\pm} \leq 4 c \frac{\eta}{(l\pm 1)^3}\max(t, \frac{\eta}{(l\pm 1)^2}, \frac{\eta}{k^2})^{-2} \leq 4 \frac{c}{k} (\frac{\eta}{k^2})^{-1}.
  \end{align*}
  Using the fact that $|I_k|= \frac{1}{2} (\frac{\eta}{k(k-1)}+
  \frac{\eta}{k(k+1)}) \leq 2\frac{\eta}{k^2}$ , the estimates
  \eqref{eq:nonresonantcoeffs}
  thus follow.
\end{proof}
We next turn to the coefficient estimates required to control the evolution of $G$.

\begin{proof}[Proof of Lemma \ref{lemma:coefficientsForG}]
\underline{Estimating \eqref{eq:Gthetares}:}
We first consider the integral
\begin{align*}
  \int_{t_k}^{T}\exp(-\nu\int_{t}^{T}l^2+(\eta-l\tau)^2 d\tau) \left| 2 \frac{(\frac{\eta}{l}-t)}{(1+(\frac{\eta}{l}-t)^2)^2}\right|, 
\end{align*}
for $k=l$, which is related to the forcing exerted by the mode $\theta_{k}$ on $G_k$.
Since we are searching for estimates uniform in $\nu$, we bound the exponential
by $1$ and observe that
\begin{align*}
  \int_{\R} \left|\frac{(\frac{\eta}{l}-t)}{(1+(\frac{\eta}{l}-t)^2)^2}\right| = 1,
\end{align*}
which yields the desired upper bound.
However, we remark that if the absolute value signs are not introduced then the
fraction is anti-symmetric with respect to $t=\frac{\eta}{k}$.

\underline{Estimating \eqref{eq:Gthetanonres}:}
  We next consider the same integral for $l\neq k$.
  In this case it holds that
  \begin{align*}
    (\frac{\eta}{l}-t)^2 \geq \frac{1}{4} \max(t, \frac{\eta}{k^2}, \frac{\eta}{l^2})^2
  \end{align*}
  and thus
  \begin{align*}
    2 \frac{(\frac{\eta}{k}-t)}{(1+(\frac{\eta}{k}-t)^2)^2} \\
    \leq 2 (\frac{\eta}{k^2})^{-3}.
  \end{align*}
  The claimed estimate then follows by again noting that the length of $I_k$ is
  controlled in terms of $\frac{\eta}{k^2}$.

\underline{Estimating \eqref{eq:Gf}:}
  Since $f\leq \frac{g}{\nu}\frac{1}{1+t^2}$, we need to control
  \begin{align*}
    \int_{t_k}^T \frac{c}{1+t^2} \frac{\eta}{l^3} \frac{1}{(1+(\frac{\eta}{l}-t)^2)^2} dt,\\
    \int_{t_k}^T \frac{c}{1+t^2} \frac{\eta}{l^3} \frac{1}{1+(\frac{\eta}{l}-t)^2} dt.
  \end{align*}
  Consider the first integral.
  We recall that $t \approx \frac{\eta}{k}$ and argue as in Section
  \ref{sec:small}. That is, for $l>\frac{k}{2}$, $l\neq k$ we may bound
  \begin{align*}
    \frac{c}{1+t^2} \frac{\eta}{l^3} \frac{1}{(1+(\frac{\eta}{l}-t)^2)^2} \leq c (\frac{\eta}{k})^{-2} \frac{\eta}{k^3} (\frac{\eta}{k})^{-4} \leq  c(\frac{\eta}{k})^{-5},
  \end{align*}
  which is then integrated over a time interval of length bounded by
  $\frac{\eta}{k^2}$ and we thus gain
  \begin{align*}
    c (\frac{\eta}{k})^{-4}.
  \end{align*}
  If $l=k$, this estimate is slightly worse as
  \begin{align*}
    c (\frac{\eta}{k})^{-2} \frac{\eta}{k^3} \int_{t_{k}}^T \frac{1}{1+(\frac{\eta}{l}-t)^2} dt \leq c (\frac{\eta}{k})^{-1}.
  \end{align*}
  Finally, if $l<\frac{k}{2}$, we bound
  \begin{align*}
    \frac{c}{1+t^2} \frac{\eta}{l^3} \frac{1}{1+(\frac{\eta}{l}-t)^2} \leq c \frac{\eta}{l^3}  \frac{1}{1+t^2} \frac{1}{(1+ \frac{1}{4}\max(\frac{\eta}{l},t)^2)^2}
    \leq \frac{c}{k} (\frac{\eta}{k})^{-5}
  \end{align*}
  and thus bound the integral by
  \begin{align*}
    \frac{c}{k} (\frac{\eta}{k})^{-4}.
  \end{align*}

  For the second integral we argue similarly.
  If $l=k$, then we control the integral by
  \begin{align*}
    \frac{c}{1+(\frac{\eta}{k})^2} \frac{\eta}{k^3} \pi \leq \frac{c}{k} (\frac{\eta}{k})^{-1}.
  \end{align*}
  If $l\neq k$, then
  \begin{align*}
    \frac{\eta}{l^3} \frac{1}{1+(\frac{\eta}{l}-t)^2} \leq \frac{1}{k}(\frac{\eta}{k})^{-1}
  \end{align*}
  is uniformly integrable and we control by
  \begin{align*}
    \frac{c}{1+t^2} \leq c (\frac{\eta}{k})^{-2}.
  \end{align*} 

  \underline{Estimating \ref{eq:Gthetac}}
  We next consider the integral
  \begin{align*}
     \int_{t_k}^{T}\exp(-\nu\int_{t}^{T}l^2+(\eta-l\tau)^2 d\tau)\frac{1}{1+(\frac{\eta}{l}-t)^2}c_{l}^{\pm},
  \end{align*}
  where we again bound the exponential by $1$, since we allow for $\nu$ to be
  very small.
  Then for
  \begin{align}
    \label{eq:Gthetacnonu}
    \frac{1}{1+(\frac{\eta}{l}-t)^2} c \frac{\eta}{(l\pm 1)^3} \frac{1}{(1+(\frac{\eta}{l}-t)^2)^2}
  \end{align}
  we distinguish three cases.

  If $l\neq k$ and $l\pm 1 \neq k$, that is all frequencies are non-resonant,
  then we may control
  \begin{align*}
    (\frac{\eta}{l}-t)^2 \geq \frac{1}{4}\max(\frac{\eta}{l^2}, \frac{\eta}{k^2})^2
  \end{align*}
  and thus \eqref{eq:Gthetacnonu} is bounded by 
  \begin{align*}
     \frac{\eta}{(l\pm 1)^3} 16 c \max(\frac{\eta}{l^2}, \frac{\eta}{k^2})^{-6}
    \leq 16 \frac{c}{k} (\frac{\eta}{k^2})^{-5}.
  \end{align*}
  Since the length of interval $I_k$ is bounded by $2 \frac{\eta}{k^2}$, the
  claimed bound hence follows.

  If $l=k$, then $l\pm 1 \neq k$ and we may control \eqref{eq:Gthetacnonu} by a
  constant times
  \begin{align*}
    & \quad c \frac{\eta}{(k\pm 1)^3} (\frac{\eta}{k^2})^{-4} \frac{1}{1+(\frac{\eta}{k}-t)^2} \\
    & \leq \frac{c}{k} (\frac{\eta}{k^2})^{-3} \frac{1}{1+(\frac{\eta}{k}-t)^2}
  \end{align*}
  Since the last factor is integrable in time with integral bounded by $\pi$,
  the claimed estimate follows.

  Finally, if $l\pm 1=k$ then $l\neq k$ and we may control \eqref{eq:Gthetacnonu} by a
  constant times
  \begin{align*}
    & \quad c \frac{\eta}{k^3} (\frac{\eta}{k^2})^{-2} \frac{1}{(1+(\frac{\eta}{l}-t)^2)^2}\\
    &\leq \frac{c}{k} (\frac{\eta}{k^2})^{-1} \frac{1}{(1+(\frac{\eta}{l}-t)^2)^2}.
  \end{align*}
  We again observe that the last factor is integrable in time with integral
  bounded by $\frac{\pi}{2}$.
  This concludes the estimate of \eqref{eq:Gthetac}.

  \underline{Estimating \ref{eq:GGd}}
  It remains to estimate
  \begin{align*}
    \int_{t_k}^{T}\exp(-\nu\int_{t}^{T}l^2+(\eta-l\tau)^2 d\tau) \frac{1}{1+(\frac{\eta}{l}-t)^2} c \frac{\eta}{(l\pm 1)^2} \frac{1}{1+(\frac{\eta}{l}-t)^2},
  \end{align*}
  where we argue similarly as in the case of estimate \eqref{eq:Gthetac} and
  bound the exponential by $1$.
  We thus have to estimate
  \begin{align}
    \label{eq:GGdnonu}
    \frac{1}{1+(\frac{\eta}{l}-t)^2} c \frac{\eta}{(l\pm 1)^2} \frac{1}{1+(\frac{\eta}{l}-t)^2},
  \end{align}
  where distinguish two cases.

  If $l\neq k\neq k-1$ all frequencies are non-resonant and we can again bound
  $(\frac{\eta}{l}-t)^2$ and $(\frac{\eta}{l}-t)^2$ from below.
  We may thus bound \eqref{eq:GGdnonu} by a constant times
  \begin{align*}
  c \frac{\eta}{(l\pm 1)^3} \max (\frac{\eta}{l^2}, \frac{\eta}{k^2})^{-4}
    \leq \frac{c}{k} (\frac{\eta}{k^2})^{-3},
  \end{align*}
  and the desired bound again follows by estimating the length of $I_k$ from
  above.

  Suppose that $l=k$ and hence $l\pm 1 \neq k$ (the case $l\pm 1=k$ is analogous). 
  Then we may instead control
  \eqref{eq:GGdnonu} by a constant times
  \begin{align*}
    & \quad  c \frac{\eta}{(k\pm 1)^3} (\frac{\eta}{k^2})^{-2} \frac{1}{1+(\frac{\eta}{k}-t)^2} \\
    &\leq 4 \frac{c}{k} (\frac{\eta}{k^2})^{-1}\frac{1}{1+(\frac{\eta}{k}-t)^2}.
  \end{align*}
  The desired bound then follows by noting that the last factor is integrable in
  time with integral bounded by $\pi$.
\end{proof}

\subsection{Proof of Stability in Theorem \ref{theorem:main}}
\label{sec:proofofmain}
As a final result of this section we combine the estimates of Theorems
\ref{theorem:easy}, \ref{theorem:intermediate} and \ref{theorem:mechanism} to
establish global in time stability.

\begin{proof}[Proof of Theorem \ref{theorem:main}]
As sketched in Section \ref{sec:intro} the strategy of our construction here is
the following:
\begin{itemize}
\item We start with given initial data at time $0$ and control it up to a time
  $t_k$ of size $C\sqrt{\eta}$ by Theorem \ref{theorem:easy}.
\item On the intermediate range of times (for which  $c\frac{\eta}{k^3}\pi \leq 1$) we
  control the evolution by Theorem \ref{theorem:intermediate}.
\item On the time interval $(\frac{\eta}{\sqrt[3]{c\eta \pi}}, 2 \eta)$ we encounter a sequences of
  resonances, each possibly leading to norm inflation by a factor. This corresponds to an echo
  chain
  \[k \mapsto k-1 \mapsto \dots \mapsto 1.\]
  There we use Theorem \ref{theorem:mechanism} to control the growth due to each echo.
\item Finally, after the time $t_0=2 \eta$ the evolution is stable by Theorem
  \ref{theorem:easy} and we have thus established global in time control of
  solutions.
  We thus have constructed global in time solutions
  exhibiting echo chains. 
\end{itemize}

We recall that the coefficient functions of the linearized problem \eqref{eq:linwave} do not depend on
$y$ explicitly and that the problem decouples after a Fourier transform in $y$.
In this proof we will hence consider $\eta \in \R$ as a given parameter, where
the estimate for general data follows by integration with respect to $\eta$.

Let thus $c$ and $\eta$ be given and let $\theta(0), G(0) \in X$. In Theorem \ref{theorem:easy} we have
shown that the evolution is stable globally in time if $c\eta\ll 1$, we hence in the
following without loss of generality restrict to the case $c\eta \gtrsim 1$.
For simplicity of notation let
\begin{align*}
  k_1&\approx \sqrt{8000 c\eta \pi},\\
  k_0&\approx \sqrt{c\eta \pi},
\end{align*}
where we round down to an integer.

Then by Theorem \ref{theorem:easy} it holds that for all $0<t\leq t_{k_1}$
\begin{align*}
40^2 \|\theta(t)\|_{X}^2 + \|G(t)\|_{X}^2 \leq 2 (40^2 \|\theta(0)\|_{X}^2 + \|G(0)\|_{X}^2).  
\end{align*}
In particular, it follows that
\begin{align*}
  \|\theta(t_{k_1})\|_{X}^2 + \|G(t_{k_1})\|_{X}^2 \leq 2 \cdot 40^2 (\|\theta(0)\|_{X}^2 + \|G(0)\|_{X}^2).
\end{align*}
Next, we use Theorem \ref{theorem:intermediate} to show that for $t_{k_1}\leq t
\leq t_{k_0}$ it holds that
\begin{align*}
   \|\theta(t)\|_{X} + \|G(t)\|_X &\leq e^{20 \cdot 3\pi \sqrt[3]{c\eta}} (\|\theta(t_{k_1})\|_{X}^2 + \|G(t_{k_1})\|_{X}^2 ) \\
  &\leq 2 \cdot 40^2 e^{20 \cdot 3\pi \sqrt[3]{c\eta}} (\|\theta(0)\|_{X}^2 + \|G(0)\|_{X}^2).
\end{align*}
Next, on the time interval $(t_{k_0}, t_{0}=2\eta)$, we have shown in Theorem
\ref{theorem:mechanism} that our solution grows at most by a factor
\begin{align*}
  \prod_{k=1}^{k_0} 2\frac{c\eta\pi}{k^3} = \frac{(2c\eta\pi)^{k_0}}{(k_0!)^3}\leq \frac{C}{(c\eta)^{3/2}} e^{3 \sqrt[3]{2c\eta\pi}},
\end{align*}
where we used Stirling's approximation to approximate
\begin{align*}
  k_0! \sim \sqrt{2\pi k_0} k_0^{k_0} e^{-k_0}
\end{align*}
in the last step. Therefore $k_0^{3k_0}$ and $(2c\eta\pi)^{k_0}$
then cancel by our choice of $k_0$ and a bound in terms of $e^{3k_0}$ remains.

In particular, it follows that
\begin{align*}
  \|\theta(t_0)\|_{X} + \|G(t_0)\|_X &\leq (c\eta)^{-3/2}  e^{\sqrt[3]{2c\eta\pi}} (\|\theta(t_{k_0})\|_{X} + \|G(t_{k_0})\|_X)\\
  &\leq  2(c\eta)^{-3/2} 40^2 e^{200 \sqrt[3]{c\eta}} (\|\theta(0)\|_{X}^2 + \|G(0)\|_{X}^2).
\end{align*}
Finally, by Theorem \ref{theorem:easy}, for all times $t\geq t_0=2 \eta$ it
holds that
\begin{align*}
  \|\theta(t)\|_{X} + \|G(t)\|_X &\leq 1.1 (\|\theta(t_0)\|_{X} + \|G(t_0)\|_X)\\
  &\leq \frac{1.1}{(c\eta)^{3/2}} 2 \cdot 40^2 e^{200\sqrt[3]{c\eta}} (\|\theta(0)\|_{X}^2 + \|G(0)\|_{X}^2),
\end{align*}
which concludes the proof.

The evolution preserves Gevrey 3 regularity with a possible loss of constant.
\end{proof}

Having established this stability result in Gevrey $3$ regularity in the
following we show that the estimate is optimal (up to the choice of
constant).
More precisely, we construct initial data which achieves growth at least by
\begin{align*}
  e^{\sqrt[3]{c\eta}}.
\end{align*}

\section{Echo Chains and Blow-up}
\label{sec:blow-up}

As a complementary result to the stability estimates for initial data in a
Gevrey $3$ class with large constant, we show that there exists data in a
critical Gevrey class that not only achieves norm inflation but blow-up in
Sobolev regularity as time tends to infinity.

\begin{theorem}
  \label{theorem:blow-up}
  Let $c, \nu>0$ be as in Theorem \ref{theorem:main} and suppose that
  $\sqrt[3]{c\eta}\gg \nu^{-1/2}$.
  Let further
  \begin{align*}
    k_2 \approx \frac{1}{10} \sqrt[3]{c\eta \frac{\pi}{2}}
  \end{align*}
  (rounded down).
  Then the solution of \eqref{eq:simplewave2} with initial data
  \begin{align*}
    \theta_l(0)=\delta_{lk_2},\ G(0)=0
  \end{align*}
  satisfies
\begin{align*}
  \exp(\sqrt[3]{c\eta}) \|\theta(0)\|_X\leq \|\theta(t)\|_X \leq \|\theta(t)\|_X + \|G(t)\|_X\leq \exp(50 \sqrt[3]{c\eta}) \|\theta(0)\|_X
\end{align*}
for all $t>2 \eta$ and $\theta(t)$ converges in $X$ as $t\rightarrow \infty$.
There thus exist global in time, asymptotically stable solutions achieving norm inflation.

Moreover, when considering the $y$-dependent formulation \eqref{eq:simplewave},
for each $\sigma \in \R$ there exists initial data $\theta(0) \in
  \mathcal{G}_{3} X$, $G(0)=0$, such that $\theta(t)$ converges in $H^{s}X$
  for all $s< \sigma$, but diverges to infinity in $H^{s}X$ for all $s>\sigma$.
\end{theorem}
\begin{figure}[htb]
  \centering
  \includegraphics[width=0.5\linewidth]{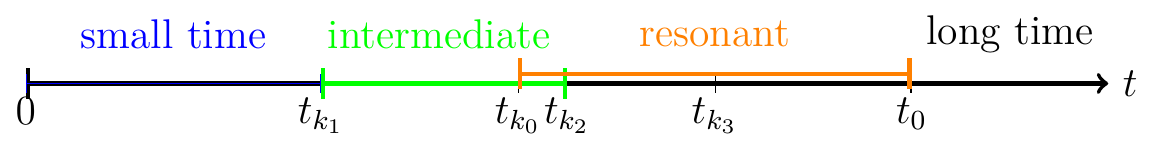}
  \caption{The time regimes considered in Section \ref{sec:blow-up}. Compared to
  Section \ref{sec:stability} we here allow an overlap of the intermediate and
  resonant regime. Furthermore, we divide the resonant regime into a part where
  $k\geq k_3$ is still large and one where $k< k_3$ is possibly small.}
  \label{fig:regimes2}
\end{figure}
We remark that the asymptotic stability of solutions has already been
establish in Section \ref{sec:small}. The effort of this section lies in the construction of the global in time
solutions exhibiting norm inflation and showing that lower bounds persist for
all times. Indeed, given such solutions we can construct solutions exhibiting blow-up as follows:
\begin{proof}[Proof of the blow-up result of Theorem \ref{theorem:blow-up}]
  Let $c, \nu>0$ be given and suppose that for all $\eta$ as in Theorem \ref{theorem:blow-up} there exist
  initial data $\theta[\eta](0)$ with $\|\theta[\eta](0)\|_X=1$ such that the associated evolution is
  asymptotically stable and such that
  \begin{align*}
    \theta^\infty[\eta]:=\theta[\eta](t)
  \end{align*}
  satisfies $\|\theta^{\infty}[\eta]\|_X =: \psi(\eta)\geq
  \exp(\sqrt[3]{c\eta})$.

  Then for given $\sigma \in \R$ there exists a density $\rho \in H^{\sigma}$
  with $\rho \not \in H^{s}$ for any $s>\sigma$ and such that the support of its
  Fourier transform is contained in the set $\{\eta: \sqrt[3]{c|\eta|}\gg
  \nu^{-1/2}\}$.
  For instance, such data can be explicitly constructed in Fourier space in terms of
  $|\eta|^{\alpha}|\log(\eta)|^{\beta}$ for suitable $\alpha, \beta$.
  
  We then consider the initial $\theta(0)$ with Fourier transform given by
  \begin{align*}
   \mathcal{F}(\theta)(0):= \frac{1}{\psi(\eta)} \mathcal{F}(\rho)(\eta) \theta[\eta](0).
  \end{align*}
  Since $\|\theta[\eta](0)\|_X=1$ and $\frac{1}{\psi(\eta)}\leq
  e^{-\sqrt[3]{c\eta}}$, clearly $\theta(0) \in \mathcal{G}_{3}X$.

  Moreover, by the asymptotic stability of the
  frequency-localized initial data it holds that
  \begin{align*}
    \mathcal{F}(\theta)(t) \rightarrow \mathcal{F}(\rho)(\eta) \frac{\theta^{\infty}[\eta]}{\psi(\eta)}.
  \end{align*}
  pointwise in $\eta$.
  By definition of $\psi(\eta)$ the last factor is normalized in $X$ and hence
  this pointwise (in frequency) limit is an element of $H^{\sigma}X$.
  In particular, by compactness of the embedding $H^{\sigma}\subset H^{s}$ for
  $s<\sigma$ we obtained the claimed convergence in $H^sX$ for $s<\sigma$. Since $\rho \not \in H^s$ for $s>\sigma$ we also obtain divergence in $H^s$, $s>\sigma$.
  This concludes the proof of the blow-up construction.
\end{proof}

Our main aim in the remainder of this section is thus to construct global in
time solutions for given $\eta$ which achieve the desired norm inflation.
As discussed in the heuristic model of Section \ref{sec:wave} the main
growth is expected to happen in the resonant time regime $(t_{k_0}, t_0=2\eta)$,
where
\begin{align*}
  k_0=\sqrt[3]{c\eta \frac{\pi}{2}}.
\end{align*}
For technical reasons we do not consider the extremal case of a full chain
starting at frequency $k_0$, but instead begin at frequency $k_2=\frac{k_0}{4}$
and only establish lower bounds on the norm inflation until the time
$k_3=\frac{k_0}{1000}$.

The corresponding time regimes and behavior of the solution are described in
more detail in the following proposition, which thus states the main steps of
the proof of Theorem \ref{theorem:blow-up}.

\begin{prop}
  \label{prop:timeregimes}
Let $c, \eta, \nu$ be given and assume that $\sqrt[3]{c\eta}\gg \max(\nu^{-1},1)$ (i.e.
choose $\eta$ large enough).
Furthermore define the following threshold values:
\begin{align*}
  k_0= \sqrt[3]{c\eta \pi}, \\
  k_1 = 4 k_0, \\
  k_2 = \frac{k_0}{10}, \\
  k_3 = \frac{k_0}{1000}.
\end{align*}
Then the solution with initial data
\begin{align*}
  \theta_l(0)= \delta_{l k_2}.
\end{align*}
has the following properties:
\begin{itemize}
\item At time $t_{k_1}$ it holds that
  \begin{align*}
    |\theta_{k_2}-1|&\leq c^2 (\frac{\eta}{k_1^2})^{-3}, \\
    |\theta_{l}| &\leq  c (\frac{\eta}{k^2})^{-2} \prod_{j=l}^{k_2} c(\frac{\eta}{j^2})^{-1}, \text{ if } l\leq k_1, \\
    |\theta_{l}|&\leq 2^{-|l-k_1|}  c (\frac{\eta}{k_1^2})^{-2}\prod_{j=k_1}^{k_2} c(\frac{\eta}{j^2})^{-1}.
  \end{align*}
  Thus at time $t_{k_1}$ the mode $k_2$ is by far the largest (the factor $c
  (\frac{\eta}{k_1^2})^{-2}\ll 1$) and we have very rapid decay of all other
  modes. 
\item At time $t_{k_2}$ it holds that
  \begin{align*}
    |\theta_{k_2}-1| &\leq c (\frac{\eta}{k_1^2})^{-1} e^{2}, \\
    |\theta_{l}| &\leq  c (\frac{\eta}{k_1^2})^{-1} e^{2}, \text{ if } l\geq k_2, \\
   |\theta_{l}|&\leq  c (\frac{\eta}{k_1^2})^{-2} e^2 \prod_{j=l}^{k_2} c(\frac{\eta}{j^2})^{-1} , \text{ if } l < k_2.
  \end{align*}
  Thus at time time still $\theta_{k_2}$ is the largest mode. Furthermore, while
  modes $l>k_2$ do not exhibit decay in $|l-k_2|$ anymore, this is still the
  case for $l<k_2$.
\item At the time $t_{k_3}$  it holds that
  \begin{align*}
    |\theta_{k_3}| | &\geq e^{k_0}, \\
    | \theta_{k_3-2}|&\geq e^{k_0}, \\
    |\theta_{l}|&\leq \frac{1}{1000} |\theta_{k_3}| \text{ for all } l \not \in \{k_3-1, k_3+1\}.
  \end{align*}
  At time $t_{k_3}$ the modes $\theta_{k_3}$ and $\theta_{k_3-2}$ are by far the
  largest modes and have achieved significant norm inflation.
\item For all times $t\geq t_0=2\eta$ it holds that
  \begin{align*}
    | \theta_{k_3-2}(t)|\geq e^{k_0}.
  \end{align*}
  While other modes might have grown even more, this growth persists. In
  particular $\|\theta(t)\|_{\ell^2} \geq e^{k_0} \|\theta(0)\|_{\ell^2}$.
\end{itemize}
\end{prop}
Since each time regime requires rather different techniques, we discuss the
regimes in different subsections.
The corresponding estimates on $G$ are then included in the respective Lemmas
\ref{lemma:smalltime}, \ref{lemma:inductionupper}, \ref{lemma:persist} and Proposition \ref{prop:inductionblow}.

\subsection{The Small Time Regime and Contraction Mappings}
\label{sec:contract}
In this section we consider the evolution of $\theta$ and $G$ in the small time
regime
\begin{align*}
  (0, t_{k_1})\approx (0, \frac{1}{10} \frac{\eta}{\sqrt[3]{c\eta\pi/2}}).
\end{align*}
We note that this choice of time interval implies that
\begin{align*}
  \frac{c}{\eta}{k^3}\leq 10^{-3}=0.001
\end{align*}
for all $k\geq k_1$ and that for $k$ smaller than this $\frac{\eta}{k}$ is not
part of this interval.
We will show that this implies that the associated integral equation for the
modes $\theta_{l}, \frac{1}{10}G_l$ is a contraction mapping in $L^\infty \ell^\infty$ on
this interval. The resulting $\ell^\infty$ bound is then subsequently improved
to the weighted decay estimate of Proposition \ref{prop:timeregimes}.

\begin{lemma}
  \label{lemma:smalltime}
  Let $\eta,c\nu, k_1, k_2$ be as in Proposition \ref{prop:timeregimes} and
  consider the solution of \eqref{eq:simplewave2} with initial data
  \begin{align*}
    \theta_l(0)=\delta_{lk_2},\ G(0)=0.
  \end{align*}
  Then at the time $t_{k_1}$ it holds that
    \begin{align*}
    |\theta_{k_2}-1|&\leq c^2 (\frac{\eta}{k_1^2})^{-2}, \\
    |\theta_{l}| &\leq  c (\frac{\eta}{k^2})^{-2} \prod_{j=l}^{k_2} c(\frac{\eta}{j^2})^{-1}, \text{ if } l\leq k_1, \\
      |\theta_{l}|&\leq 2^{-|l-k_1|}  c (\frac{\eta}{k_1^2})^{-2}\prod_{j=k_1}^{k_2} c(\frac{\eta}{j^2})^{-1}, \text{ if } l\neq k_2, l \geq k_1, \\
      |G_{l}|&\leq 10 c (\frac{\eta}{k^2})^{-2} \prod_{j=l}^{k_2} c(\frac{\eta}{j^2})^{-1}, \text{ if } l\leq k_1, l\neq k_2,\\
      |G_{l}|&\leq 10 \cdot 2^{-|l-k_1|}  c (\frac{\eta}{k_1^2})^{-2}\prod_{j=k_1}^{k_2} c(\frac{\eta}{j^2})^{-1}, \text{ if } l \geq k_1, \\
      |G_{k_2}| & \leq c^2 (\frac{\eta}{k_1^2})^{-2}.
  \end{align*}
\end{lemma}
We thus observe that at time $t_{k_1}$ our data exhibits a sharp concentration
on the mode $\theta_{k_2}$ (see Figure \ref{fig:peak} for an illustration). Moreover, the exponential decay in terms of
$|k-k_2|$ is stronger than possible growth by $\frac{c\eta}{k^3}$ due to the
resonance mechanism.
\begin{figure}[htbp]
  \centering
  \includegraphics[width=0.5\linewidth]{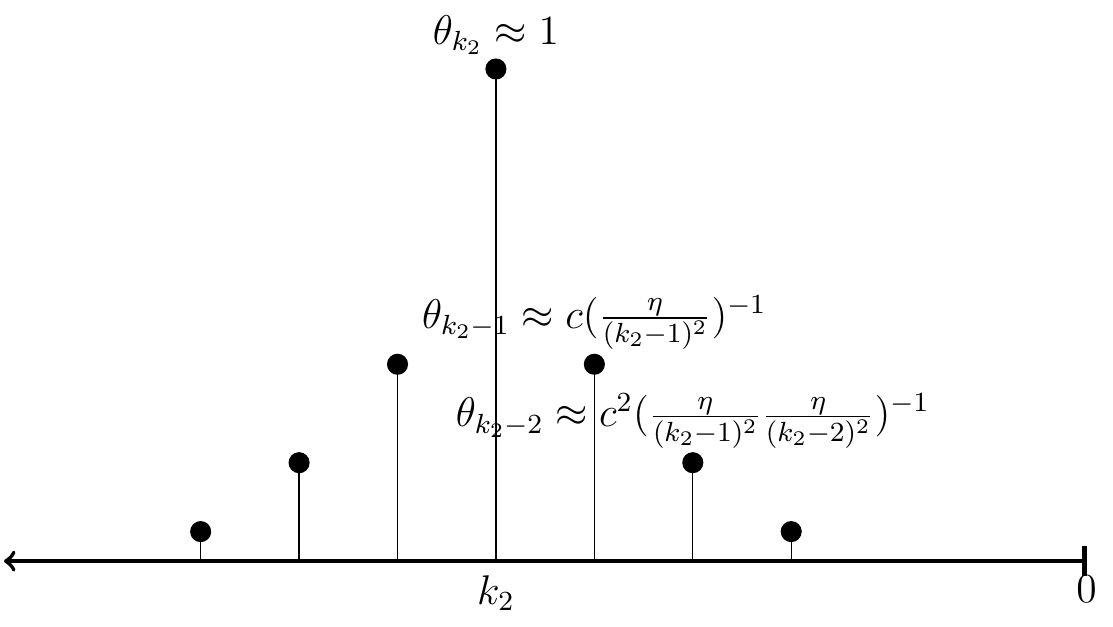}
  \caption{Distribution of $\theta_{l}$ at time $t_{k_1}$. We observe a peak at
    the frequency $k_2$ and rapid decay away from this frequency. }
  \label{fig:peak}
\end{figure}

\begin{proof}[Proof of Lemma \ref{lemma:smalltime}]
Given our choice of initial data we observe that $\theta$ and $G$ satisfy the
integral equations \eqref{eq:integralequations}:
\begin{align*}
  \begin{split}
   \theta_l(T)- \delta_{l k_2} &= \int_{0}^T c_{l}^{+} \theta_{l+1} + c_{l}^{-} \theta_{l-1} + d_{l}^{+} G_{l+1} + d_{l}^{-} G_{l-1} dt, \\
   G_{l}(T) &= \int_{0}^{T}\exp(-\nu\int_{t}^{T}l^2+(\eta-l\tau)^2 d\tau) 2 \frac{(\frac{\eta}{l}-t)}{(1+(\frac{\eta}{l}-t)^2)^2} \theta_{l} dt \\
  &\quad + \int_{0}^{T}\exp(-\nu\int_{t}^{T}l^2+(\eta-l\tau)^2 d\tau) f il \frac{\nu}{g} (d_l^{+} G_{l+1} + d_l^{-} G_{l-1}) dt \\
  &\quad + \int_{0}^{T}\exp(-\nu\int_{t}^{T}l^2+(\eta-l\tau)^2 d\tau) f il \frac{\nu}{g} (c_l^{+} \theta_{l+1} + c_l^{-} \theta_{l-1}) dt \\
  & \quad + \int_{0}^{T}\exp(-\nu\int_{t}^{T}l^2+(\eta-l\tau)^2 d\tau) \frac{1}{1+(\frac{\eta}{l}-t)^2}\left( c_{l}^{+} \theta_{l+1} + c_{l}^{-} \theta_{l-1}\right) dt \\
  & \quad + \int_{0}^{T}\exp(-\nu\int_{t}^{T}l^2+(\eta-l\tau)^2 d\tau) \frac{1}{1+(\frac{\eta}{l}-t)^2} \left(d_{l}^{+} G_{l+1} + d_{l}^{-} G_{l-1} \right) dt.
  \end{split}
\end{align*}
We then observe that
\begin{align*}
  \int_{0}^T c_{l}^{\pm} &\leq 0.01, \\
  \int_{0}^T d_{l}^{\pm} &\leq 0.01,
\end{align*}
since $\frac{c\eta}{k^3}\leq 4^{-3}$ for $k\geq k_1$ and the frequencies are not
yet resonant for $k\leq k_1$.
Similarly,
\begin{align*}
  \int_{0}^{T}\exp(-\nu\int_{t}^{T}l^2+(\eta-l\tau)^2 d\tau) f il \frac{\nu}{g} d_l^{\pm}&\leq 0.01, \\
  \int_{0}^{T}\exp(-\nu\int_{t}^{T}l^2+(\eta-l\tau)^2 d\tau) f il \frac{\nu}{g} c_l^{\pm}&\leq 0.01,\\
  \int_{0}^{T}\exp(-\nu\int_{t}^{T}l^2+(\eta-l\tau)^2 d\tau) \frac{1}{1+(\frac{\eta}{l}-t)^2} c_{l}^{\pm} & \leq 0.01, \\
  \int_{0}^{T}\exp(-\nu\int_{t}^{T}l^2+(\eta-l\tau)^2 d\tau) \frac{1}{1+(\frac{\eta}{l}-t)^2} d_{l}^{\pm} & \leq 0.01.
\end{align*}
The only possibly large contribution is hence given by
\begin{align*}
  \int_{0}^{T}\exp(-\nu\int_{t}^{T}l^2+(\eta-l\tau)^2 d\tau) 2 |\frac{(\frac{\eta}{l}-t)}{(1+(\frac{\eta}{l}-t)^2)^2}|  dt &\leq 2.
\end{align*}
Similarly to the results of Section \ref{sec:small} we thus consider the
equations as equations for $\theta$ and $\frac{G}{10}$ instead, so that
all coefficient functions are bounded by $0.1$.

We then define
\begin{align*}
  \hat{\theta}_l:=\theta_{l}- \delta_{l k_2}
\end{align*}
and view these equations as a fixed point iteration for
\begin{align*}
  \begin{pmatrix}
    \hat{\theta} \\ \frac{1}{10} G
  \end{pmatrix}
  = B[k_2] + L \begin{pmatrix}
    \hat{\theta} \\ \frac{1}{10} G
  \end{pmatrix},
\end{align*}
on the space
\begin{align*}
  L^{\infty} \ell^\infty.
\end{align*}
By the above choice $L$ is a contraction with norm less than $1/2$ and hence we
can control
\begin{align*}
  \|\theta(t)\|_{\ell^\infty} + \frac{1}{10} \|G(t)\|_{\ell^\infty} \leq 2 \|B[k_2]\|_{L^\infty \ell^\infty}.
\end{align*}
We note that here the components of $B[k_2]$ are given by
\begin{align*}
  \int_0^T \frac{c\eta}{k_2^3} \frac{1}{(1+(\frac{\eta}{k_2}-t)^2)^2} dt \leq \frac{c}{k_2} (\frac{\eta}{k_2^2})^{-2}
\end{align*}
and
\begin{align*}
  \int_{0}^{T}\exp(-\nu\int_{t}^{T}k_2^2+(\eta-k_2\tau)^2 d\tau) 2 \frac{(\frac{\eta}{k_2}-t)}{(1+(\frac{\eta}{k_2}-t)^2)^2} &\leq \min((\frac{\eta}{k_2})^{-2}, \frac{1}{\nu k_2^2} (\frac{\eta}{k_2})^{-3}), \\
  \int_{0}^{T}\frac{\exp(-\nu\int_{t}^{T}(k_2\pm 1)^2+(\eta-(k_2\pm 1)\tau)^2 d\tau)}{(1+(\frac{\eta}{k}-t)^2)^2} f(t) \frac{c\eta}{k_2^3} &\leq c(\frac{\eta}{k_2})^{-2}, \\
  \int_{0}^{T}\frac{\exp(-\nu\int_{t}^{T}(k_2\pm 1)^2+(\eta-(k_2\pm 1)\tau)^2 d\tau)}{1+(\frac{\eta}{k}-t)^2} f(t) \frac{c\eta}{k_2^3} &\leq c(\frac{\eta}{k_2})^{-2},\\
  \int_{0}^{T}\frac{\exp(-\nu\int_{t}^{T}(k_2\pm 1)^2+(\eta-(k_2\pm 1)\tau)^2 d\tau)}{(1+(\frac{\eta}{k_2\pm 1}-t)^2) (1+(\frac{\eta}{k}-t)^2)^2}  \frac{c\eta}{k_2^3} &\leq c(\frac{\eta}{k_2})^{-2}, \\
  \int_{0}^{T}\frac{\exp(-\nu\int_{t}^{T}(k_2\pm 1)^2+(\eta-(k_2\pm 1)\tau)^2 d\tau)}{(1+(\frac{\eta}{k_2\pm 1}-t)^2)(1+(\frac{\eta}{k}-t)^2)}  \frac{c\eta}{k_2^3}  &\leq c(\frac{\eta}{k_2})^{-2},
\end{align*}
respectively.
Therefore, it holds that
\begin{align*}
   \|\theta(t)\|_{\ell^\infty} + \frac{1}{10} \|G(t)\|_{\ell^\infty} \leq C (\frac{\eta}{k_2})^{-2}
\end{align*}
for a universal constant $C$.

We next use these rough upper bounds to establish the claimed improved bounds and decay.
  For this purpose we observe that
  \begin{align*}
    \int_{0}^{t_{k_1}} c_{l}^\pm \leq
    \begin{cases}
      0.01 & \text{ if }l\pm 1 > k_1, \\
      c (\frac{\eta}{(l\pm 1)^2})^{-2} & \text{ if } (l\pm 1)\leq k_1,
    \end{cases}
  \end{align*}
  since the latter frequencies have not yet been resonant.
  In particular, it follows that
  \begin{align*}
    |\theta_{k_2}(t) -1| \leq  c (\frac{\eta}{l^2})^{-2} (\|\hat{\theta}\|_{L^\infty \ell^\infty}+ \|G\|_{L^\infty \ell^\infty}),
  \end{align*}
  since $k_2-1, k_2+1$ are non-resonant.

  Given this size of $\theta_{k_2}$ the claimed decay in $|l-k_2|$ then follows by repeated
  insertion of the above estimates into the integral equation (using that
  $\hat{\theta}(0)=0$, $G(0)=0$).
  More precisely, we observe that
  \begin{align*}
    |G_{k_2}(T)| &\leq \min((\frac{\eta}{k_2})^{-2}, \frac{1}{\nu k_2^2} (\frac{\eta}{k_2})^{-3}) \|\theta_{k_2}\|_{L^\infty \ell^\infty}\\
    & \quad + c(\frac{\eta}{k_2})^{-2} (\|\theta_{k_2+1}\|_{L^\infty \ell^\infty}+ \|\theta_{k_2-1}\|_{L^\infty \ell^\infty}+\|G_{k_2+1}\|_{L^\infty \ell^\infty}+\|G_{k_2-1}\|_{L^\infty \ell^\infty}).
  \end{align*}
  Thus, inserting our bounds by $1$ and $\|B[k_2]\|_{L^\infty \ell^\infty}$ we
  observe the desired improvement for $|G_{k_2}(T)|$.

  For bounds on modes at frequencies further away from $k_2$, we require
  multiple iterations of this argument.
  For instance, we observe that after the first insertion 
  \begin{align*}
    \|\theta_{k_2-10}\|_{L^\infty \ell^{\infty}} &\leq c (\frac{\eta}{l^2})^{-2} (\|\theta_{k_2-9}\|_{L^\infty \ell^{\infty}}+ \|\theta_{k_2-11}\|_{L^\infty \ell^{\infty}})\\
    &\quad +c (\|G_{k_2-9}\|_{L^\infty \ell^{\infty}}+ \|G_{k_2-11}\|_{L^\infty \ell^{\infty}})
  \end{align*}
  does not enjoy a better bound than $\theta_{k_2-2}$.
  However, these improved estimates then hold for all frequencies $\leq
  k_2-2$. Thus, inserting the improved estimates once more we establish the desired bound
  for $k_2-3$ and the same (then suboptimal) bound for all modes smaller
  than $k_2-3$.
  Thus, repeating the argument $N$ times we obtain the desired bounds for modes
  with $|k-k_2|\leq N$ and thus the full result by letting $N\rightarrow \infty$.

\end{proof}

\subsection{The Intermediate Time Regime and Upper Bounds}
\label{sec:upper}
We next study the intermediate time regime
\begin{align*}
  (t_{k_1}, t_{k_2}).
\end{align*}
Since in this regime $\frac{c\eta\pi}{k^3}$ is not necessarily small anymore a
fixed point iteration is not possible anymore. Moreover, resonances can result
in growth of certain modes.
However, since by Lemma \ref{lemma:smalltime} we know that at least at time
$t_{k_1}$ the corresponding modes possibly becoming resonant are small we can
control the growth on each interval $I_k$ in this time regime by induction.

\begin{lemma}
  \label{lemma:inductionupper}
  Let $c, \eta, k_1, k_2$ be as in Proposition \ref{prop:timeregimes} and for $k_2\leq k\leq k_1$ define the constant $C_k$ by
  \begin{align*}
    C_{k_1}&= 1, \\
    C_{k-1} &= (1+(\frac{\eta}{k^2})^{-1})C_k.
  \end{align*}
  Then for all such $k$ it holds that 
  \begin{align}
    \label{eq:inductionstep2}
    \begin{split}
    |\theta_{k_2}(t_{k})|&\leq C_{k}, \\
    |\theta_{l}(t_k)|&\leq C_k \prod_{j=l}^{k_2} (c (\frac{\eta}{j^2})^{-1}) \text{ for } j\leq k, \\
    |\theta_{l}(t_k)|&\leq C_k \prod_{j=k-1}^{k_2} (c (\frac{\eta}{j^2})^{-1}) \text{ for } j\geq k, \\
      |G_{l}(t_k)|&\leq 10 C_k c (\frac{\eta}{k_1^2})^{-2} \prod_{j=l}^{k_2} c(\frac{\eta}{j^2})^{-1}, \text{ if } l\leq k, l\neq k_2,\\
      |G_{l}(t_k)|&\leq 10C_k \cdot 2^{-|l-k_1|}  c (\frac{\eta}{k_1^2})^{-2}\prod_{j=k-1}^{k_2} c(\frac{\eta}{j^2})^{-1}, \text{ if } l \geq k, \\
      |G_{k_2}(t_k)| & \leq C_k c^2 (\frac{\eta}{k_1^2})^{-2}.
   \end{split}
  \end{align}
\end{lemma}
We note that the upper bounds \eqref{eq:inductionstep2} here have a ``dent'',
where the bounds for $\theta_{k+2}$ and $\theta_{k}$ are the same and much larger than the
one for $\theta_{k+1}$ (see Figure \ref{fig:dent}).
The reason for this is that the resonance during the time interval $I_{k+1}$ may
cause the modes $k+1\pm 1= k+2, k$ to grow while $\theta_{k}$ remains relatively
unchanged. Thus the upper bounds for $k+2,k$ are much larger than for $k$,
resulting in the ``dent''.
On the time interval $I_k$ the mode $\theta_{k}$ then is resonant and may cause
the modes $\theta_{k+1}$ and $\theta_{k-1}$ to grow by a large factor (shown in
red in Figure \ref{fig:dent}), thus resulting in a new ``dent''.

\begin{figure}[htbp]
  \centering
  \includegraphics[width=0.5\linewidth]{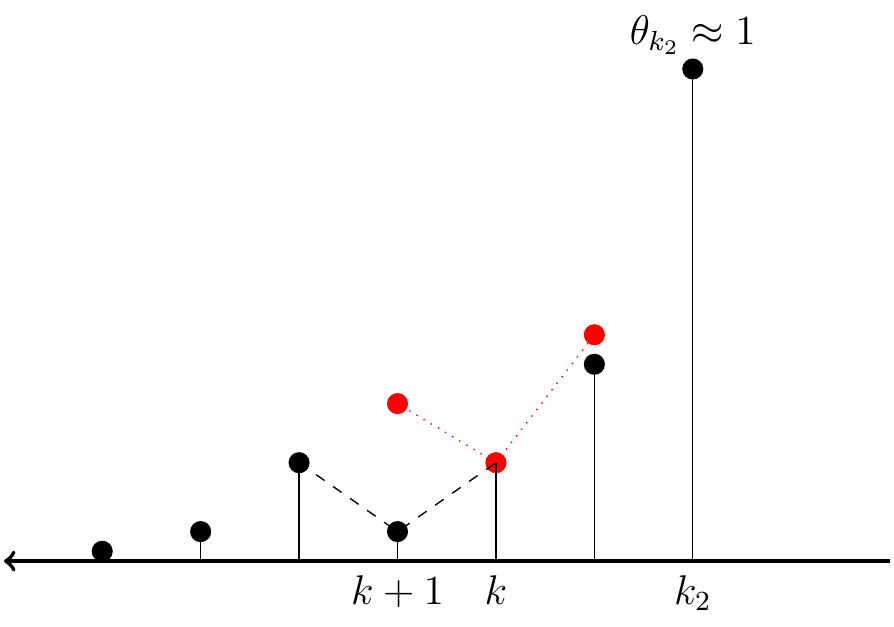}
  \caption{The resonance mechanism causes neighboring modes to grow.
    In this figure we show the upper bounds of Lemma \ref{lemma:inductionupper}
    at the time $t_{k}$ as black dots.
    In particular, we observe a ``dent'' at the frequency $k+1$.
    During the time interval $I_k$ then the mode $\theta_k$ becomes resonant
    leading to the change of upper bounds colored in red. The ``dent'' moves.}
  \label{fig:dent}
\end{figure}

We further remark that \eqref{eq:inductionstep2} holds for $k=k_1$ by Lemma
\ref{lemma:smalltime}.
Moreover, since $1\leq k \leq k_1$ it holds that
\begin{align*}
  \frac{\eta}{k^2}\geq \frac{\eta}{k_1^2}\geq \frac{1}{10} c^{-1} k_1 \geq 100 k_1.
\end{align*}
In particular, since there are less than $k_1$ many such $k$ we observe that
  \begin{align*}
    C_{k} \leq (1+\frac{1}{100} \frac{1}{k_1})^{k_1}\leq e^{\frac{1}{100}}
  \end{align*}
and thus $C_k$ remains uniformly bounded and comparable to $1$.
The estimates \eqref{eq:inductionstep2} hence formalize the statement that modes
other than $k-1,k+1$ only change slightly, while the modes $k-1, k+1$ might
potentially grow. However, since the upper bound on the $\theta_{k}$ was much
smaller than the one on the mode $\theta_{k-1}$ it suffices to that all modes
$l\geq k-1$ then at time $t_{k-1}$ satisfy the same upper bound as the mode $k-1$.

Additionally, using the integral equation we deduce that 
  \begin{align*}
    |\theta_{k_2}(t_{k_2})- \theta_{k_2}(t_{k_1})| \leq 4 e^{\frac{1}{100}}c (\frac{\eta}{k_2^2})^{-1},
  \end{align*}
where we used the bound on $\int c_{k_2}^{\pm}$ by the definition of $k_2$ and
the above induction bound.
Thus, the claimed control at time $t_{k_2}$ stated in Proposition
\ref{prop:timeregimes} indeed follows from this lemma.

\begin{proof}[Proof of Lemma \ref{lemma:inductionupper}]
  As remarked Lemma \ref{lemma:smalltime} ensures that \eqref{eq:inductionstep2}
  holds for $k=k_1$ and we then aim to proceed by induction, where we again use
  a bootstrap argument.
  
  More precisely, let $k$ be given and suppose that \eqref{eq:inductionstep2}
  holds for this $k$.
  Since $C_{k+1}> C_{k}$ and the products also only become larger when replacing
  $k$ by $k-1$, it follows that there exists a time interval $(t_{k},
  T_{\star})$ such that \eqref{eq:inductionstep2} with $k-1$ holds on this
  interval with slight modification: 
  \begin{align*}
        \begin{split}
    |\theta_{k_2}|&\leq C_{k-}, \\
    |\theta_{l}|&\leq C_{k-1} \prod_{j=l}^{k_2} (c (\frac{\eta}{j^2})^{-1}) \text{ for } j\leq k-1, \\
    |\theta_{l}|&\leq C_{k-1} \prod_{j=k-2}^{k_2} (c (\frac{\eta}{j^2})^{-1}) \text{ for } j\geq k-1, \\
      |G_{l}|&\leq 10 C_{k-1} c (\frac{\eta}{k_1^2})^{-2} \prod_{j=l}^{k_2} c(\frac{\eta}{j^2})^{-1}, \text{ if } l< k-1, l\neq k_2,\\
      |G_{l}|&\leq 10C_{k-1} \cdot 2^{-|l-k_1|}  c (\frac{\eta}{k_1^2})^{-2}\prod_{j=k-2}^{k_2} c(\frac{\eta}{j^2})^{-1}, \text{ if } l > k-1, l\neq k \\
      |G_{k}| & \leq 10 C_{k-1}\prod_{j=k-2}^{k_2} (c (\frac{\eta}{j^2})^{-1}),\\
      |G_{k_2}| & \leq C_{k-1} c^2 (\frac{\eta}{k_1^2})^{-2}.
   \end{split}
  \end{align*}
  for all $t \in (t_{k}, T_{\star})$.
  We emphasize here that $G_{k}$ is only required to satisfy an upper bound
  comparable to the one of $\theta_{k}$.
  The reason for this is that the contribution by 
  \begin{align*}
    \int_{{t_{k}}}^{T}\exp(-\nu\int_{t}^{T}k^2+(\eta-k\tau)^2 d\tau) 2 \frac{(\frac{\eta}{k}-t)}{(1+(\frac{\eta}{k}-t)^2)^2}
  \end{align*}
  is potentially large and thus for times $T\approx \frac{\eta}{k}$ we do not
  necessarily expect the unmodified estimate of \eqref{eq:inductionstep2} to
  hold.
  However, we claim that the above modified bootstrap assumptions allow us to
  recover the unmodified estimates at the final time.
  Indeed, we observe that for $t_{k-1}-3 \leq \tau \leq t_{k-1}$ it holds
  that
  \begin{align*}
    (\eta-k\tau)^2 \geq 0.5 (\frac{\eta}{k})^2 = 0.5 \eta \frac{\eta}{k^2}.
  \end{align*}
  Hence, using the fact that $0.5 \eta \nu\gg 1$, it follows that for $t_{k-1}-3
  \leq t \leq T \leq t_{k-1}$
  \begin{align*}
    \exp(-\nu\int_{t}^{T}k^2+(\eta-k\tau)^2 d\tau) \leq \exp(-\frac{\eta}{k^2} (T-t)).
  \end{align*}
  We may then use the integral equations to express
  \begin{align*}
    G_{k}(t_{k-1}) &= \exp(-\nu\int_{t_{k-1}-3}^{t_{k_1}}k^2+(\eta-k\tau)^2 d\tau) G_{k}(t_{k-1}-3) \\
    & \quad + \int_{t_{k-1}-3}^{t_{k-1}}  \exp(-\nu\int_{t}^{T}k^2+(\eta-k\tau)^2 d\tau) \text{rhs}(t)dt,
  \end{align*}
  where we abbreviated the terms of the evolution equation as a
  right-hand-side$(t)$.
  We then observe that
  \begin{align*}
    \exp(-\nu\int_{t_{k-1}-3}^{t_{k_1}}k^2+(\eta-k\tau)^2 d\tau) \leq \exp(-3\frac{\eta}{k^2})
  \end{align*}
  yields exponential decay and that all coefficients included in the
  right-hand-side are non-resonant on the interval $(t_{k-1}-3, t_{k-1})$. Thus
  $G_{k}(t_{k-1})$ satisfies the desired improved bounds.

  It remains to be shown that the maximal time $T_{\star}\leq t_{k-1}$ for which
  the (modified) bootstrap assumptions are satisfied is given by $T_{\star}=t_{k-1}$.
  Indeed, suppose not then
  \begin{align*}
    \theta_{k_2}(t) - \theta_{k_2}(t_k) &= \int_{t_k}^t c_{k_2}^{+} \theta_{k_2+1} + c_{k_2}^{-1} \theta_{k_2}^{-1} \\
                                        &\leq c (\frac{\eta}{k_1^2})^{-2} C_{k-1}  c (\frac{\eta}{k_1^2})^{-1}\\
                                        & \leq c^2 (\frac{\eta}{k_1^2})^{-2} C_{k-1}
  \end{align*}
  Thus, using the induction assumption
  \begin{align*}
    |\theta_{k_2}(t)| \leq C_{k} + c^2 (\frac{\eta}{k_1^2})^{-2} C_{k-1} < C_{k-1}.
  \end{align*}

  Similarly, if $l \not \in \{k-1, k+1\}$ the non-resonant integrals are bounded
  by $c (\frac{\eta}{l^2})^{-2}$, while the larger of the neighbors is larger by
  a factor at most $c^{-1} \frac{\eta}{l^2}$ compared to the bound on the mode
  $l$.
  Thus, in total we lose at most a factor $(1+2 (\frac{\eta}{k^2})^{-1})$, as claimed.
  The estimates for $G$ in the non-resonant cases are analogous.\\
  
  It remains to discuss the effect of resonant modes.
  Here we observe that
  \begin{align*}
    \int_{t_k}^{t_{k-1}} \frac{c\eta}{k_3} \frac{1}{(1+(\frac{\eta}{k}-t)^2)^2} dt \leq \frac{c\eta}{k^3} \pi
  \end{align*}
  might cause the modes $\theta_{k-1}$ and $\theta_{k+1}$ to grow by this factor times the bound
  on the mode $\theta_k$.
  However, the bound on $\theta_{k-1}$ formulated in the bootstrap is already larger than
  $c^{-1}(\frac{\eta}{k^2})$ times this bound. Hence, this growth is consistent
  with \eqref{eq:inductionstep2}. Similarly the growth of $\theta_{k+1}$ is
  controlled since the ``dent'' moved in the induction step and now
  $\theta_{k+1}$ needs only to satisfy the same upper bound as $\theta_{k-1}$.

  We also observe that $G_k$ may grow due to 
  \begin{align*}
    \int_{t_k}^T \exp(-\nu\int_{t}^Tk^2+(\eta-kt)^2) \frac{2(\frac{\eta}{k}-t)}{(1+(\frac{\eta}{k}-t)^2)^2} \theta_{k} dt.
  \end{align*}
  Since $k\geq k_2$ this contribution can be estimated from above by
  $\frac{1}{\nu k_2^3}\ll 1$ in the present setting. However, for later
  reference we remark that for our upper bounds it suffices to note that
  \begin{align*}
   \int_{t_k}^T |\frac{2(\frac{\eta}{k}-t)}{(1+(\frac{\eta}{k}-t)^2)^2}| \leq 2. 
  \end{align*}
\end{proof}

\subsection{Norm Inflation in the Resonant Regime}
\label{sec:norminf}

The core of our norm inflation mechanism is given by the resonant growth in the
time interval
\begin{align*}
  (t_{k_2}, t_0),
\end{align*}
which is formulate in the following proposition.

\begin{prop}
  \label{prop:inductionblow}
  Let $c,\eta$ be as in Theorem \ref{theorem:blow-up} and let $k_0\in N$ be such
  that (by rounding down)
  \begin{align*}
    k_0\approx \sqrt[3]{c\eta\pi}.
  \end{align*}
  Let further $1\leq k \leq \frac{k_0}{4}$ be such that $\nu k^2\geq 4$.
  Then if at time $t_k$ it holds that
  \begin{align}
    \label{eq:inductiongrowth}
    \begin{split}
    |\theta_k(t_k)|\geq 0.5 \max(\|\theta(t_k)\|_{\ell^\infty}, \|G(t_k)\|_{\ell^\infty}), \\
    |\theta_k(t_k)| \geq 4 |G_k(t_k)|,
  \end{split}
  \end{align}
  then at time $t_{k-1}$ it holds that
  \begin{align}
    \label{eq:lowerb}
    |\theta_{k\pm 1}(t_{k-1})| \geq \frac{c\eta \pi}{3k^3} |\theta_k(t_k)|
  \end{align}
  and \eqref{eq:inductiongrowth} holds with $k$ replaced by $k-1$.

  In particular it holds that
  \begin{align*}
    |\theta_{k_3}(t_{k_3})| \geq e^{k_0} |\theta_{k_2}(t_{k_2})|,\\
    |\theta_{k_3-2}(t_{k_3})| \geq e^{k_0} |\theta_{k_2}(t_{k_2})|,\\
  \end{align*}
\end{prop}

We remark that here Lemma \ref{lemma:inductionupper} ensures that
\eqref{eq:inductiongrowth} holds for $k=k_2$.
The condition $\nu k^2\geq 4$ here is the first and only time we explicitly use
the dissipation and is used to establish lower bounds (but not required for
upper bounds).
More precisely, we recall that  
\begin{align*}
  \nu \p_y \phi = \p_y \p_x \Delta_t^{-2} \theta + \p_y \p_x^{-1}\Delta_{t}^{-1} G. 
\end{align*}
Hence, in general control of the velocity requires control of both $\theta$ and
$G$. In particular, if $G$ is not smaller than $\theta$ but of comparable size
and with an opposite sign it could be that even at resonant times $\p_y \theta$
is very small even if $\theta$ is large.
Thus, in order to avoid such cancellations we use dissipation to ensure that $G$
is much smaller than $\theta$.
Here we stress that for $\nu\geq 4$ this condition is trivial, but for $\nu$
small restricts us to considering $k$ large.
We thus introduced a time threshold $t_{k_3}$ with $k_3= \frac{1}{1000}k_0$
until which this lower bound is satisfied.
As we show in Section \ref{sec:asym} after this time all upper bounds can be
established in the same way and lower bounds remain true at least for
frequencies larger than $k_3+1$.

Before beginning the proof, we briefly discuss the lower bound.
Iteratively applying \eqref{eq:lowerb} we observe that
\begin{align*}
  |\theta_{k_3 +1 \pm 1}(t_{k_3})| \geq |\theta_{k_2}(t_{k_2})| \prod_{j=k_2}^{k_3} \frac{c\eta \pi}{3j^3}.
\end{align*}
Hence, it suffices to bound the latter product from below.
  Indeed, we obtain a lower bound by 
  \begin{align*}
    \prod_{k=\frac{k_0}{1000}}^{k_0/2} \frac{c\eta \pi}{2k^3} = \frac{(c\eta\pi/2)^{k_0/2}}{(\frac{k_0}{2}!)^3}\frac{(\frac{k_0}{1000}!)^3} {(c\eta\pi/2)^{k_0/1000}}.
  \end{align*}
  We may then use Stirling's approximation to compute the first factor as
  approximately
  \begin{align*}
    2^{3 \frac{k_0}{2} } e^{3 \frac{k_0}{2}} (\sqrt{2\pi \frac{k_0}{2}})^{-3}
  \end{align*}
  and the second factor as
  \begin{align*}
    1000^{-3 \frac{k_0}{1000}} e^{-3 \frac{k_0}{1000}} (\sqrt{2\pi \frac{k_0}{1000}})^{3}.
  \end{align*}
  Using the fact that $n^{1/n}\rightarrow 1$ as $n\rightarrow \infty$ (and is
  about $1.007$ for $n=1000$) the first factors are easily dominated by the
  growth of $2^{3 \frac{k_0}{2}}$.
  Thus, we obtain a lower bound by
  \begin{align*}
    e^{\frac{3}{2}k_0 - \frac{3}{1000}k_0}\geq e^{k_0},
  \end{align*}
  as claimed.

\begin{proof}[Proof of Proposition \ref{prop:inductionblow}]
The heuristic idea of our proof is that at time $t_{k-1}$
\begin{align*}
  \theta_{k-1}(t_{k-1}) &\approx \theta_{k-1}(t_{k}) + \theta_k(t_k) \int_{t_{k}}^{t_{k-1}} \frac{c\eta}{k^3} \frac{1}{(1+(\frac{\eta}{k}-t)^2)^2} \\
  &\approx \theta_{k-1}(t_{k}) + \theta_k(t_k) \frac{c\eta \pi}{2k^3}. 
\end{align*}
Since by our choice of $k$ it holds that
\begin{align*}
  \frac{c\eta \pi}{2k^3} \geq 4^3=64
\end{align*}
this suggests that $\theta_k(t_k) \frac{c\eta \pi}{2k^3}$ dominates all other
contributions, which implies the lower bound \eqref{eq:lowerb}.
Moreover, while all other modes except $\theta_{k\pm 1}$ may also grow by some factor, this factor is
much smaller than $\frac{c\eta \pi}{2k^3}$. Hence, at time $t_{k-1}$ the mode
$\theta_{k-1}$ will be one of the largest modes and hence satisfy
\eqref{eq:inductiongrowth} with $k$ replaced by $k-1$.

It remains to make this heuristic rigorous, for which we employ a bootstrap
approach similar to the one of Section \ref{sec:G}. We remark that, since we do not require
$\ell^2$-based estimates, we here do not need to decompose into single mode
data.
For simplicity of notation we may without loss of generality assume that
\begin{align*}
  \theta_k(t_k).
\end{align*}

We then make the bootstrap assumptions that for $t_k\leq T\leq T_{\star}$ it
holds that
\begin{enumerate}[start=1, label={(E\arabic*)}]
\item \label{item:thetak}$|\theta_{k}(T)-1|\leq 0.1$. 
\item \label{item:thetal}$|\theta_{l}(T)|\leq 4$ for all $l \not \in \{k-1,k+1\}$.
\item \label{item:thetak1}$|\theta_{k\pm 1}(T)-\theta_{k\pm 1}(t_k) - \int_{t_k}^T \frac{c\eta}{k^3}
  \frac{1}{(1+(\frac{\eta}{k}-t)^2)^2} \theta_k(t) dt -\int_{t_k}^T \frac{c\eta}{k^3}
  \frac{1}{1+(\frac{\eta}{k}-t)^2} G_k(t) dt| \\ \leq 0.1$. 
\item \label{item:Gk}$|G_k(T)- \exp(-\nu \int_{t_k}^T k^2+(\eta-ks)^2ds)
  G_k(t_k)\\ -\int_{t_k}^T \exp(\nu \int_{t_k}^T k^2+(\eta-ks)^2ds)
  \frac{2(\frac{\eta}{k}-t)}{(1+(\frac{\eta}{k}-t)^2)^2}\theta_k(t) dt |\leq 0.1$. 
\item \label{item:Gl}$|G_l(T)| \leq 4$. 
\end{enumerate}
By local well-posedness these estimates are satisfied at least for a short time.
Similarly to the proof of Proposition \ref{prop:res} we in the following show
that these estimates self-improve and hence remain valid until time $t_{k-1}$.

\underline{Improving \ref{item:thetak}:}
We recall that by the integral equations \eqref{eq:integralequations}
\begin{align*}
  \theta_{k}(T)-1= \int_{t_k}^T c_{k}^{+}\theta_{k+1} +c_{k}^{-} \theta_{k-1} + d_{k}^{+}G_{k+1} + d_{k}^{-} G_{k-1}.
\end{align*}
Furthermore, by the estimates collected in Lemma \ref{lemma:coeffestimates} it
holds that
\begin{align}
  \label{eq:nonrestheta}
  \begin{split}
  \int_{t_k}^t c_{k}^{\pm} &\leq \frac{4c}{k} (\frac{\eta}{k^2})^{-2}, \\
  \int_{t_k}^t d_{k}^{\pm} &\leq \frac{4c}{k}.
\end{split}
\end{align}
Using the bootstrap estimates \ref{item:thetak1}--\ref{item:Gl} it then follows
that
\begin{align*}
  |\theta_{k}(T)-1| \leq 20 c < 0.1,
\end{align*}
and thus \ref{item:thetak} is improved.

\underline{Improving \ref{item:thetal}:}
We recall that by the integral equations \eqref{eq:integralequations} it holds
that
\begin{align*}
  \theta_l(T)- \theta_l(t_k) &= \int_{t_k}^T c_{l}^{+} \theta_{l+1} + c_{l}^{-} \theta_{l-1} + d_{l}^{+} G_{l+1} + d_{l}^{-} G_{l-1} dt.
\end{align*}
Since $l \not \in \{k-1,k+1\}$ all coefficient functions are non-resonant and
hence satisfy \eqref{eq:nonrestheta}.
Using the bootstrap assumptions \ref{item:thetak1}--\ref{item:Gl} it then
follows that the right-hand-side is much smaller than $2$, while
$|\theta_l(t_k)| \leq 2 |\theta_{k}(t_k)|=2$ by \eqref{eq:inductiongrowth}.
Hence, this bootstrap bound improves.

\underline{Improving \ref{item:thetak1}:}
By the integral equations it holds that \eqref{eq:integralequations}
\begin{align*}
  \theta_{k\pm 1}(T)-\theta_{k\pm 1}(t_k) = \int_{t_k}^T c_{k\pm 1}^{\mp} \theta_{k} +c_{k \pm 1}^{\pm} \theta_{k\pm 2} + d_{k \pm 1}^{\mp}G_{k} + d_{k\pm 2}^{\pm} G_{k\pm 2}.
\end{align*}
We thus only need to estimate
\begin{align*}
  \int_{t_k}^Tc_{k \pm 1}^{\pm} \theta_{k\pm 2}+ d_{k\pm 2}^{\pm} G_{k\pm 2}
\end{align*}
from above. Here we again use the estimates
\eqref{eq:nonrestheta} and control $\theta_{k\pm 2}, G_{k\pm 2}$ by the
bootstrap assumptions \ref{item:thetal}, \ref{item:Gl}.

We remark that 
\begin{align*}
  \int_{t_k}^T c_{k\pm 1}^{\mp} \leq \frac{c\eta\pi}{2k^3},\\
  \int_{t_k}^T d_{k\pm 1}^{\mp} \leq \frac{c\eta\pi}{k^3}
\end{align*}
and that $\theta_{k}(t), G_k(t)\leq 2$ by the bootstrap assumptions
\ref{item:thetak},\ref{item:Gk}.
Hence, this estimate also provides an upper bound on the size of $\theta_{k\pm 1}(T)$.

\underline{Improving \ref{item:Gl}:}
We recall that by the integral equations \eqref{eq:integralequations} it holds that
\begin{align*}
   G_{l}(T)- &\exp\left(-\nu\int_{t_k}^{T}l^2+(\eta-lt)^2 dt\right) G_{l}(t_k)\\
  &= \int_{t_k}^{T}\exp(-\nu\int_{t}^{T}l^2+(\eta-l\tau)^2 d\tau)2\frac{(\frac{\eta}{l}-t)}{(1+(\frac{\eta}{l}-t)^2)^2} \theta_{l} dt \\
  &\quad + \int_{t_k}^{T}\exp(-\nu\int_{t}^{T}l^2+(\eta-l\tau)^2 d\tau) f \frac{\nu}{g} il (d_l^{+} G_{l+1} + d_l^{-} G_{l-1}) dt \\
  &\quad + \int_{t_k}^{T}\exp(-\nu\int_{t}^{T}l^2+(\eta-l\tau)^2 d\tau) f \frac{\nu}{g} il (c_l^{+} \theta_{l+1} + c_l^{-} \theta_{l-1}) dt \\
  & \quad + \int_{t_k}^{T}\exp(-\nu\int_{t}^{T}l^2+(\eta-l\tau)^2 d\tau)\frac{1}{1+(\frac{\eta}{l}-t)^2}\left( c_{l}^{+} \theta_{l+1} + c_{l}^{-} \theta_{l-1}\right) dt \\
  & \quad + \int_{t_k}^{T}\exp(-\nu\int_{t}^{T}l^2+(\eta-l\tau)^2 d\tau) \frac{1}{1+(\frac{\eta}{l}-t)^2} \left(d_{l}^{+} G_{l+1} + d_{l}^{-} G_{l-1} \right) dt.
\end{align*}
Since $l\neq k$ the estimate \eqref{eq:Gthetanonres} of Lemma
\ref{lemma:coefficientsForG} holds:
\begin{align*}
  \int_{t_k}^{T}\exp(-\nu\int_{t}^{T}l^2+(\eta-l\tau)^2 d\tau) \left| 2 \frac{(\frac{\eta}{l}-t)}{(1+(\frac{\eta}{l}-t)^2)^2} \right| dt\leq 2 (\frac{\eta}{k^2})^{-2}.
\end{align*}
Similarly, by the estimates \eqref{eq:Gf}, \eqref{eq:Gthetac} and \eqref{eq:GGd}
all other integrals can be controlled by
\begin{align*}
  c (\frac{\eta}{k^2})^{-1}
\end{align*}
times the $L^\infty$ norms of $G_{l\pm 1}, \theta_{l\pm 1}$, which are
controlled by the bootstrap assumptions \ref{item:thetak}--\ref{item:Gl}.
Therefore, in conclusion
\begin{align*}
  |G_{l}(T)- \exp\left(-\nu\int_{t_k}^{T}l^2+(\eta-lt)^2 dt\right) G_{l}(t_k)| &\leq 0.1\\
  \Rightarrow |G_l(T)|\leq 2+ 0.1 &<4.
\end{align*}

\underline{Improving \ref{item:Gk}:}
Similarly to the improvement of \ref{item:Gl} we observe that
\begin{align*}
   G_{k}(T)&- \exp\left(-\nu\int_{t_k}^{T}k^2+(\eta-kt)^2 dt\right) G_{k}(t_k)\\
  &- \int_{t_k}^{T}\exp(-\nu\int_{t}^{T}k^2+(\eta-k\tau)^2 d\tau)2\frac{(\frac{\eta}{k}-t)}{(1+(\frac{\eta}{k}-t)^2)^2} \theta_{k} dt \\
  &= \int_{t_k}^{T}\exp(-\nu\int_{t}^{T}k^2+(\eta-k\tau)^2 d\tau) f \frac{\nu}{g} ik (d_k^{+} G_{k+1} + d_k^{-} G_{k-1}) dt \\
  &\quad + \int_{t_k}^{T}\exp(-\nu\int_{t}^{T}k^2+(\eta-k\tau)^2 d\tau) f \frac{\nu}{g} ik (c_k^{+} \theta_{k+1} + c_k^{-} \theta_{k-1}) dt \\
  & \quad + \int_{t_k}^{T}\exp(-\nu\int_{t}^{T}k^2+(\eta-k\tau)^2 d\tau)\frac{1}{1+(\frac{\eta}{k}-t)^2}\left( c_{k}^{+} \theta_{k+1} + c_{k}^{-} \theta_{k-1}\right) dt \\
  & \quad + \int_{t_k}^{T}\exp(-\nu\int_{t}^{T}k^2+(\eta-k\tau)^2 d\tau) \frac{1}{1+(\frac{\eta}{k}-t)^2} \left(d_{k}^{+} G_{k+1} + d_{k}^{-} G_{k-1} \right) dt.
\end{align*}
The contributions on the right-hand-side are then controlled by \eqref{eq:Gf},
\eqref{eq:Gthetac} and \eqref{eq:GGd} of Lemma \ref{lemma:coefficientsForG} and
the bootstrap assumptions and thus \ref{item:Gk} improves.

Thus, all bootstrap estimates improve and hence remain valid at least until time $t_{k-1}$.

\underline{Establishing the lower bound \eqref{eq:lowerb} and \eqref{eq:inductiongrowth}:}
As the last step of our proof we show that the bootstrap estimates
\ref{item:thetak}--\ref{item:Gl} at time $t_{k-1}$ imply the desired lower bound
and that the solution satisfies \eqref{eq:inductiongrowth} at that time.

Here we first observe that by \ref{item:thetak}, \ref{item:thetal},
\ref{item:Gk}, \ref{item:Gl} all modes except $\theta_{k\pm 1}$ are bounded
above by $4$.

We next study \ref{item:thetak1}:
\begin{align*}
  |\theta_{k\pm 1}(T)-\theta_{k\pm 1}(t_k) &- \int_{t_k}^T \frac{c\eta}{k^3}
  \frac{1}{(1+(\frac{\eta}{k}-t)^2)^2} \theta_k(t) dt \\
  &-\int_{t_k}^T \frac{c\eta}{k^3}
  \frac{1}{1+(\frac{\eta}{k}-t)^2} G_k(t) dt| \leq 0.1
\end{align*}
Here it holds that
\begin{align}
 \label{eq:mainfactor} 
  \int_{t_k}^{t_{k-1}} \frac{c\eta}{k^3} \frac{1}{(1+(\frac{\eta}{k}-t)^2)^2} \approx \frac{c\eta\pi}{2k^3}\geq 32,
\end{align}
within a factor $1.2$, since $\theta_k$ is controlled by \ref{item:thetak} and
$|t_k-\frac{\eta}{k}|, |t_{k-1}-\frac{\eta}{k}|$ are sufficiently large to
approximate by the integral over all of $\R$.

It thus only remains to show that the contribution by
\begin{align*}
  \int_{t_k}^{t_{k-1}} \frac{c\eta}{k^3}\frac{1}{1+(\frac{\eta}{k}-t)^2} G_k(t) dt
\end{align*}
does not cancel this growth. That is, this integral should not be close to
\begin{align*}
  \int_{t_k}^{t_{k-1}} \frac{c\eta}{k^3}\frac{1}{(1+(\frac{\eta}{k}-t)^2)^2} \theta_k(t) dt.
\end{align*}
We here recall that by \ref{item:Gk}
\begin{align*}
  |G_k(T)&- \exp(\nu \int_{t_k}^T k^2+(\eta-ks)^2ds)
  G_k(t_k)\\
  &- \int_{t_k}^T \exp(-\nu \int_{t_k}^T k^2+(\eta-ks)^2ds)
  \frac{2(\frac{\eta}{k}-t)}{(1+(\frac{\eta}{k}-t)^2)^2}\theta_k(t) dt |\leq 0.1.
\end{align*}
Hence $G_k(T)$ is determined by $G_{k}(t_k)$ and $\theta_k(t_k)$ up to a
negligible error (compared to \eqref{eq:mainfactor}).

Let us first discuss
\begin{align*}
  \int_{t_k}^{t_{k-1}} \frac{c\eta}{k^3}\frac{1}{1+(\frac{\eta}{k}-t)^2}\exp(\nu \int_{t_k}^T k^2+(\eta-ks)^2ds) G_k(t_k) dt.
\end{align*}
Here, by assumption \eqref{eq:inductiongrowth}, we know that
\begin{align*}
  \exp(-\nu \int_{t_k}^T k^2+(\eta-ks)^2ds) |G_k(t_k)| \leq |G_k(t_k)|\leq \frac{1}{4}.
\end{align*}
and hence
\begin{align*}
  \int_{t_k}^{t_{k-1}} \frac{c\eta}{k^3}\frac{1}{1+(\frac{\eta}{k}-t)^2}\exp(-\nu \int_{t_k}^t k^2+(\eta-ks)^2ds) G_k(t_k) dt \leq \frac{1}{2} \frac{c\eta\pi}{k^32}
\end{align*}
is smaller than \eqref{eq:mainfactor} regardless of $\nu$.
Moreover, if $k$ and $\nu$ are such that $\nu k^2$ is large, then we may instead control
\begin{align*}
  \int_{t_k}^{t_{k-1}} \frac{c\eta}{k^3}\frac{1}{1+(\frac{\eta}{k}-t)^2}\exp(-\nu \int_{t_k}^t k^2+(\eta-ks)^2ds) G_k(t_k) \\
  \leq \frac{c\eta}{k^3} \frac{1}{\nu k^2}  |G_k(t_k)|,
\end{align*}
which is smaller than \eqref{eq:mainfactor} provided $\frac{1}{\nu k^2}$ is much
smaller than $\pi$.

It hence remains to discuss
\begin{align*}
  \int_{t_k}^T \exp(-\nu \int_{t_k}^T k^2+(\eta-ks)^2ds)
  \frac{2(\frac{\eta}{k}-t)}{(1+(\frac{\eta}{k}-t)^2)^2}\theta_k(t) dt.
\end{align*}
Here we observe that
\begin{align*}
  \int_{t_k}^T \frac{2(\frac{\eta}{k}-t)}{(1+(\frac{\eta}{k}-t)^2)^2} dt = \frac{1}{1+(\frac{\eta}{k}-T)^2} - \frac{1}{1+(\frac{\eta}{k}-t_k)^2}
\end{align*}
and that
\begin{align*}
  G_k(t)= - \frac{1}{1+(\frac{\eta}{k}-t)^2}\theta_k(t)
\end{align*}
would imply that the contribution by $G_k(t)$ exactly cancels
\eqref{eq:mainfactor}.
Thus, we need to make use of the decay due to the dissipation to rule out such
cancellation.
Indeed, bounding
\begin{align*}
  \int_{t_k}^T k^2+(\eta-ks)^2ds \geq k^2(T-t_k),
\end{align*}
we may control
\begin{align*}
  \int_{t_k}^T \exp(-\nu \int_{t_k}^T k^2+(\eta-ks)^2ds)
  |\frac{2(\frac{\eta}{k}-t)}{(1+(\frac{\eta}{k}-t)^2)^2}| \leq \frac{2}{\nu k^2},
\end{align*}
which is much smaller than $1$ by assumption on $\nu$ and $k$.

Thus, in summary
\begin{align*}
  |\theta_{k\pm 1}(t_{k-1}) | \approx \frac{c\eta\pi}{2k^3} |\theta_{k}(t_k)|\geq 32
\end{align*}
as claimed in \eqref{eq:lowerb} and all other modes are bounded above by $4$ and
thus also \eqref{eq:inductiongrowth} holds.
\end{proof}

\subsection{Persistence of Lower Bounds}
\label{sec:asym}

In this subsection we consider the evolution on the time interval
\begin{align*}
  (t_{k_3}, \infty),
\end{align*}
where we recall that the evolution after time $t_0=2\eta$ is asymptotically stable by Theorem
\ref{theorem:easy}.
As remarked in Sections \ref{sec:upper} and \ref{sec:norminf} here we do not
anymore derive lower bounds on the norm inflation since a priori there could be
cancellation. Instead we derive upper bounds on the possible growth, which then
allow us to also prove lower bounds on $\theta_{k_3+2}(t)$ for all times $t\geq
t_{k_3}$.
Thus norm inflation persists for all times.

\begin{lemma}
  \label{lemma:persist}
  Suppose that at the time $t_{k_3}$ it holds that
  \begin{align*}
    |\theta_{k_3+2}| \geq 0.5 \|\theta\|_{\ell^\infty} + 10 \|G\|_{\ell^\infty}.   
  \end{align*}
  and for $k\leq k_3$ define
  \begin{align*}
    C_{k_3}&= 2 |\theta_{k_3+2}(t_{k_3})|, \\
    C_{k-1}&= (1+(\frac{\eta}{k^2})^{-1}) C_{k}.
  \end{align*}
  Then for all $k\leq k_3$ it holds that 
  \begin{align}
    \label{eq:induction3}
    \begin{split}
    |\theta_l(t_k)|&\leq C_k,  \text{ for } l\geq k_3+2, \\
    |\theta_l(t_k)|&\leq C_{k}\prod_{j=k}^{k_1} \frac{c\eta \pi}{j^3}, \text{ for } k_3+1> l>k+1,\\
    |\theta_l(t_k) & \leq C_{k}\prod_{j=k+1}^{k_1} \frac{c\eta \pi}{j^3}, \text{ for } l\leq k, \\
    |G_l(t_k)|&\leq C_{k} (\frac{\eta}{k^2})^{-2} \prod_{j=k}^{k_1} \frac{c\eta \pi}{j^3}, \text{ for } k_3+1> l>k+1,\\
    |G_l(t_k)| & \leq 2C_{k}\prod_{j=k+1}^{k_1} \frac{c\eta \pi}{j^3}, \text{ for } l\leq k.
    \end{split}
    \end{align}
    Moreover, the bounds for $t_0$ persist for all times $t>t_0$ up to a
    possible growth by a constant $(1+c)$ and for all $t\geq t_{k_{2}}$ it holds that
    \begin{align*}
      |\theta_{k_3+2}(t)| \geq e^{-2} |\theta_{k_3+2}(t_{k_3})|.
    \end{align*}
  \end{lemma}
  As in Section \ref{sec:upper} we here note that the upper bounds \eqref{eq:induction3} have a ``dent'',
  where $\theta_{k+2}$ and $\theta_{k}$ satisfy the same upper bounds and
  $\theta_{k-1}$ is potentially much smaller (see Figure \ref{fig:dent2}).
  Here we further observe that the resonance mechanism during a time interval
  $I_{k}$ may only cause the modes $k-1,k+1$ to exhibit large change, while
  modes larger than $k+1$ remain mostly unchanged.
  In particular, the mode $k_3+2$ only mildly changes after time $t_{k_3}$
  and hence the lower bound persists. 
  \begin{figure}[htbp]
    \centering
    \includegraphics[width=0.5\linewidth]{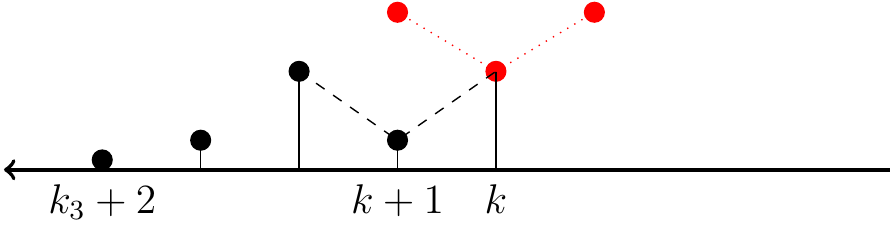}
    \caption{Growth bounds after time $t_{k_3}$.
      Following a similar strategy as in Section \ref{sec:upper} we obtain upper
      bounds for the resonances during the time intervals $I_k$. Here the mode
      $\theta_k$ highlighted in red may cause the modes $\theta_{k+1}$ and $\theta_{k-1}$ to grow,
      thus resulting in a new dent.
    }
    \label{fig:dent2}
  \end{figure}
  We further remark that the assumptions at the time $t_{k_3}$ are ensured by
  Proposition \ref{prop:inductionblow} of the preceding section.
  This lemma hence proves that the lower bound on $\theta_{k_3+2}$ persists
  until time $t_0$ up to a loss of constant. Thus norm inflation in
  $\ell^\infty$ (and hence also in $\ell^2$) has been achieved at that time.
  Moreover, this lower bound then persists also for all future times by the same
  argument as in Section \ref{sec:small}.

  \begin{proof}[Proof of Lemma \ref{lemma:persist}]
    After normalization we may assume that $C_{k_3}=1$ and deduce that deduce
    that $C_{k}\leq e^{\frac{1}{1000}}$ for all $k$.
    Hence, supposing that the claimed estimates hold we observe that for all
    $t\geq t_{k_2}$ it holds that 
  \begin{align*}
    |\theta_{k_3+1}|  &\leq e^{\frac{1}{1000}} \frac{c\eta\pi}{k_3^3}, \\
    |\theta_{k_3+3}|  &\leq e^{\frac{1}{1000}}.
  \end{align*}
  Since $k_3+2$ is non-resonant on the interval $(t_{k_3}, \infty)$ this then implies that
  \begin{align*}
    |\theta_{k_3+2}(t)- \theta_{k_3+2}(t_{k_3})| \ll 1
  \end{align*}
  and thus $\theta_{k_3+2}(t) \approx \theta_{k_3+2}(t_{k_3})$, as desired.
  
  The proof of the induction step for \eqref{eq:induction3} follows by the same
  argument as in the proofs of Proposition \ref{prop:inductionblow} and of Lemma
  \ref{lemma:inductionupper}. In particular, we again establish slightly rougher
  upper bounds for the mode $G_{k}$ on $(t_{k}, t_{k-1}-3)$ and then recover the
  desired bounds by using the fast exponential decay.
  We omit the details for brevity.
  We further remark that for $t\geq t_{0}$ all integrals are non-resonant (see
  also Section \ref{sec:small}). Thus we may use the same proof on the interval $(t_0,\infty)$.
  
  This concludes the proof of Proposition \ref{prop:timeregimes} and thus of
  Theorem \ref{theorem:blow-up}.
\end{proof}

In this article we have shown that, despite viscous dissipation, the
Boussinesq equations linearized around \emph{traveling waves} exhibit norm inflation and blow-up.
Yet, for certain critical data we can prescribe the blow-up in a fine enough way
that damping of the velocity field persists.
These results hence show that damping of the \emph{velocity} field is a more robust
effect than asymptotic stability of the \emph{vorticity} and \emph{temperature}.
Moreover, while classically considered a nonlinear effect, we show that
\emph{echoes} and their corresponding resonances are a feature of the linearized
problem around traveling waves.

\subsection*{Acknowledgments}
Funded by the Deutsche Forschungsgemeinschaft (DFG, German Research Foundation) – Project-ID 258734477 – SFB 1173.

\bibliography{citations2}

\newcommand{\etalchar}[1]{$^{#1}$}
\begin{thebibliography}{BBCZD21}

\bibitem[BBCZD21]{bedrossian21}
Jacob Bedrossian, Roberta Bianchini, Michele Coti~Zelati, and Michele Dolce.
\newblock Nonlinear inviscid damping and shear-buoyancy instability in the
  two-dimensional {B}oussinesq equations.
\newblock {\em arXiv preprint arXiv:2103.13713}, 2021.

\bibitem[BCZV19]{BCZVvortex2017}
Jacob Bedrossian, Michele Coti~Zelati, and Vlad Vicol.
\newblock {Vortex axisymmetrization, inviscid damping, and vorticity depletion
  in the linearized 2D {E}uler equations}.
\newblock {\em Annals of PDE}, 5(1):1--192, 2019.

\bibitem[Bed20]{bedrossian2016nonlinear}
Jacob Bedrossian.
\newblock Nonlinear echoes and {L}andau damping with insufficient regularity.
\newblock {\em Tunisian Journal of Mathematics}, 3:121--205, 2020.

\bibitem[BM14]{bedrossian2013asymptotic}
Jacob Bedrossian and Nader Masmoudi.
\newblock Asymptotic stability for the {C}ouette flow in the {2D} {E}uler
  equations.
\newblock {\em Applied Mathematics Research eXpress}, 2014(1):157--175, 2014.

\bibitem[BM15]{bedrossian2015inviscid}
Jacob Bedrossian and Nader Masmoudi.
\newblock Inviscid damping and the asymptotic stability of planar shear flows
  in the 2{D} {E}uler equations.
\newblock {\em Publ. Math. Inst. Hautes \'Etudes Sci.}, 122:195--300, 2015.

\bibitem[Cha06]{chae2006global}
Dongho Chae.
\newblock Global regularity for the 2d {B}oussinesq equations with partial
  viscosity terms.
\newblock {\em Advances in Mathematics}, 203(2):497--513, 2006.

\bibitem[CKN99]{chae1999local}
Dongho Chae, Sung-Ki Kim, and Hee-Seok Nam.
\newblock Local existence and blow-up criterion of {H}{\"o}lder continuous
  solutions of the {B}oussinesq equations.
\newblock {\em Nagoya Mathematical Journal}, 155:55--80, 1999.

\bibitem[CZZ19]{coti2019degenerate}
Michele Coti~Zelati and Christian Zillinger.
\newblock On degenerate circular and shear flows: the point vortex and power
  law circular flows.
\newblock {\em Communications in Partial Differential Equations},
  44(2):110--155, 2019.

\bibitem[DM18]{dengmasmoudi2018}
Yu~Deng and Nader Masmoudi.
\newblock Long time instability of the {C}ouette flow in low {G}evrey spaces.
\newblock {\em arXiv preprint arXiv:1803.01246}, 2018.

\bibitem[DWXZ20]{dong2020stability2}
Boqing Dong, Jiahong Wu, Xiaojing Xu, and Ning Zhu.
\newblock Stability and exponential decay for the 2d anisotropic {B}oussinesq
  equations with horizontal dissipation.
\newblock {\em arXiv preprint arXiv:2009.13445}, 2020.

\bibitem[DWZ20]{deng2020stability}
Wen Deng, Jiahong Wu, and Ping Zhang.
\newblock Stability of couette flow for 2d {B}oussinesq system with vertical
  dissipation.
\newblock {\em arXiv preprint arXiv:2004.09292}, 2020.

\bibitem[DWZZ18]{doering2018long}
Charles~R Doering, Jiahong Wu, Kun Zhao, and Xiaoming Zheng.
\newblock Long time behavior of the two-dimensional {B}oussinesq equations
  without buoyancy diffusion.
\newblock {\em Physica D: Nonlinear Phenomena}, 376:144--159, 2018.

\bibitem[DZ19]{dengZ2019}
Yu~Deng and Christian Zillinger.
\newblock Echo chains as a linear mechanism: Norm inflation, modified exponents
  and asymptotics.
\newblock {\em arXiv preprint arXiv:1910.12914}, 2019.

\bibitem[EW15]{elgindi2015sharp}
Tarek~M Elgindi and Klaus Widmayer.
\newblock Sharp decay estimates for an anisotropic linear semigroup and
  applications to the surface quasi-geostrophic and inviscid {B}oussinesq
  systems.
\newblock {\em SIAM Journal on Mathematical Analysis}, 47(6):4672--4684, 2015.

\bibitem[How61]{howard1961note}
Louis~N Howard.
\newblock Note on a paper of john w. miles.
\newblock {\em Journal of Fluid Mechanics}, 10(4):509--512, 1961.

\bibitem[IJ19]{ionescu2018inviscid}
Alexandru~D Ionescu and Hao Jia.
\newblock Inviscid damping near the couette flow in a channel.
\newblock {\em Communications in Mathematical Physics}, pages 1--82, 2019.

\bibitem[IJ20]{ionescu2020nonlinear}
Alexandru~D Ionescu and Hao Jia.
\newblock Nonlinear inviscid damping near monotonic shear flows.
\newblock {\em arXiv preprint arXiv:2001.03087}, 2020.

\bibitem[Jia20]{jia2019linear}
Hao Jia.
\newblock Linear inviscid damping in gevrey spaces.
\newblock {\em Archive for Rational Mechanics and Analysis}, 235(2):1327--1355,
  2020.

\bibitem[LWX{\etalchar{+}}21]{lai2021optimal}
Suhua Lai, Jiahong Wu, Xiaojing Xu, Jianwen Zhang, and Yueyuan Zhong.
\newblock Optimal decay estimates for 2d {B}oussinesq equations with partial
  dissipation.
\newblock {\em Journal of Nonlinear Science}, 31(1):1--33, 2021.

\bibitem[MSHZ20]{masmoudi2020stability}
Nader Masmoudi, Belkacem Said-Houari, and Weiren Zhao.
\newblock Stability of {C}ouette flow for 2d {B}oussinesq system without
  thermal diffusivity.
\newblock {\em arXiv preprint arXiv:2010.01612}, 2020.

\bibitem[MV11]{Villani_long}
Cl{\'e}ment Mouhot and C{\'e}dric Villani.
\newblock On {L}andau damping.
\newblock {\em Acta mathematica}, 207(1):29--201, 2011.

\bibitem[MWGO68]{malmberg1968plasma}
J.~H. Malmberg, C.~B. Wharton, R.~W. Gould, and T.~M. O'{N}eil.
\newblock Plasma wave echo experiment.
\newblock {\em Physical Review Letters}, 20(3):95--97, 1968.

\bibitem[Orr07]{orr1907stability}
William~M'F Orr.
\newblock The stability or instability of the steady motions of a perfect
  liquid and of a viscous liquid.
\newblock In {\em Proceedings of the Royal Irish Academy. Section A:
  Mathematical and Physical Sciences}, pages 69--138. JSTOR, 1907.

\bibitem[TWZZ20]{tao2020stability}
Lizheng Tao, Jiahong Wu, Kun Zhao, and Xiaoming Zheng.
\newblock Stability near hydrostatic equilibrium to the 2d {B}oussinesq
  equations without thermal diffusion.
\newblock {\em Archive for Rational Mechanics and Analysis}, 237(2):585--630,
  2020.

\bibitem[Wid18]{widmayer2018convergence}
Klaus Widmayer.
\newblock Convergence to stratified flow for an inviscid 3d {B}oussinesq
  system.
\newblock {\em Communications in Mathematical Sciences}, 16(6):1713--1728,
  2018.

\bibitem[WSP20]{wu2020stabilizing}
Jiahong Wu, Oussama~Ben Said, and Uddhaba~Raj Pandey.
\newblock The stabilizing effect of the temperature on buoyancy-driven fluids.
\newblock {\em arXiv preprint arXiv:2005.11661}, 2020.

\bibitem[WXZ19]{wu2019stability}
Jiahong Wu, Xiaojing Xu, and Ning Zhu.
\newblock Stability and decay rates for a variant of the 2d
  {B}oussinesq--{B}{\'e}nard system.
\newblock {\em Communications in Mathematical Sciences}, 17(8):2325--2352,
  2019.

\bibitem[WZZ18]{Zhang2015inviscid}
Dongyi Wei, Zhifei Zhang, and Weiren Zhao.
\newblock Linear inviscid damping for a class of monotone shear flow in
  {S}obolev spaces.
\newblock {\em Communications on Pure and Applied Mathematics}, 71(4):617--687,
  2018.

\bibitem[WZZ19]{WZZvorticitydepl}
Dongyi Wei, Zhifei Zhang, and Weiren Zhao.
\newblock {Linear inviscid damping and vorticity depletion for shear flows}.
\newblock {\em Annals of PDE}, 5(1):1--101, 2019.

\bibitem[YL18]{yang2018linear}
Jincheng Yang and Zhiwu Lin.
\newblock Linear inviscid damping for {C}ouette flow in stratified fluid.
\newblock {\em Journal of Mathematical Fluid Mechanics}, 20(2):445--472, 2018.

\bibitem[YOD05]{yu2005fluid}
J.~H. Yu, T.~M. O'{N}eil, and C.~F. Driscoll.
\newblock Fluid echoes in a pure electron plasma.
\newblock {\em Physical review letters}, 94(2):025005, 2005.

\bibitem[Zil16]{Zill5}
Christian Zillinger.
\newblock Linear inviscid damping for monotone shear flows in a finite periodic
  channel, boundary effects, blow-up and critical {S}obolev regularity.
\newblock {\em Arch. Ration. Mech. Anal.}, 221(3):1449--1509, 2016.

\bibitem[Zil17]{Zill6}
Christian Zillinger.
\newblock On circular flows: linear stability and damping.
\newblock {\em J. Differential Equations}, 263(11):7856--7899, 2017.

\bibitem[Zil19]{zillinger2019linear}
Christian Zillinger.
\newblock Linear inviscid damping in {S}obolev and {G}evrey spaces.
\newblock {\em arXiv preprint arXiv:1911.00880}, 2019.

\bibitem[Zil20]{zillinger2020boussinesq}
Christian Zillinger.
\newblock On the {B}oussinesq equations with non-monotone temperature profiles.
\newblock {\em arXiv preprint arXiv:2011.02316}, 2020.

\bibitem[Zil21a]{zillinger2020landau}
Christian Zillinger.
\newblock On echo chains in {L}andau damping: Traveling wave-like solutions and
  {G}evrey 3 as a linear stability threshold.
\newblock {\em Annals of PDE}, 7(1):1--29, 2021.

\bibitem[Zil21b]{zillinger2020enhanced}
Christian Zillinger.
\newblock On enhanced dissipation for the {B}oussinesq equations.
\newblock {\em Journal of Differential Equations}, 282:407--445, 2021.

\end{thebibliography}
\bibliographystyle{alpha}

\end{document}